\documentclass[12pt]{amsart}
\usepackage{bbding}
\usepackage[margin=2.0cm]{geometry}

\usepackage{xcolor}
\usepackage[normalem]{ulem}

\usepackage{hyperref}

\usepackage{amsfonts}
\usepackage{mathrsfs}
\usepackage[nobysame]{amsrefs}
\usepackage{cite}

\newcommand{\R}{{\mathbb R}}
\newcommand{\N}{{\mathbb N}}

\newcommand{\HH}{\mathcal{H}}
\newtheorem{theorem}{Theorem}[section]
\newtheorem{lemma}[theorem]{Lemma}
\newtheorem{prop}[theorem]{Proposition}
\newtheorem{coro}[theorem]{Corollary}

\newtheorem{example}[theorem]{Example}

\title{The $L_p$ dual Minkowski problem for $p>1$ and $q>0$}
\author{K\'aroly J. B\"or\"oczky}
\address{Alfr\'ed R\'enyi Institute of Mathematics, Hungarian Academy
  of Sciences, Re\'altanoda u. 13-15, H-1053 Budapest, Hungary, and
Department of Mathematics, Central European University, N\'ador u. 9, H-1051, Budapest, Hungary}
\email{boroczky.karoly.j@renyi.mta.hu}

\author{Ferenc Fodor}
\address{Department of Geometry, Bolyai Institute, University of Szeged,
	Aradi v\'ertan\'uk tere 1, 6720 Szeged, Hungary}
\email{fodorf@math.u-szeged.hu}


\begin{document}
\maketitle

\begin{abstract}
General $L_p$ dual curvature measures have recently been introduced by Lutwak, Yang and Zhang \cite{LYZ18}. These new measures unify several other geometric measures of the Brunn-Minkowski theory and the dual Brunn-Minkowski theory. $L_p$ dual curvature measures arise from $q$th dual inrinsic volumes by means of Alexandrov-type variational formulas. Lutwak, Yang and Zhang \cite{LYZ18} formulated the $L_p$ dual Minkowski problem, which concerns the characterization of $L_p$ dual curvature measures.   
In this paper, we solve the existence part of the $L_{p}$ dual Minkowski problem for $p>1$ and $q>0$, and we also discuss the regularity of the solution. 
\end{abstract}

\section{Introduction}
In this paper we solve the existence part of the Minkowski problem for  $L_p$ dual curvature measures with parameters $p>1$ and $q>0$. 
The $L_p$ dual curvature measures emerged recently \cite{LYZ18} as a family of geometric measures which unify several important families of measures in the Brunn-Minkowski theory and its dual theory of convex bodies.  

Our setting is Euclidean $n$-space $\R^n$ with $n\geq 2$. We write $o$ to denote the origin in $\R^n$, $\langle\cdot,\cdot\rangle$ for the standard inner product, and $\|\cdot\|$ for its induced norm. We denote the unit ball  by $B^n=\{x\in\R^n:\|x\|\leq 1\}$, the unit sphere by $S^{n-1}=\partial B^n$. $\HH^k(\cdot)$ stands for the $k$-dimensional Hausdorff measure, and for the $n$-dimensional volume (Lebesgue measure) we use the notation $V(\cdot)$. In particular, the volume of the unit ball is $\kappa_n=V(B^n)=\frac{\pi^{\frac n2}}{\Gamma(\frac n2 +1)}$ and its surface area is  $\omega_n=\HH^{n-1}(S^{n-1})=d\kappa_n$, where $\Gamma$ is Euler's gamma function (cf. Artin \cite{A1964}). We call a compact convex set $K\subset\R^n$ with non-empty interior a convex body.  We use the symbol $\mathcal{K}^n_o$ to denote the family of compact convex sets in $\R^n$ containing the origin, and $\mathcal{K}^n_{(o)}$ to denote the family of all convex bodies $K$ which contain $o$ in their interior, that is, $o\in{\rm int}\,K$. For detailed information on the theory of convex bodies we refer to the recent books by Gruber \cite{Gru07} and Schneider \cite{Sch14}. 

For a convex compact set $K\subset\R^n$, the support function $h_K(u):S^{n-1}\to \R$ is defined as $h_K(u)=\max \{\langle x,u\rangle : x\in K\}.$ For a $u\in S^{n-1}$, the face of $K$ with exterior unit normal $u$ is 
$F(K,u)=\{x\in K: \langle x,u\rangle=h_K(u) \}$.
For $x\in{\partial} K$, let the spherical image of $x$ be defined as
$\nu_K(\{x\})=\{u\in S^{n-1}: h_K(u)=\langle x,u\rangle\}$. For a Borel set $\eta\subset S^{n-1}$, the reverse spherical image is
$$
\nu_K^{-1}(\eta)=\{x\in{\partial} K:\, \nu_K(x)\cap \eta\neq \emptyset\}.
$$
If $K$ has a unique supporting hyperplane at $x$, then we say that $K$ is smooth at $x$, and in this case $\nu_K(\{x\})$ contains exactly one element that we denote by $\nu_K(x)$ and call it the exteriior unit normal of $K$ at $x$. 

The classical Minkowski problem in the Brunn-Minkowski theory of convex bodies is concerned with the characterization of the so-called surface area measure. The surface area measure of a convex body can be defined in a direct way as follows. Let $\partial'K$ denote the subset of the boundary  of $K$ where there is a unique outer unit normal vector. It is known that $\HH^{n-1}(\partial K\setminus\partial'K)=0$. Then $\nu_K:\partial'K\to S^{n-1}$ is a function that is usually called the spherical Gauss map. The surface area measure of $K$, denoted by $S(K,\cdot)$, is a Borel measure on $S^{n-1}$ such that for any Borel set $\eta\subset S^{n-1}$, it holds that $S(K,\eta)=\HH^{n-1}(\nu_K^{-1}(\eta))$. It is an important property of the surface area measure that it satisfies Minkowski's variational formula 
\begin{equation}
\label{Alexandrov}
\lim_{\varepsilon\to 0^+}\frac{V(K+\varepsilon L)-V(K)}{\varepsilon}=\int_{S^{n-1}}h_L\,d S(K, \cdot)
\end{equation}
for any convex body $L\subset\R^n$.

The classical Minkowski problems asks for necessary and sufficient conditions for a Borel measure on $S^{n-1}$ to be the surface area measure of a convex body. 
A particularly important case of the Minkowski problem is for discrete measures. Let $P\subset \R^n$ be a polytope, which is defined as the convex hull of a finite number of points in $\R^n$ provided ${\rm int}P\neq \emptyset$. Those faces whose dimension is $n-1$ are called facets. A polytope $P$ has a finite number of facets and the union of facets covers the boundary of $P$. 
Let $u_1,\ldots, u_k\in S^{n-1}$ be the exterior unit normal vectors of the facets of $P$. Then $S(P,\cdot)$ is a discrete measure on $S^{n-1}$ concentrated on the set $\{u_1,\ldots, u_k\}$, and
$S(P,\{u_i\})=\HH^{n-1}(F(P,u_i))$, $i=1,\ldots, k$. The Minkowski problem asks the following: let $\mu$ be a discrete positive Borel measure on $S^{n-1}$. Under what conditions does there exist a polytope $P$ such that $\mu=S(P,\cdot)$? Furthermore, if such a $P$ exists, is it unique? This polytopal version, along with the case when the surface area measure of $K$ is absoultely continuous with respect to the spherical Lebesgue measure, was solved by Minkowski \cites{M1897,M1903}. He also proved the uniqueness of the solution. For general measures the problem was solved by Alexandrov \cites{A1938, A1939} and independently by Fenchel and Jensen. The argument for existence uses the Alexandrov variational formula of the surface area measure, and the uniqueness employs the Minkowski inequality for mixed volumes. In summary, the necessary and sufficient conditions for the existence of the solution of the Minkowski problem for $\mu$ are that for any linear subspace $L\leq \R^n$ with $\dim L\leq n-1$, $\mu(L\cap S^{n-1})<\mu (S^{n-1})$, and that the centre of mass of $\mu$ is at the origin, that is,  $\int_{S^{n-1}}u\,\mu(du)=0$. 

Similar questions have been posed for $K\in\mathcal{K}^n_{o}$, and at least partially solved, for other measures associated with convex bodies in the Brunn-Minkowski theory, for example, the integral curvature measure $J(K,\cdot)$ of Alexandrov (see (\ref{Jdef}) below), or the $L_p$ surface area measure
$dS_p(K,\cdot)=h_K^{1-p}dS(K,\cdot)$ for $p\in \R$ introduced by Lutwak \cite{Lut93}, where 
$S_1(K,\cdot)=S(K,\cdot)$ ($p=1$) is the classical surface area measure, and $S_0(K,\cdot)$ ($p=0$) is the
 the cone volume measure (logarithmic Minkowski problem). Here some care is needed if $p>1$, when we only consider the case $o\in\partial K$ if the resulting $L_p$ surface area measure $S_p(K,\cdot)$ is finite.
For a detailed overview of these measures and their associated Minkowski problems and further references see, for example, Schneider \cite{Sch14}, and Huang, Lutwak, Yang and Zhang \cite{HLYZ16}.

Lutwak built the dual Brunn-Minkowski theory in the 1970s as a "dual" counterpart of the classical theory. Although there is no formal duality between the classical and dual theories, one can say roughly that in the dual theory the radial function plays a similar role as the support function in the classical theory. The dual Brunn-Minkowski theory concerns the class of compact star shaped sets of $\R^n$. A compact set $S\subset\R^n$ is star shaped with respect to a point $p\in S$ if for all $s\in S$, the segment $[p,s]$ is contained in $S$. We denote the class of compact sets in $\R^n$ that are star shaped with respect to $o$ by $\mathcal{S}^n_{0}$, and the set of those elements of $\mathcal{S}^n_{0}$ that contain $o$ in their interiors are denoted by $\mathcal{S}^n_{(o)}$. Clearly, $\mathcal{K}^n_{o}\subset \mathcal{S}^n_{o}$ and $\mathcal{K}^n_{(o)}\subset \mathcal{S}^n_{(o)}$. For a star shaped set $S\in\mathcal{S}_o^n$, we define the radial function of $S$ as $\varrho_S(u)=\max\{t\geq 0:\,tu\in S\}$ for $u\in S^{n-1}$.

Dual intrinsic volumes for convex bodies $K\in\mathcal{K}^n_{(o)}$ were defined by Lutwak \cite{L1975} whose definition works for all $q\in\R$. For $q>0$, we extend Lutwak's definition of the $q$th dual intrinsic volume $\widetilde{V}_q(\cdot)$ to a compact convex set $K\in \mathcal{K}^n_o$ as 
\begin{equation}\label{dualintrvol}
\widetilde{V}_q(K)=\frac1n\int_{S^{n-1}}\varrho^q_K(u)\,\mathcal{H}^{n-1}(du),
\end{equation}
which is normalized in such a way that $\widetilde{V}_n(K)=V(K)$.
We observe that $\widetilde{V}_q(K)=0$ if ${\rm dim}\,K\leq n-1$, and $\widetilde{V}_q(K)>0$ if $K$ is full dimensional. We note that dual intrinsic volumes for $q=0,\ldots, d$ are the coefficients of the dual Steiner polynomial for star shaped compact sets, where the radial sum replaces the Minkowski sum. The $q$th dual intrinsic volumes, which arise as coefficients naturally satisfy \eqref{dualintrvol}, and this provides the possibility to extend their definition for arbitrary $q\in\R$ in the case when $o\in\mathrm{int} K$ and for $q>0$ when $o\in K$.

For a Borel set $\eta\subset S^{n-1}$, let 
$$
\alpha^*_K(\eta)=\{u\in S^{n-1}: \varrho_K(u)u\in F(K,v)\text{ for some }v\in\eta\}=
\{u\in S^{n-1}: \varrho_K(u)u\in\nu_K^{-1}(\eta)\}.
$$
The set $\alpha_K^*(\eta)$ is called the reverse radial Gauss image of $\eta$, cf. 
Huang, Lutwak, Yang, Zhang \cite{HLYZ16} and Lutwak, Yang, Zhang \cite{LYZ18}. For a convex body $K\in\mathcal{K}_{(o)}^n$ and $q\in\R$, the $q$th dual curvature measure $\widetilde{C}_q(K,\cdot)$ is a Borel measure on $S^{n-1}$ defined in \cite{HLYZ16} as 
\begin{equation}\label{dualcurvmeasure} 
\widetilde{C}_q(K,\eta)=\frac1n\int_{\alpha^*_K(\eta)}\varrho_K^{q}(u)\HH^{n-1}(du).
\end{equation} 
Similar to the case of $q$th dual intrinsic volumes, the notion of $q$th dual curvature measures can be extended to compact convex sets $K\in\mathcal{K}_o^n$ when $q>0$ using \eqref{dualcurvmeasure}. Here if ${\rm dim}K\leq n-1$, then $\widetilde{C}_q(K,\cdot)$ is the trivial measure.
We note that
the so-called cone volume measure $V(K,\cdot)=\frac1n\,S_0(K,\cdot)=\frac1n\,h_KS(K,\cdot)$, and Alexandrov's integral curvature measure $J(K,\cdot)$ 
can both be represented as dual curvature measures as
\begin{eqnarray}
\label{conevol}
\mbox{$V(K,\cdot)=\frac1n$}\,S_0(K,\cdot)&=& \widetilde{C}_n(K,\cdot)\\
\label{Jdef}
J(K^*,\cdot)&=& \widetilde{C}_0(K,\cdot)\mbox{ \ provided $o\in{\rm int}K$}.
\end{eqnarray}
Based on Alexandrov's integral curvature measure, the $L_p$ Alexandrov integral curvature measure
$$
dJ_p(K,\cdot)=\varrho_K^p\,dJ(K,\cdot)
$$ 
was introduced by Huang, Lutwak, Yang, Zhang \cite{HLYZ18+} for $p\in\R$ and $K\in\mathcal{K}_{(o)}^n$.

We note that the $q$th dual curvature measure is a natural extension of the cone volume measure 
$V(K,\cdot)=\frac1n\,h_KS(K,\cdot)$ also in the variational sense, Corollary~4.8 of Huang, Lutwak, Yang, Zhang \cite{HLYZ16}
states the following generalization of Minkowski's formula \eqref{Alexandrov}. For arbitrary convex bodies $K,L\in \mathcal{K}^n_{(o)}$, we have
\begin{equation}
\label{dualAlexandrov}
\lim_{\varepsilon\to 0^+}\frac{\widetilde{V}_q(K+\varepsilon L)-\widetilde{V}_q(K)}{\varepsilon}=\int_{S^{n-1}}\frac{h_L}{h_K}\, d\widetilde{C}_q(K,\cdot).
\end{equation}
In this paper, we actually use not \eqref{dualAlexandrov}, but Lemma~\ref{varyz0}, which is
a variational formula in the sense of Alexandrov for dual curvature measures of polytopes.

For integers $q=0,\ldots, n$, dual curvature measures arise in a similar way as in the Brunn-Minkowski theory by means of localized dual Steiner polynomials. These measures satisfy \eqref{dualcurvmeasure}, and hence their definition can be extended for $q\in\R$. 
Huang, Lutwak, Yang and Zhang \cite{HLYZ16} proved that the $q$th dual curvature measure of a convex body $K\in\mathcal{K}_{(o)}^n$ can also be obtained from the $q$th dual intrinsic volume by means of an Alexandrov-type variational formula. 

Lutwak, Yang, Zhang \cite{LYZ18} introduced a more general version of the dual curvature measure
where a star shaped set $Q\in\mathcal{S}_{(o)}^n$ is also involved; namely, 
for a Borel set $\eta\subset S^{n-1}$, $q\in\R$ and $K\in \mathcal{K}^n_{(o)}$, we have
\begin{equation}
\label{dualcurvmeasureQ} 
\widetilde{C}_q(K,Q,\eta)=
\frac 1n\int_{\alpha^*_K(\eta)}\varrho_K^{q}(u)\varrho_Q^{n-q}(u)\HH^{n-1}(du)
\end{equation} 
and the associated $q$th dual intrinsic volume with parameter body $Q$ is
\begin{equation}
\label{dualintrinsicQ} 
\widetilde{V}_q(K,Q)=\widetilde{C}_q(K,Q,S^{n-1})=
\frac 1n\int_{S^{n-1}}\varrho_K^{q}(u)\varrho_Q^{n-q}(u)\HH^{n-1}(du).
\end{equation} 
According to  Lemma~5.1 in \cite{LYZ18},
 if $q\neq 0$ and the Borel function $g:\,S^{n-1}\to \R$ is bounded, then
\begin{equation}
\label{intgCqoin}
\int_{S^{n-1}}g(u)\,d\widetilde{C}_{q}(K,Q,u)
=\frac1n\int_{\partial' K} g(\nu_K(x))\langle \nu_K(x),x\rangle\|x\|_Q^{q-n}\,d\HH^{n-1}(x)
\end{equation}
where $\|x\|_Q=\min\{\lambda\geq 0:\,\lambda x\in Q\}$ is a continuous, even and $1$-homogeneous function 
satisfying $\|x\|_Q>0$ for $x\neq o$.
The advantage of introducing the star body $Q$ is apperant in the  equiaffine invariant formula
(see Theorem~6.8 in \cite{LYZ18}) stating that
if $\varphi\in{\rm SL}(n,\R)$, then
\begin{equation}
\label{Cqaffineinv0}
\int_{S^{n-1}}g(u)\,d\widetilde{C}_{q}(\varphi K,\varphi Q,u)=
\int_{S^{n-1}} g\left(\frac{\varphi^{-t} u}{\|\varphi^{-t} u\|}\right)d\widetilde{C}_{q}( K,Q,u).
\end{equation}

For $q>0$, we extend these notions and fundamental observations to any convex body containing the origin on its boundary. In particular, for $q>0$, $K\in \mathcal{K}^n_{o}$ and $Q\in\mathcal{S}_{(o)}^n$,
we can define the associated curvature measure by 
\eqref{dualcurvmeasureQ} and the associated dual intrinsic volume by \eqref{dualintrinsicQ},
where $\widetilde{C}_{q}(K,Q,\cdot)$ is a finite Borel measure on $S^{n-1}$, and
$\widetilde{V}_{q}(K,Q,\cdot)$ is finite. 
 In addition, for $q>0$, we extend 
\eqref{intgCqoin} in Lemma~\ref{intgCqoQ} and \eqref{Cqaffineinv0} 
in Lemma~\ref{CqslnQ} to any $K\in \mathcal{K}^n_{o}$.

$L_p$ dual curvature measures were also introduced by Lutwak, Yand and Zhang \cite{LYZ18}. They provide a common framework that unifies several other geomeric meassures of the ($L_p$) Brunn-Minkowski theory and the dual theory: $L_p$ surface area measures, $L_p$ integral curvature measures, and dual curvature measures, cf. \cite{LYZ18}. 
For $q\in \R$,  $Q\in\mathcal{S}_{(o)}^n$, $p\in \R$  and $K\in\mathcal{K}_{(o)}^n$, 
we define
the $L_p$ $q$th dual curvature measure $\widetilde{C}_{p,q}(K,Q,\cdot)$ of $K$ with respect to $Q$  by the formula
\begin{equation}\label{lpdualcurvmeasureQ}
d\widetilde{C}_{p,q}(K,Q,\cdot)=h_K^{-p}d\widetilde{C}_q(K,\cdot).
\end{equation}

While we also discuss the measures $\widetilde{C}_{p,q}(K,Q,\cdot)$ involving a
$Q\in\mathcal{S}_{(o)}^n$, we concentrate on $\widetilde{C}_{p,q}(K,\cdot)$
in this paper, which represents many  fundamental measures associated to a $K\in\mathcal{K}_{(o)}^n$. Basic examples are 
\begin{eqnarray*}
\widetilde{C}_{p,n}(K,\cdot)&=&S_p(K,\cdot)\\
\widetilde{C}_{0,q}(K,\cdot)&=&\widetilde{C}_q(K,\cdot)\\
\widetilde{C}_{p,0}(K,\cdot)&=&J_p(K^*,\cdot).
\end{eqnarray*}

Alexandrov-type variational formulas for the dual intrinsic volumes were proved by Lutwak, Yang and Zhang, cf. Theorem~6.5 in \cite{LYZ18}, which produce the $L_p$ dual curvature measures $\widetilde{C}_{p,q}(K,Q,\cdot)$.
In this paper we will use a simpler variational formula, cf. Lemma~\ref{varyz0} for the $q$th dual intrinsic volumes that we specialize for our particular setting. 

In Problem~8.1 in \cite{LYZ18} the authors introduced the $L_p$ dual Minkowski existence problem: Find necessary and sufficient conditions that for fixed $p,q\in\R$ and $Q\in\mathcal{S}^n_{(o)}$ and a given Borel measure $\mu$ on $S^{n-1}$ there exists a convex body $K\in\mathcal{K}_{(o)}^n$ such that $\mu=\widetilde{C}_{p,q}(K,Q,\cdot)$. As they note in \cite{LYZ18}, this version of the Minkowski problem includes earlier considered other variants ($L_p$ Minkowski problem, dual Minkowski problem, $L_p$ Aleksandrov problem) for special choices of the parameters.  
For $Q=B^n$ and an absolutely continuous measure $\mu$, the $L_p$ dual Minkowski problem constitutes
 solving the Monge-Amp\`{e}re equation
\begin{equation}
\label{Monge-Ampere}
\det(\nabla^2 h+h{\rm Id})=h^{p-1}(\|\nabla h\|^2+h^2)^{\frac{n-q}2}f
\end{equation}
for an $L_1$ non-negative Borel function $f$ with $\int_{S^{n-1}}fd\HH^{n-1}>0$ (see 
\eqref{Monge-Ampere0} in Section~\ref{secregularity}).
Here $\nabla h$ and  $\nabla^2 h$ are the gradient and the Hessian of $h$ with repect to a moving orthonormal frame on $S^{n-1}$. Actually, if $Q\in\mathcal{S}^n_{(o)}$, then the related Monge-Amp\`{e}re equation is
(see 
\eqref{Monge-AmpereQ0} in Section~\ref{secregularity})
\begin{equation}
\label{Monge-AmpereQ}
\det(\nabla^2 h(u)+h(u){\rm Id})=h^{p-1}(u)\|\nabla h(u)+h(u)\,u\|_Q^{n-q}f.
\end{equation}


The case of the  $L_p$ dual Minkowski problem for even measures has received much attention but is not discussed here, see B\"or\"oczky, Lutwak, Yang, Zhang \cite{BLYZ13} concerning the $L_p$ surface area $S_p(K,\cdot)$, B\"or\"oczky, Lutwak, Yang, Zhang, Zhao \cite{BLYZZ}, Jiang Wu \cite{JiW17} and Henk, Pollehn \cite{HeP17+} concerning 
the $q$th dual curvature measure
$\widetilde{C}_q(K,\cdot)$, and Huang, Zhao \cite{HuZ18+} concerning the $L_p$ dual curvature measure
for detailed discussion of history and recent results.

Let us indicate the known solutions of the $L_p$ dual Minkowski problem when only mild conditions are imposed on the given measure $\mu$ or on the function $f$ in \eqref{Monge-Ampere}. We do not state the exact conditions, rather aim at a general overview.
For any finite Borel measure $\mu$ on $S^{n-1}$ such that the measure of any open hemi-sphere is positive, we have that
\begin{itemize}
\item if $p>0$ and $p\neq 1,n$, then $\mu=S_p(K,\cdot)= n\cdot \widetilde{C}_{p,n}(K,\cdot)$ for some 
$K\in \mathcal{K}^n_o$,
where the case $p>1$ and $p\neq n$ is due to Chou, Wang \cite{ChW06} and
Hug, Lutwak, Yang, Zhang \cite{HLYZ05}, while the case  $0<p<1$ is due to Chen, Li, Zhu \cite{CLZ01};

\item if $p\geq 0$ and $q<0$, then $\mu=\widetilde{C}_{p,q}(K,\cdot)$ for some 
$K\in \mathcal{K}^n_o$ where the case $p=0$  ($\mu=\widetilde{C}_{q}(K,\cdot)$)
is due to Zhao \cite{Zhao17} (see also Li, Sheng, Wang \cite{LSW18+}), and the case $p>0$ is due to Huang, Zhao \cite{HuZ18+} and Gardner {\it et al} \cite{GHXYW18+}.

\end{itemize}

In addition, if $p>q$ and $f$ is $C^\alpha$ for $\alpha\in(0,1]$, then \eqref{Monge-Ampere} has a unique positive $C^{2,\alpha}$ solution
according to Huang, Zhao \cite{HuZ18+}.

Naturally, the $L_p$ dual Monge-Ampere equation \eqref{Monge-Ampere} has a solution in the above cases for any non-negative $L_1$ function $f$ whose integral on any 
open hemi-sphere is positive. In  addition, if $-n<p\leq 0$ and $f$ is any non-negative $L_{\frac{n}{n+p}}$ function 
on $S^{n-1}$ such that $\int_{S^{n-1}}f\,d\mu>0$, then \eqref{Monge-Ampere} has a solution,  where the case
$p=0$ is due to Chen, Li, Zhu \cite{CLZ0}, and the case $p\in(-n,0)$ is due to
Bianchi, B\"or\"oczky, Colesanti,  Yang\cite{BBCY18+}.

We also note that if $p\leq 0$ and $\mu$ is discrete such that any $n$ elements of ${\rm supp}\mu$ are independent vectors, then 
$\mu=S_p(K,\cdot)= n\cdot \widetilde{C}_{p,n}(P,\cdot)$ for some 
polytope $P\in \mathcal{K}^n_{(o)}$ according to  Zhu \cite{Zhu15,Zhu17}.

In this paper, we first solve the discrete $L_p$ dual Minkowski problem if $p>1$ and $q>0$. 

\begin{theorem}
\label{polytope}
Let $Q\in\mathcal{S}_{(o)}^n$, $p>1$ and $q>0$ with $p\neq q$, and let $\mu$ be a discrete measure on $S^{n-1}$ that is not concentrated on any closed hemisphere. Then there exists a polytope $P\in\mathcal{K}^n_{(o)}$ such that $\widetilde{C}_{p,q}(P,Q,\cdot)=\mu$.
\end{theorem}
\noindent{\bf Remark } If $p>q$, then the solution is unique according Theorem~8.3 in 
Lutwak, Yang and Zhang \cite{LYZ18}.\\

We note that, in fact, we prove the existence of a polytope $P_0\in\mathcal{K}^n_{(o)}$ satisfying
$$
\widetilde{V}_{q}(P_0,Q)^{-1}\widetilde{C}_{p,q}(P_0,Q,\cdot)=\mu,
$$
which $P_0$ exists even if $p=q$ (see Theorem~\ref{polytopecor}).

Let us turn to a general, possibly non-discrete Borel measure $\mu$ on $S^{n-1}$.
As the example at the end of the paper by Hug, Lutwak, Yang, Zhang\cite{HLYZ05} shows, even if $\mu$ has a positive continuous density function with respect to the Hausdorff measure on $S^{n-1}$, for $q=n$ and $1<p<n$,
 it may happen that the only solution $K$ has the origin on its boundary.
In this case, $h_K$ has some zero on $S^{n-1}$ even if it occurs with negative exponent in
$\widetilde{C}_{p,q}(K,\cdot)$. Therefore if $Q\in\mathcal{S}_{(o)}^n$, $p>1$ and $q>0$, the natural form the $L_p$ dual Minkowski problem is the following (see Chou, Wang \cite{ChW06}  and Hug, Lutwak, Yang, Zhang \cite{HLYZ05} if $q=n$).
For a given non-trivial finite Borel measure $\mu$, find a convex body $K\in\mathcal{K}^n_{o}$ such that 
\begin{equation}
\label{Monge-Ampere-mu}
d\widetilde{C}_{q}(K,Q,\cdot)=h_K^{p}d\mu.
\end{equation}
It is natural to assume that $\HH^{n-1}(\Xi_K)=0$ in \eqref{Monge-Ampere-mu} for
\begin{equation}
\label{XiKdef}
\Xi_K=\{x\in\partial K:\,\mbox{there exists exterior normal $u\in S^{n-1}$ at $x$ with $h_K(u)=0$}\},
\end{equation}
which property ensures that the surface area measure $S(K,\cdot)$ is absolute continuous with respect to 
$\widetilde{C}_{q}(K,Q,\cdot)$
(see Corollary~\ref{SKCqQ}).
Actually, if $q=n$ and $Q=B^n$, then $d\widetilde{C}_{n}(K,\cdot)=\frac1n\,h_K\,dS(K,\cdot)$, and \cite{ChW06} and \cite{HLYZ05} 
consider the problem
\begin{equation}
\label{Monge-Ampere-muSK}
dS(K,\cdot)=nh_K^{p-1}d\mu,
\end{equation}
where the results of \cite{HLYZ05}  about \eqref{Monge-Ampere-muSK} yield the uniqueness of the solution
in \eqref{Monge-Ampere-muSK} for $q=n$, $p>1$ and $Q=B^n$ only  under the condition
$\HH^{n-1}(\Xi_K)=0$ (see Section~\ref{secLpdual} for more detailed discussion).

\begin{theorem}
\label{main}
Let $Q\in\mathcal{S}_{(o)}^n$, $p>1$ and $q>0$ with $p\neq q$, and let $\mu$ be a finite Borel measure on $S^{n-1}$ that is not concentrated on any closed hemisphere. Then there exists a $K\in\mathcal{K}^n_{o}$ 
with $\HH^{n-1}(\Xi_K)=0$ and ${\rm int}K\neq\emptyset$ such that 
$d\widetilde{C}_{q}(K,Q,\cdot)=h_K^{p}d\mu$, where $K\in\mathcal{K}^n_{(o)}$ provided $p>q$.
\end{theorem}

The solution in Theorem~\ref{main} is known  to be unique in some cases:
\begin{itemize}
\item if $p>q$ and $\mu$ is discrete ($K$ is a polytope) according to  Lutwak, Yang and Zhang \cite{LYZ18},
\item if $p>q$, $Q$ is a ball and $\mu$ has a $C^\alpha$ density function $f$ for $\alpha\in(0,1]$
according to
Huang, Zhao \cite{HuZ18+},
\item if $p>1$, $Q$ is a ball and $q=n$ according to Hug, Lutwak, Yang, Zhang\cite{HLYZ05}.
\end{itemize}

For Theorem~\ref{main}, in fact, we prove the existence of a convex body $K_0\in\mathcal{K}^n_{(o)}$ such that
$$
\widetilde{V}_{q}(K_0,Q)^{-1}\widetilde{C}_{p,q}(K_0,Q,\cdot)=\mu,
$$
which $K_0$ exists even if $p=q$ (see Theorem~\ref{withvolume}).

Concerning regularity, we prove the following statements based on Caffarelli \cite{Caf90a,Caf90b} (see Section~\ref{secregularity}). We note that if ${\partial}Q$ is $C^2_+$ for $Q\in\mathcal{S}_{(o)}^n$, then
 $Q$ is convex.

\begin{theorem}
\label{regularityaway0}
Let $p>1$, $q>0$, $Q\in\mathcal{S}_{(o)}^n$, $0<c_1<c_2$ and let $K\in\mathcal{K}^n_{o}$
with  $\HH^{n-1}(\Xi_K)=0$ and ${\rm int}K\neq\emptyset$ be such that 
$$
d\widetilde{C}_{q}(K,Q,\cdot)=h_K^{p}f\,d\HH^{n-1}
$$
for some Borel function $f$ on $S^{n-1}$ satisfying $c_1\leq f\leq c_2$.
\begin{description}
\item[(i)] ${\partial}K\backslash \Xi_K=\{z\in {\partial}K:\,h_K(u)>0\mbox{ for all }u\in N(K,z)\}$ and
${\rm bd }K\backslash \Xi_K$
 is $C^1$ and contains no segment, moreover $h_K$ is $C^1$ on $\R^n\backslash N(K,o)$.
 \item[(ii)] If $f$ is continuous, then 
each $u\in S^{n-1}\backslash N(K,o)$ has a neighborhood $U$ on $S^{n-1}$ such that the restriction of $h_K$ 
to $U$ is $C^{1,\alpha}$ for any $\alpha\in(0,1)$.
\item[(iii)] If $f$ is in $C^{\alpha}(S^{n-1})$ for some $\alpha\in(0,1)$,
 and ${\partial}Q$ is $C^2_+$, then 
${\partial}K\backslash \Xi_K$ is $C^2_+$, and
each $u\in S^{n-1}\backslash N(K,o)$ has a neighborhood where the restriction of $h_K$ 
  is $C^{2,\alpha}$.
\end{description}
\end{theorem}

We note that in Theorem~\ref{regularityaway0} (ii), the same neighborhood $U$ of $u$ works for every 
$\alpha\in(0,1)$. In addition, Theorem~\ref{regularityaway0} (i) yields that
for any convex
$W\subset \R^n\backslash N(K,o)$, $h_K(u+v)<h_K(u)+h_K(v)$ for independent $u,v\in W$.
For the case $o\in{\rm int}\, K$ in Theorem~\ref{regularityaway0}, see the
more appealing statements in Theorem~\ref{regularityaway0allpq}.

We recall that according to Theorem~\ref{main}, if $p>q>0$ and $p>1$, then $K\in\mathcal{K}^n_{(o)}$
holds for the solution $K$ of the $L_p$ dual Minkowski problem.
On the other hand, Example~\ref{1<p<q} shows that if $1<p<q$, then the solution $K$ of the $L_p$ dual Minkowski problem provided by Theorem~\ref{main} may satisfy that $o\in{\partial}K$ and $o$ is not a smooth point.
Next we show that still $K$ is strictly convex in this case, at least if $q\leq n$.

\begin{theorem}
\label{regularitystrictconvexity}
If $1<p<q\leq n$,  $Q\in\mathcal{S}_{(o)}^n$, $0<c_1<c_2$ and $K\in\mathcal{K}^n_{o}$ 
with  $\HH^{n-1}(\Xi_K)=0$ and ${\rm int}K\neq\emptyset$ be such that 
$$
d\widetilde{C}_{q}(K,Q,\cdot)=h_K^{p}f\,d\HH^{n-1}
$$
for some Borel function $f$ on $S^{n-1}$ satisfying $c_1\leq f\leq c_2$, then $K$ is strictly convex; or equivalently, $h_K$ is $C^1$ on $\R^n\backslash o$.
\end{theorem}

If $q=n$, then Theorems~\ref{regularityaway0} and \ref{regularitystrictconvexity} are due to Chou, Wang \cite{ChW06}. We do not know whether 
Theorem~\ref{regularitystrictconvexity} holds if $q>n$ (see the comments at the end of Section~7).

We note that even if $Q=B^n$ in Theorem~\ref{regularitystrictconvexity},  the equiaffine invariant formula \eqref{Cqaffineinv0} for $K\in\mathcal{K}^n_{o}$ simplifies the proof of 
Theorems~\ref{regularitystrictconvexity}; namely,
we use dual curvature measures with a parameter body different from Euclidean balls
 in the argument for Lemma~\ref{strictconvexityslnRinvariant}. The reason is that
if $o\in{\partial}K$ and $N(K,o)$ contains a pair of vectors with obtuse angle, then we need to transform
$K$ via a linear transform $\varphi\in{\rm SL}(n,\R)$ in such a way that
any two vectors in $N(\varphi K,o)$ make an acute angle.

We note that if $o\in{\rm int} K$, then the ideas leading to Theorem~\ref{regularityaway0} work for any $p,q\in \R$.

\begin{theorem}
\label{regularityaway0allpq}
Let $p,q\in\R$, $Q\in\mathcal{S}_{(o)}^n$, $0<c_1<c_2$ and let $K\in\mathcal{K}^n_{(o)}$
 be such that 
$$
d\widetilde{C}_{p,q}(K,Q,\cdot)=f\,d\HH^{n-1}
$$
for some Borel function $f$ on $S^{n-1}$ satisfying $c_1\leq f\leq c_2$. We have that
\begin{description}
\item[(i)] $K$ is smooth and strictly convex, and $h_K$ is $C^1$  on $\R^n\backslash \{o\}$;
 \item[(ii)] if $f$ is continuous, then the restriction of $h_K$ 
to $S^{n-1}$ is 
in $C^{1,\alpha}$ for any $\alpha\in(0,1)$;
\item[(iii)] if $f\in C^{\alpha}(S^{n-1})$ for $\alpha\in(0,1)$,
 and ${\partial}Q$ is $C^2_+$, then 
${\partial}K$ is $C^2_+$, and
 $h_K$  is $C^{2,\alpha}$ on $S^{n-1}$.
\end{description}
\end{theorem}

The rest of the paper is organized as follows. 
Up to Section~5, we assume $Q=B^n$ in order to simplify formulae.
We discuss properties of the dual curvarture measure in Section~2, 
and prove Theorem~\ref{polytope} in Section~3.
Fundamental properties of $L_p$ the dual curvarture measures are considered in Section~4, and we use all these results to prove Theorem~\ref{main} in Section~5.
Section~6 presents the way how to extend the arguments to the case when $Q$ is any star body.
The regularity of the solution is discussed in Section~7.

\section{On the dual curvature measure}
\label{secdualcurvature}

The goal of this section is for $q>0$, to extend the results of Huang, Lutwak, Yang and Zhang \cite{HLYZ16} about the dual curvature measure 
$\widetilde{C}_{q}(K,\cdot)$ when $K\in \mathcal{K}^n_{(o)}$ to the case when $K\in \mathcal{K}^n_o$. 
For any measure, we take the measure of the empty set to be zero.

For any compact convex set $K$ in $\R^n$ and $z\in{\partial}K$, we write $N(K,z)$ to denote the normal cone at $z$; namely,
$$
N(K,z)=\{y\in\R^n:\,\langle y,x-z\rangle\leq 0\mbox{ \ for $x\in K$}\}.
$$
If $z\in {\rm int}K$, then simply $N(K,z)=\{o\}$.
For compact, convex sets $K,L\subset\R^n$, we define their Hausdorff distance as 
$$
\delta_H(K,L):=\sup_{u\in S^{n-1}} |h_K(u)-h_L(u)|. 
$$
It is a metric  on the space of compact convex sets, and the induced metric space is locally compact according to  the Blaschke selection theorem.  
 For basic properties of Hausdorff distance we refer to Schneider \cite{Sch14}, and also to Gruber \cite{Gru07}.

First we extend Lemma~3.3 in \cite{HLYZ16}. 
Let $K\in \mathcal{K}^n_o$  with ${\rm int}K\neq \emptyset$. We recall that the so called singular points  
$z\in {\partial}K$ where ${\rm dim}N(K,z)\geq 2$ form a Borel set of zero $\HH^{n-1}$ measure, and hence its complement $\partial' K$ of smooth points is also a Borel set. For $z\in \partial' K$, we write $\nu_K(z)$ to denote the unique exterior normal at $z$. As a slight abuse of notation, for a Borel set $\eta\subset S^{n-1}$, we write
$\nu_K^{-1}(\eta)$ to denote its total inverse Gauss image; namely, the set of all $z\in{\partial}K$ with $N(K,z)\cap \eta\neq\emptyset$. In particular, if $o\in {\partial}K$, then we have
\begin{equation}
\label{invGauss0}
\Xi_K=\nu_K^{-1}\left(N(K,o)\cap S^{n-1}\right).
\end{equation}
If $o\in{\rm int}K$, then $\Xi_K=\emptyset$.
We also observe that the dual of $N(K,o)$ is
$$
N(K,o)^*=\{y\in\R^n:\,\langle y,x\rangle\leq 0\mbox{ for }x\in N(K,o)\}=
{\rm cl}\{\lambda x:\, \lambda\geq 0\mbox{ and }x\in K\},
$$
and hence
$$
\Xi_K=K\cap {\partial}N(K,o)^*.
$$
If $o\in{\rm int}K$, then simply $N(K,o)^*=\R^n$. 
We observe that if $o\in {\partial}K$, then 
\begin{eqnarray}
\label{alpha*Nko1}
\alpha_K^*\left(S^{n-1}\cap N(K,o)\right)&=&S^{n-1}\backslash\left({\rm int}N(K,o)^*\right)\\
\label{alpha*Nko2}
\alpha_K^*\left(S^{n-1}\backslash N(K,o)\right)&=&S^{n-1}\cap\left({\rm int}N(K,o)^*\right)
\end{eqnarray}
Since $\varrho_K(u)=0$ for $u\in S^{n-1}\backslash N(K,o)^*$, and 
$\HH^{n-1}(S^{n-1}\cap{\partial}N(K,o)^*)=0$, we deduce from \eqref{alpha*Nko1} that
if $q>0$, then
\begin{equation}
\label{CqNKo}
\widetilde{C}_q\left(K,N(K,o)\cap S^{n-1}\right)=
\widetilde{C}_{q}\left(K,\{u\in S^{n-1}:\,h_K(u)=0\}\right)=0.
\end{equation}

We note that the radial map $\widetilde{\pi}:\R^n\backslash \{o\}\to S^{n-1}$, $\tilde{\pi}(x)=x/\|x\|$ is locally Lipschitz. 
We write $\tilde{\pi}_K$ to denote the restriction of $\tilde{\pi}$ onto the $(n-1)$-dimensional Lipschitz manifold
$({\partial}K)\backslash \Xi_K=({\partial}K)\cap {\rm int}N(K,o)^*$. For any
$z\in(\partial'K )\backslash \Xi_K$, the Jacobian of $\tilde{\pi}_K$ at $z$ is 
\begin{equation}
\label{piKGaussian}
\langle \nu_K(z),\tilde{\pi}_K(z)\rangle \|z\|^{-(n-1)}=\langle \nu_K(z),z\rangle \|z\|^{-n}.
\end{equation}

For $u\in S^{n-1}$, we write $r_K(u)=\varrho_K(u)u\in {\partial}K$. Since $\tilde{\pi}_K$ is locally Lipschitz,
$\HH^{n-1}$ almost all points of $S^{n-1}\cap ({\rm int}N(K,o)^*)$ are in the image of 
$(\partial' K)\cap ({\rm int}N(K,o)^*)$ by $\tilde{\pi}_K$. Therefore for $\HH^{n-1}$ almost all points  
$u\in S^{n-1}\cap ({\rm int}N(K,o)^*)$, there is a unique exterior unit normal $\alpha_K(u)$
at $r_K(u)\in {\partial}K$. For the other points $u\in S^{n-1}\cap ({\rm int}N(K,o)^*)$, we just choose an exterior unit normal $\alpha_K(u)$
at $r_K(u)\in {\partial}K$.

The extensions of Lemma~3.3 and Lemma~3.4 in \cite{HLYZ16}  to the case when the origin may lie on the boundary of convex bodies are the following.

\begin{lemma}
\label{intgCqo}
If $q>0$, $K\in \mathcal{K}^n_o$ with ${\rm int}K\neq \emptyset$, and the Borel function $g:\,S^{n-1}\to \R$ is bounded, then
\begin{eqnarray}
\label{intgCqo1}
\int_{S^{n-1}}g(u)\,d\widetilde{C}_{q}(K,u)&=&\frac1n
\int_{S^{n-1}\cap ({\rm int}N(K,o)^*)}g(\alpha_K(u))\varrho_K(u)^q\,d\HH^{n-1}(u)\\ 
\label{intgCqo2}
&=&\frac1n\int_{\partial' K\backslash \Xi_K} g(\nu_K(x))\langle \nu_K(x),x\rangle\|x\|^{q-n}\,d\HH^{n-1}(x),\\
\label{intgCqo3}
&=&\frac1n\int_{\partial' K} g(\nu_K(x))\langle \nu_K(x),x\rangle\|x\|^{q-n}\,d\HH^{n-1}(x)
\end{eqnarray}
\end{lemma}
\begin{proof} 
To prove \eqref{intgCqo1},  $g$ can be approximated by finite  linear combination of
characteristic functions of Borel sets of $S^{n-1}$, and hence we may assume that
$g=\mathbf{1}_\eta$ for a Borel set $\eta\subset S^{n-1}$. In this case,
$$
\int_{S^{n-1}\cap N(K,o)}\mathbf{1}_\eta\,d\widetilde{C}_{q}(K,\cdot)=0
$$
by \eqref{CqNKo}, and
$$
\int_{S^{n-1}\backslash N(K,o)}\mathbf{1}_\eta\,d\widetilde{C}_{q}(K,\cdot)=
\widetilde{C}_{q}(K,\eta\backslash N(K,o))=
\int_{S^{n-1}\cap ({\rm int}N(K,o)^*)}\mathbf{1}_\eta(\alpha_K(u))\varrho_K(u)^q\,d\HH^{n-1}(u)
$$
by \eqref{alpha*Nko2} and the definition of $\widetilde{C}_{q}(K,\cdot)$, verifying \eqref{intgCqo1}.

In turn, \eqref{intgCqo1} yields \eqref{intgCqo2} by \eqref{piKGaussian}. 
For \eqref{intgCqo3}, we observe that if $x\in\Xi_K\cap \partial' K$, then $\langle \nu_K(x),x\rangle=0$.
\end{proof}

Now we prove that the $q$th dual curvature measure is continuous on $\mathcal{K}^n_o$ for $q>0$.

\begin{lemma}
\label{Vqcont}
For $q>0$, $\widetilde{V}_q(K)$ is a continuous function of $K\in \mathcal{K}^n_o$ with respect to the Hausdorff distance.
\end{lemma}
\begin{proof}
Let $R>0$ be such that $K\subset {\rm int }R\,B^n$.
 Let $K_m\in \mathcal{K}^n_o$ be a sequence of compact convex sets tending to $K$ with respect to Hausdorff distance. In particular, we may assume that $K_m\subset RB^n$ for all $K_m$. 

If ${\rm dim}\,K\leq n-1$, then 
there exists $v\in S^{n-1}$ such that $K\subset v^\bot$, where $v^\bot$ denotes the orthogonal (linear) complement of $v$. For $t\in [0,1)$, we write
$$
\Psi(v,t)=\{x\in \R^n:\,|\langle v,x\rangle|\leq t\}
$$
to denote the closed region of width $2t$ between two hyperplanes orthogonal to $v$ and symmetric to $0$. 

There exists a $t_0\in (0,1)$ such that for any
$t\in (0,t_0)$ and $v\in S^{n-1}$ it holds that 
$\HH^{n-1}(S^{n-1}\cap \Psi(v,t))<3t(n-1)\kappa_{n-1}$.

Let $\varepsilon\in (0,t_0)$.
We claim that there exists an $m_\varepsilon$ such that for all $m> m_\varepsilon$ and for any $u\in S^{n-1}\backslash \Psi(v,\varepsilon)$, we have
\begin{equation}
\label{qintrinsiceps}
\varrho_{K_m}(u)\leq \varepsilon.
\end{equation}
Since $K_m\to K$ in the Hausdorff metric, there exists an index $m_\varepsilon$ such that for all $m>m_\varepsilon$ it holds that
$K_m\subset K+\varepsilon^2 B^n\subset  \Psi(v,\varepsilon^2)$. Then if $u\in S^{n-1}\backslash \Psi(v,\varepsilon)$, then
$$
\varepsilon^2\geq \langle v,\varrho_{K_m}(u)u\rangle=\varrho_{K_m}(u)\langle v,u\rangle\geq
\varrho_{K_m}(u)\cdot \varepsilon,
$$
yielding (\ref{qintrinsiceps}). We deduce from (\ref{qintrinsiceps}) and $K_m\subset RB^n$ that
for any $\varepsilon\in(0,t_0)$, if $m>m_\varepsilon$, then
\begin{align*}
\widetilde{V}_q(K_m)&\leq \int_{S^{n-1}\backslash \Psi(v,\varepsilon)}\varepsilon^q\,d\mathcal{H}^{n-1}(u)
+\int_{S^{n-1}\cap \Psi(v,\varepsilon)}R^q\,d\mathcal{H}^{n-1}(u)\\
&\leq n\kappa_n \varepsilon^q+3\varepsilon (n-1)\kappa_{n-1}R^q,
\end{align*}
therefore $\lim_{m\to\infty}\widetilde{V}_q(K_m)=0=\widetilde{V}_q(K)$.

Next, let ${\rm int}\, K\neq \emptyset$ such that $o\in\partial K$. 
Since the functions $\varrho_{K_m}(u)$, $m=1,\ldots$ are uniformly bounded, by Lebesgue's dominated convergence theorem it is sufficient to prove that
\begin{equation}
\label{rhokmulimit}
\lim_{m\to \infty}\varrho_{K_m}(u)=\varrho_{K}(u)\mbox{ \ for $u\in S^{n-1}\backslash\partial N(K,o)^*$},
\end{equation}
as $\HH^{n-1}(S^{n-1}\cap \partial N(K,o)^*)=0$.
Now, let $\varepsilon\in[0,1)$.
\smallskip

{\em Case 1. } Let $u\in S^{n-1}\cap{\rm int}\,N(K,o)^*$.
\smallskip

Then $\varrho_K(u)>0$, and
$(1-\varepsilon)\varrho_K(u)\in{\rm int}\,K$ and $(1+\varepsilon)\varrho_K(u)\not\in K$. 
Thus, there exists an index $m(u,\varepsilon)>0$ such that for all $m>m(u,\varepsilon)$ it holds that
$(1-\varepsilon)\varrho_K(u)\in K_m$ and $(1+\varepsilon)\varrho_K(u)\not\in K_m$, or in other words,
$$
(1-\varepsilon)\varrho_K(u)\leq \varrho_{K_m}(u)\leq (1+\varepsilon)\varrho_K(u),
$$
which yields (\ref{rhokmulimit}) in this case.
\smallskip

{\em Case 2. }  Let $u\in S^{n-1}\backslash N(K,o)^*$.
\smallskip 

Then $\varrho_K(u)=0$, and there exists $v\in S^{n-1}\cap{\rm int}\,N(K,o)$ such that $\langle u,v\rangle>0$.
As $K_m\to K$, there exists an index $m(u,v,\varepsilon)>0$ such that for all $m> m(u,v,\varepsilon)$ it holds that $K_m\subset K+\varepsilon \langle u, v\rangle B^n$, and thus $h_{K_m}(v)<\varepsilon\langle u,v\rangle$. Therefore, for all $m> m(u,v,\varepsilon)$,  
$$
\varepsilon\langle u,v\rangle>h_{K_m}(v)\geq \langle \varrho_{K_m}(u)u,v\rangle=
\varrho_{K_m}(u)\langle u,v\rangle.
$$
This yields that $\varrho_{K_m}(u)< \varepsilon$ for all $m> m(u,v,\varepsilon)$, and thus
 (\ref{rhokmulimit}) holds by $\varrho_{K}(u)=0$.

Finally, let ${\rm int}\, K\neq \emptyset$ and 
 $o\in {\rm int}\, K$. The argument for this case is
analogous to the one used above in Case 1.
\end{proof}

The following Proposition~\ref{Cqcont} extends  Lemma~3.6 from
Huang, Lutwak, Yang, Zhang \cite{HLYZ16} about $K\in \mathcal{K}^n_{(o)}$
to the case when $K\in \mathcal{K}^n_{o}$.

\begin{prop}
\label{Cqcont}
If $q>0$, and $\{K_m\}$, $m\in\N$, tends to $K$ for $K_m,K\in \mathcal{K}^n_o$, then 
$\widetilde{C}_q(K_m,\cdot)$ tends weakly to $\widetilde{C}_q(K,\cdot)$.
\end{prop}
\begin{proof} 
Since any element of $\mathcal{K}^n_o$ can be approximated by elements of $\mathcal{K}^n_{(o)}$, we may assume that each $K_m\in \mathcal{K}^n_{(o)}$.
We fix $R>0$ such that $K\subset {\rm int}R\,B^n$, and hence we may also assume that $K_m\subset RB^n$ for all $K_m$. We need to prove that if $g:\,S^{n-1}\to\R$ is continuous, then
\begin{equation}
\label{gintKmlimit}
\lim_{m\to\infty}\int_{S^{n-1}}g(u)\,d\widetilde{C}_q(K_m, u)=\int_{S^{n-1}}g(u)\,d\widetilde{C}_q(K, u)
\end{equation}

First we assume that $o\in {\partial}\,K$. If ${\rm dim}\,K\leq n-1$, then $\widetilde{C}_q(K,\cdot)$ is the constant zero measure.
Since $\widetilde{C}_q(K_m, S^{n-1})=\widetilde{V}_q(K_m)$ tends to zero 
according to Lemma~\ref{Vqcont}, we conclude \eqref{gintKmlimit} in this case.

Therefore we may assume that ${\rm int} K\neq\emptyset$ and $o\in{\partial}\,K$. To simplify notation, we set
$$
\sigma=N(K,o)^*.
$$
According to Lemma~\ref{intgCqo}, \eqref{gintKmlimit} is equivalent to
\begin{equation}
\label{gintKmlimit0}
\lim_{m\to\infty}\int_{S^{n-1}}g(\alpha_{K_m}(u))\varrho_{K_m}(u)^q\,d\HH^{n-1}(u)=
\int_{S^{n-1}\cap ({\rm int}\,\sigma)}g(\alpha_K(u))\varrho_K(u)^q\,d\HH^{n-1}(u).
\end{equation}
Since $\tilde{\pi}_K$ is Lipschitz and $\HH^{n-1}(S^{n-1}\cap ({\partial}\sigma))=0$, 
to verify \eqref{gintKmlimit0}, and in turn \eqref{gintKmlimit},
it is sufficient to prove
\begin{eqnarray}
\label{limitinNK0}
\lim_{m\to\infty}\int_{\tilde{\pi}_K(({\rm int}\,\sigma)\cap \partial' K)}g(\alpha_{K_m}(u))\varrho_{K_m}(u)^q\,d\HH^{n-1}(u)&=&
\int_{\tilde{\pi}_K(({\rm int}\,\sigma)\cap \partial' K)}g(\alpha_K(u))\varrho_K(u)^q\,d\HH^{n-1}(u)\\
\label{limitoutNK0}
\lim_{m\to\infty}\int_{S^{n-1}\backslash \sigma}g(\alpha_{K_m}(u))\varrho_{K_m}(u)^q\,d\HH^{n-1}(u)
&=&0.
\end{eqnarray}
Now we prove \eqref{limitinNK0} and \eqref{limitoutNK0}, it follows from 
$K_m\subset RB^n$, the continuity of $g$ and Lemma~\ref{Vqcont} that there exists $M>0$ such that
\begin{equation}
\label{Kmbounded}
\begin{array}{rcll}
|\varrho_{K_m}(u)|&\leq &R &\mbox{ \ for $u\in S^{n-1}$},\\
|g(u)|&\leq &M &\mbox{ \ for $u\in S^{n-1}$},\\
\widetilde{C}_q(K_m, S^{n-1})&\leq & M &\mbox{ \ for $m\in \N$}
\end{array}
\end{equation}
We deduce from \eqref{Kmbounded} that Lebesgue's Dominated Convergence Theorem applies both
in \eqref{limitinNK0} and \eqref{limitoutNK0}. For \eqref{limitinNK0}, let 
$u\in \tilde{\pi}_K(({\rm int}\,\sigma)\cap \partial' K)$. Readily, 
$\lim_{m\to\infty}\varrho_{K_m}(u)^q=\varrho_{K}(u)^q$. Since $\alpha_K(u)$ is the unique normal at 
$\varrho_{K}(u)u\in{\partial}K$, we have $\lim_{m\to\infty}\alpha_{K_m}(u)=\alpha_{K}(u)$, and hence
$\lim_{m\to\infty}g(\alpha_{K_m}(u))=g(\alpha_{K}(u))$ by the continuity of $g$. In turn, we conclude 
\eqref{limitinNK0} by Lebesgue's Dominated Convergence Theorem.

Turning to \eqref{limitoutNK0}, it follows from Lebesgue's Dominated Convergence Theorem, $q>0$ and 
\eqref{Kmbounded} that it is sufficient to prove   
that if $\varepsilon>0$ and $u\in S^{n-1}\backslash\sigma$, then
\begin{equation}
\label{limitoutNKepsilon}
\varrho_{K_m}(u)\leq \varepsilon
\end{equation}
for $m\geq m_0$ where $m_0$ depends on $u,\{K_m\},\varepsilon$. 
Since $u\not\in \sigma=N(K,o)^*$,  there exists $v\in N(K,o)$ such that $\langle v,u\rangle=\delta>0$.
As $h_{K}(v)=0$ and $K_m$ tends to $K$, there exists $m_0$ such that
$h_{K_m}(v)\leq \delta\varepsilon$ if $m\geq m_0$. In particular, if $m\geq m_0$, then
$$
\varepsilon\delta\geq h_{K_m}(v)\geq \langle v, \varrho_K(u)u\rangle=\varrho_K(u)\delta,
$$
yielding \eqref{limitoutNKepsilon}, and in turn \eqref{limitoutNK0}.

Finally,  the argument leading to  \eqref{limitinNK0}
implies \eqref{gintKmlimit} also in the case when $o\in{\rm int} K$, 
completing the proof of 
Proposition~\ref{Cqcont}.
\end{proof}

\section{Proof of Theorem~\ref{polytope} for $Q=B^n$}
\label{secpolytopeBn}

To verify Theorem~\ref{polytope}, we prove the following statement, which also holds if $p=q$.

\begin{theorem}
\label{polytopecor}
Let $p>1$ and $q>0$, and let $\mu$ be a discrete measure on $S^{n-1}$ that is not concentrated on any closed hemisphere. Then there exists a polytope $P\in\mathcal{K}^n_{(o)}$ such that 
$\widetilde{V}_{q}(P)^{-1}\widetilde{C}_{p,q}(P,\cdot)=\mu$.
\end{theorem}

We recall that $\tilde{\pi}:\,\R^n\backslash\{o\}\to S^{n-1}$ is the radial projection, and for a convex body $K$ in $\R^n$ and $u\in S^{n-1}$, the face of $K$ with exterior unit normal $u$ is the set
$$
F(K,u)=\{x\in K:\,\langle x,u\rangle=h_K(u)\}.
$$
 We observe that if $P\in\mathcal{K}^n_{o}$ is a polytope with ${\rm int}\,P\neq \emptyset$, and 
$v_1,\ldots,v_l\in S^{n-1}$ are the exterior normals of the facets of $P$ not containing the origin, 
then 
\begin{equation}
\label{CqP}
\begin{array}{rcll}
{\rm supp} \,\widetilde{C}_q(P,\cdot)&=&\{v_1,\ldots,v_l\}, \text{ {and}}&\\[1ex]
\widetilde{C}_q(P,\{v_i\})&=&\displaystyle \frac1n\int_{\tilde{\pi}(F(P,v_i))}
\varrho_{P}^q(u)\, d\HH^{n-1}(u)&
\mbox{ \ for $i=1,\ldots,l$}.
\end{array}
\end{equation}

Let $p>1$, $q>0$ and $\mu$ be a discrete measure on $S^{n-1}$ that is not concentrated on any closed hemi-sphere.
Let ${\rm supp}\,\mu=\{u_1,\ldots,u_k\}$, and let $\mu(\{u_i\})=\alpha_i>0$, $i=1,\ldots,k$. For any 
$z=(t_1,\ldots,t_k)\in (\R_{\geq 0})^k$, we define
\begin{eqnarray}
\nonumber
\Phi(z)&=&\sum_{i=1}^k\alpha_it_i^p,\\
\label{P(z)}
P(z)&=&\{x\in\R^n:\,\langle x,u_i\rangle\leq t_i\,\forall i=1,\ldots,k\},\\
\nonumber
\Psi(z)&=&\widetilde V_q(P(z)).
\end{eqnarray}
Since $\alpha_i>0$ for $i=1,\ldots,k$, the set $Z=\{z\in (\R_{\geq 0})^k:\,\Phi(z)=1\}$ is compact, and hence Lemma~\ref{Vqcont} yields the existence of 
$z_0\in Z$ such that
$$
\Psi(z_0)=\max\{\Psi(z):\,z\in Z\}.
$$
We prove that $o\in {\rm int}\, P(z_0)$ and there exists $\lambda_0>0$ such that 
$$
\widetilde{V}_{q}(\lambda_0P(z_0))^{-1}\widetilde{C}_{p,q}(\lambda_0P(z_0),\cdot)=\mu.
$$

\begin{lemma}
\label{interior}
If $p>1$ and  $q>0$, then
$o\in {\rm int}\, P(z_0)$.
\end{lemma}
\begin{proof}
It is clear from the construction that $o\in P(z_0)$. 
We assume that $o\in\partial P(z_0)$, and seek a contradiction. Without loss of generality, we may assume that
$z_0=(t_1,\ldots,t_k)\in (\R_{\geq 0})^k$, where there exists $1\leq m<k$ such that
$t_1=\ldots=t_m=0$ and $t_{m+1}, \ldots, t_k>0$. For sufficiently     small $t>0$, we define
\begin{eqnarray*}
\tilde{z}_t&=&\left(\overbrace{0,\ldots,0}^m,(t_{m+1}^p-\alpha t^p)^{\frac1p},\ldots,
(t_k^p-\alpha t^p)^{\frac1p}\right)
\mbox{ \ \ for $\alpha=\frac{\alpha_1+\ldots+\alpha_m}{\alpha_{m+1}+\ldots+\alpha_k}$}, \text{ and}\\
z_t&=&\left(\overbrace{t,\ldots,t}^m,
(t_{m+1}^p-\alpha t^p)^{\frac1p},\ldots,(t_k^p-\alpha t^p)^{\frac1p}\right).
\end{eqnarray*}
Simple substitution shows that $\Phi(z_t)=1$, so $z_t\in Z$. 

We prove that there exist $\tilde{t}_0,\tilde{c}_1,\tilde{c}_2>0$ depending on $p$, $q$, $\mu$ and $z_0$ such that if $t\in(0,\tilde{t}_0]$, then
\begin{eqnarray}
\label{Pztc1}
\Psi(\tilde{z}_t)&\geq&\Psi(z_0)-\tilde{c}_1t^p,\\
\label{Pztc2}
\Psi(z_t)&\geq&\Psi(\tilde{z}_t)+\tilde{c}_2t,
\end{eqnarray}
therefore
\begin{equation}
\label{Pztc1c2}
\Psi(z_t)\geq\Psi(z_0)-\tilde{c}_1t^p+\tilde{c}_2t.
\end{equation}
We choose $R>0$ such that $P(z_0)\subset {\rm int}RB^n$ and $R\geq \max\{t_{m+1},\ldots,t_k\}$.

We start with proving (\ref{Pztc1}), and set
 $\varrho_0=\min\{t_{m+1},\ldots,t_k\}$. We frequently use {the following form of Bernoulli's inequality that says} that if $\tau\in (0,1)$ and $\eta>0$, then
\begin{equation}
\label{1minustau}
(1-\tau)^\eta\geq 1-\max\{1,\eta\}\cdot \tau.
\end{equation} 
It follows from \eqref{1minustau} and $\varrho_0\leq t_i\leq R$, $i=m+1,\ldots,k$,
 that there exist $s_0,c_0>0$, depending on $z_0$, $\mu$ and $p$ such that
if $t\in(0,s_0)$, then
\begin{equation}
\label{rho0half}
(t_i^p-\alpha t^p)^{\frac1p}>t_i-c_0t^p>\varrho_0/2\mbox{ \ for $i=m+1,\ldots,k$}.
\end{equation} 
Consider the cone ${N(P(z_0),o)^*}=\{x\in\R^n:\,\langle x,u_i\rangle\leq 0\;\forall i=1,\ldots,m\}$. Then $\varrho_{P(z_0)}(u)>0$
 for $u\in S^{n-1}$ if and only if $u\in {N(P(z_0),o)^*}$. Let $u\in {N(P(z_0),o)^*}\cap S^{n-1}$. Assume that $t\in (0,s_0)$ is so small that not only the automatic relations $P(\tilde{z}_t)\subset P(z_0)$ and $o\in\partial P(\tilde{z}_t)$ hold but also the facial structures of $P(z_0)$ and $P(\tilde{z}_t)$ are isomorphic. This guarantees that $\varrho_{P(\tilde{z}_t)}(u)>0$
 for $u\in S^{n-1}$ if and only if $u\in {N(P(z_0),o)^*}$. Then $\varrho_{P(\tilde{z}_t)}(u)\cdot u$ 
lies in a facet $F(P(\tilde{z}_t),u_i)$ for an $i\in\{m+1,\ldots,k\}$, thus
$$
\langle \varrho_{P(\tilde{z}_t)}(u)u,u_i\rangle=(t_i^p-\alpha t^p)^{\frac1p}>t_i-c_0t^p>\varrho_0/2.
$$
Combining the last estimate with $\varrho_{P(\tilde{z}_t)}(u)\leq R$, we deduce that $\langle u,u_i\rangle\geq \frac{\varrho_0}{2R}$.
Let $s>0$ be defined by $\langle su,u_i\rangle=t_i$. {Then $s\geq \varrho_{P(z_0)(u)}$, and equality holds if $\varrho_{P(z_0)}(u)u\in F(P(z_0), u_i)$.} Hence
$$
s-\varrho_{P(\tilde{z}_t)}(u)=
\frac{\langle su,u_i\rangle-\langle \varrho_{P(\tilde{z}_t)}(u)u,u_i\rangle}{\langle u,u_i\rangle}
\leq \frac{t_i-(t_i-c_0t^p)}{\langle u,u_i\rangle}
\leq \frac{2Rc_0}{\varrho_0}\cdot t^p,
$$
and hence $\varrho_{P(\tilde{z}_t)}(u)\geq \varrho_{P(z_0)}(u)-\frac{2Rc_0}{\varrho_0}\cdot t^p$.
We choose $t_0>0$ with $t_0\leq s_0$  depending on $z_0$ and $p$ such that
$\frac{2Rc_0}{\varrho_0}\cdot t_0^p<\varrho_0/2$.
Since $\varrho_0\leq \varrho_{P(z_0)}(u)\leq R$, we deduce from \eqref{1minustau} that
there exists $c_1>0$ depending on $\mu$, $z_0$, $q$ and $p$ 
that if $t\in(0,t_0)$ and $u\in C\cap S^{n-1}$, then
$$
\varrho_{P(\tilde{z}_t)}(u)^q\geq \left( \varrho_{P(z_0)}(u)-\frac{2Rc_0}{\varrho_0}\cdot t^p\right)^q
\geq \varrho_{P(z_0)}(u)^q-c_1\cdot t^p,
$$
which yields {(\ref{Pztc1}) by \eqref{dualintrvol} and by taking into account that $N(P(\tilde{z}_t),o)^*=N(P(z_0),o)^*$}.

The main idea of the proof of (\ref{Pztc2}) is that we construct a set $\widetilde{G}_t\subset S^{n-1}$ for sufficiently small
$t>0$ whose $\mathcal{H}^{n-1}$ measure is of order $t$, and if $u\in \widetilde{G}_t$,
then $\varrho_{P(z_t)}(u)\geq r$ for
a suitable constant $r>0$ while $\varrho_{P(\tilde{z}_t)}(u)=0$. In order to show that the constants involved really depend only on $p$, $q$ $\mu$ and $P(z_0)$, we start to set them with respect to $P(z_0)$.

We may assume, possibly after reindexing $u_1,\ldots, u_m$, that ${\rm dim}F(P(z_0),u_1)=n-1$.
In particular, there exist $r>0$ and $y_0\in F(P(z_0),u_1)\backslash\{o\}$ such that
$$
\langle y_0,u_i\rangle \leq h_{P(z_0)}(u_i)-8r\mbox{ \ for $i=2,\ldots,k$}.
$$
For $v=y_0/\|y_0\|\in S^{n-1}\cap u_1^\bot$, we consider $y=y_0+4rv$, and hence
$4r\leq \|y\|\leq R$, and
 $$
\langle y,u_i\rangle \leq h_{P(z_0)}(u_i)-4r\mbox{ \ for $i=2,\ldots,k$}.
$$
Note that $P(\tilde{z}_t)\to P(z_0)$ as $t\to 0^+$ and also $P(\tilde{z}_t)\subset P(z_0)$ for $t>0$. Therefore there exists a positive $t_1\leq\min\{r,t_0\}$, 
depending only on $p$, $q$, $\mu$ and $z_0$ such that if $t\in(0,t_1]$, then
\begin{equation}
\label{rhocond2k}
\langle y,u_i\rangle \leq h_{P(\tilde{z}_t)}(u_i)-2r\mbox{ \ for $i=2,\ldots,k$}
\mbox{ \ and \ }P(z_t)\subset RB^n.
\end{equation}
{For two vectors $a,b\in\R^n$, we denote by $[a,b]$ ($(a,b)$) the closed (open) segment with endpoints $a$ and $b$.
Let the $(n-2)$-dimensional unit ball $G$ be defined as} 
$$
G=u_1^\bot\cap v^\bot\cap B^n.
$$
{Then} we have {that} $y+r G\subset F(P(z_0),u_1)$ and $(y+r G)+r[o,u_1]\subset y+2r B^n$.
Let $G_t=(y+r G)+t(o,u_1]$ be the $(n-1)$-dimensional right spherical cylinder of height $t<\min\{t_1,r\}$, whose base  $y+r G$ {does not belong to $G_t$}. We deduce from (\ref{rhocond2k}) and 
$h_{P(z_t)}(u_1)=t$ that
$G_t\subset P(z_t)\backslash {N(P(z_0),o)^*}\subset P(z_t)\backslash P(\tilde{z}_t)$.

Let $\widetilde{G}_t$ be the the radial projection of $G_t$ to $S^{n-1}$. For $x\in G_t$,
we have $\langle x,v\rangle=\|y\|\geq 4r$ and $\|x\|\leq R$, therefore
$$
\mathcal{H}^{n-1}( \widetilde{G}_t)
=\int_{G_t}\left\langle \frac{x}{\|x\|},v\right\rangle \|x\|^{-(n-1)}\,d\mathcal{H}^{n-1}(x)
\geq \frac{4r\mathcal{H}^{n-1}(G_t)}{R^n}=\frac{4r\cdot r^{n-2}\kappa_{n-2}}{R^n}
\cdot t=\frac{4r^{n-1}\kappa_{n-2}}{R^n} \cdot t.
$$
Since  $\varrho_{P(\tilde{z}_t)}(u)\leq \varrho_{P(z_t)}(u)$ for all $u\in S^{n-1}$, and if $u\in \widetilde{G}_t$,
then $\varrho_{P(z_t)}(u)\geq \|y\|\geq 4r$ and $\varrho_{P(\tilde{z}_t)}(u)=0$, we deduce that
\begin{align*}
\Psi(z_t)&=\frac1n\int_{S^{n-1}} \varrho^q_{P(z_t)}(u) d \HH^{n-1}(u)\\
&=\frac1n\int_{S^{n-1}\setminus \widetilde{G}}\varrho^q_{P(z_t)}(u) d \HH^{n-1}(u)
+\frac1n\int_{\widetilde{G}_t}\varrho^q_{P(z_t)}(u) d \HH^{n-1}(u)\\
&\geq\frac1n\int_{S^{n-1}}\varrho^q_{P(\tilde{z}_t)}(u) d\HH^{n-1}(u)+
\frac1n\int_{\widetilde{G}_t}\varrho^q_{P(z_t)}(u) d \HH^{n-1}(u)\\
&\geq \Psi (\tilde{z}_t)+\frac{(4r)^q\cdot 4r^{n-1}\kappa_{n-2}}{nR^n} \cdot t,
\end{align*}
which proves (\ref{Pztc2}). Combining  (\ref{Pztc1}) and (\ref{Pztc2}), we obtain (\ref{Pztc1c2}). 

Finally, we deduce from $p>1$ and (\ref{Pztc1c2}) that if $t>0$ is sufficiently small, then $\Psi(P(z_t))>\Psi(P(z_0))$, which contradicts the optimality of $z_0$, and yields Lemma~\ref{interior}.
\end{proof}

As we already know that $o\in {\rm int}\,P(z_0)$ by Lemma~\ref{interior}, 
we can freely decrease $h_{P(z_0)}(u_i)$ for $i=1,\ldots,k$, and increase it if
${\rm dim}\,F(P(z_0),u_i)=n-1$. {To control what happens to $\Psi(z)$ when we perturb $P(z_0)$,} we use Lemma~\ref{varyz0}, {which is a consequence of Theorem~4.4 in \cite{HLYZ16}}.  {Let} $\R_+$ denote set of the positive real numbers.

\begin{lemma}[Huang, Lutwak, Yang, Zhang, \cite{HLYZ16}]
\label{varyz0}
 If $q\neq 0$, $\eta\in(0,1)$ and $z_t=(z_1(t),\ldots,z_k(t))\in\R_+^k$ for $t\in(-\eta,\eta)$ are such that
 $\lim_{t\to 0^+}\frac{z_i(t)-{z_i(0)}}{t}=z'_i(0)\in\R$ 
for $i=1,\ldots,k$ exists, then the $P(z_t)$ defined in \eqref{P(z)} satisfies that
$$
\lim_{t\to 0^+}\frac{\widetilde{V}_q(P(z_t))-\widetilde{V}_q(P(z_0))}{t}
=q\sum_{i=1}^k\frac{z'_i(0)}{h_{P(z_0)}(u_i)}\cdot \widetilde{C}_q(P(z_0),\{u_i\}).
$$
 \end{lemma}

For the sake of completeness, {in Section~6} we prove {a} general version of Lemma~\ref{varyz0} about the variation of $\widetilde{V}_q(P(z(t)),Q)$  {in the case when $Q$ is an arbitrary star body, cf. Lemma~\ref{varyz0Q}}.

{We note that ${\rm supp}\,C_q(P(z_0),\cdot)\subset\{u_1,\ldots,u_k\}$, 
where $\widetilde{C}_q(P(z_0),\{u_i\})>0$ if and only if\\ ${\rm dim}\,F(P(z_0),u_i)=n-1$.}

\begin{lemma}
\label{all-facets}
If $p>1$ and $q>0$, then
${\rm dim}\,F(P(z_0),u_i)=n-1$ for $i=1,\ldots,k$.
\end{lemma}
\begin{proof} 
We suppose that ${\rm dim}\,F(P(z_0),u_1)<n-1$, and seek a contradiction. We may assume that
${\rm dim}\,F(P(z_0),u_k)=n-1$. 
For small $t\geq 0$, we consider
$$
\tilde{z}(t)=(t_1-t,t_2,\ldots,t_k),
$$
and $\theta(t)=\Phi(P(\tilde{z}(t))$. In particular, $\theta(0)=1$ and $\theta'(0)=-p\alpha_1t_1^{p-1}$, and hence
$$
z(t)=\theta(t)^{-1/p}\tilde{z}(t)=(z_1(t),\ldots,z_k(t))\in Z
$$
satisfies $\frac{d}{dt} \theta(t)^{-1/p}|_{t=0^+}=\alpha_1t_1^{p-1}$ and
$z'_i(0)=\alpha_1t_1^{p-1}t_i>0$ for $i=2,\ldots,k$. We deduce from Lemma~\ref{varyz0} 
and $\widetilde{C}_q(P(z_0),\{u_1\})=0$ that
$$
\lim_{t\to 0^+}\frac{\widetilde{V}_q(P(z(t)))-\widetilde{V}_q(P(z_0))}{t}
=q\sum_{i=2}^k\frac{z'_i(0)}{h_{P(z_0)}(u_i)}\cdot \widetilde{C}_q(P(z_0),\{u_i\})\geq
\frac{q\,z'_k(0)}{h_{P(z_0)}(u_k)}\cdot \widetilde{C}_q(P(z_0),\{u_k\})>0,
$$
therefore $\widetilde{V}_q(P(z(t)))>\widetilde{V}_q(P(z_0))$ for small $t>0$.
This contradicts the optimality of $z_0$, and proves Lemma~\ref{all-facets}. 
\end{proof}

\noindent{\bf Proof of Theorem~\ref{polytopecor} } 
According to Lemmas~\ref{interior} and \ref{all-facets},

we have ${\rm dim}\,F(P(z_0),u_i)=n-1$ for $i=1,\ldots,k$, $o\in {\rm int}\,P(z_0)$
and $h_{P(z_0)}(u_i)=t_i$ for $i=1,\ldots,k$. 
Let $(g_1,\ldots,g_k)\in\R^k$ satisfying $\sum_{i=1}^kg_i\alpha_it_i^{p-1}=0$
{such that not all $g_i$ are zero}.
If $t\in(-\varepsilon,\varepsilon)$ for small $\varepsilon>0$, then consider
$$
\tilde{z}(t)=(t_1+g_1t,\ldots,t_k+g_kt),
$$
and $\theta(t)=\Phi(P(\tilde{z}(t))$. In particular, $\theta(0)=1$ and
$$
\theta'(0)=p\sum_{i=1}^kg_i\alpha_it_i^{p-1}=0.
$$ 
Therefore
$$
z(t)=\theta(t)^{-1/p}\tilde{z}(t)=(z_1(t),\ldots,z_k(t))\in Z
$$
satisfies $\frac{d}{dt} \theta(t)^{-1/p}|_{t=0}=0$ and
$z'_i(0)=g_i$ for $i=1,\ldots,k$. We deduce from Lemma~\ref{varyz0} 
and $h_{P(z_0)}(u_i)=t_i$ for $i=1,\ldots,k$ that
$$
\lim_{t\to 0}\frac{\widetilde{V}_q(P(z(t)))-\widetilde{V}_q(P(z_0))}{t}
=q\sum_{i=1}^k\frac{g_i}{t_i}\cdot \widetilde{C}_q(P(z_0),\{u_i\}).
$$
Since $\widetilde{V}_q(P(z(t)))$ attains its maximum at $t=0$ by the optimality of $z_0$, we have
\begin{equation}
\label{varyVq0}
\sum_{i=1}^k\frac{g_i}{t_i}\cdot \widetilde{C}_q(P(z_0),\{u_i\})=0.
\end{equation}
In particular, (\ref{varyVq0}) holds whenever 
$(g_1,\ldots,g_k)\in\R^k{\setminus \{o\}}$  satisfies $\sum^k_{i=1}g_i\alpha_it_i^{p-1}=0$, or in other words, there exists {a} $\lambda\in\R$ such that
$$
\lambda\cdot \frac{\widetilde{C}_q(P(z_0),\{u_i\})}{t_i}=\alpha_it_i^{p-1}
\mbox{ \ for $i=1,\ldots,k$}.
$$
Since $\lambda>0$ and $p>1$, there exists {a} $\lambda_0>0$ such that
$\lambda=\lambda_0^{-p}\widetilde{V}_q(P(z_0))$, and hence
$$
\alpha_i=\widetilde{V}_q(\lambda_0P(z_0))^{-1}h_{\lambda_0P(z_0)}(u_i)^{-p}\widetilde{C}_q(\lambda_0P(z_0),\{u_i\})
\mbox{ \ for $i=1,\ldots,k$}.
$$
In other words, 
$$
\mu=
 \widetilde{V}_q(\lambda_0P(z_0))^{-1}h_{\lambda_0P(z_0)}(u_i)^{-p}\widetilde{C}_q(\lambda_0P(z_0),\cdot).
$$ 
This finishes the proof of Theorem~\ref{polytopecor}. \hfill $\Box$\\

\noindent{\bf Proof of Theorem~\ref{polytope} in the case of $Q=B^n$  } We have $p\neq q$.
According to Theorem~\ref{polytopecor}, there exists a polytope 
$P_0\in\mathcal{K}^n_{(o)}$ such that 
$\widetilde{V}_{q}(P_0)^{-1}\widetilde{C}_{p,q}(P_0,\cdot)=\mu$.
 For $\lambda=\widetilde{V}_{q}(P_0)^{\frac{-1}{q-p}}$ and $P=\lambda P_0$, we have
$$
\widetilde{C}_{p,q}(P,\cdot)=
\lambda^{q-p}\widetilde{C}_{p,q}(P_0,\cdot)=
\widetilde{V}_{q}(P_0)^{-1}\widetilde{C}_{p,q}(P_0,\cdot)=\mu.
$$
\hfill $\Box$

\section{On the $L_p$ dual curvature measures}
\label{secLpdual}

According to Lemma~5.1 in Lutwak, Yang, Zhang \cite{LYZ18},
if $K\in \mathcal{K}^n_{(o)}$, $p\in \R$ and $q>0$, then for any Borel function $g:\,S^{n-1}\to\R$, we have {that}
\begin{equation}
\label{CpqbdK}
\int_{S^{n-1}}g(u)\,d\widetilde{C}_{p,q}(K,u)=
\frac1n\int_{\partial' K} g(\nu_K(x))\langle \nu_K(x),x\rangle^{1-p}\|x\|^{q-n}\,d\HH^{n-1}(x).
\end{equation}
As a simple consequence of Lemma~\ref{intgCqo},  we {can} partially extend \eqref{CpqbdK} to allow $o\in{\partial}K$.

\begin{coro}
\label{intgCqocor}
If $p>1$, $q>0$, $K\in \mathcal{K}^n_o$ with ${\rm int}K\neq \emptyset$,
$\widetilde{C}_{p,q}(K,S^{n-1})<\infty$ and $\HH^{n-1}(\Xi_K)=0$, and the Borel function $g:\,S^{n-1}\to \R$ is bounded, then
$$
\int_{S^{n-1}}g(u)\,d\widetilde{C}_{p,q}(K,u)=
\frac1n\int_{\partial' K} g(\nu_K(x))\langle \nu_K(x),x\rangle^{1-p}\|x\|^{q-n}\,d\HH^{n-1}(x).
$$
\end{coro}
\begin{proof}
Knowing that $\widetilde{C}_{p,q}(K,S^{n-1})<\infty$, it follows from  Lemma~\ref{intgCqo} and 
$\HH^{n-1}(\Xi_K)=0$ that
\begin{eqnarray*}
\int_{S^{n-1}}g(u)\,d\widetilde{C}_{p,q}(K,u)&=&\int_{S^{n-1}}g(u)h_K(u)^{-p}\,d\widetilde{C}_{q}(K,u)\\
&=&\frac1n\int_{\partial' K\backslash\Xi_K} g(\nu_K(x))h_K(\nu_K(x))^{-p}\langle \nu_K(x),x\rangle\|x\|^{q-n}\,d\HH^{n-1}(x)\\
&=&\frac1n\int_{\partial' K} g(\nu_K(x))\langle \nu_K(x),x\rangle^{1-p}\|x\|^{q-n}\,d\HH^{n-1}(x).
\end{eqnarray*}
\end{proof}

Next, we prove a basic estimate on the inradius of $K$ in terms of {its} $L_p$ dual curvature measure.
For a convex body $K\in \mathcal{K}^n_{(o)}$, we write $r(K)$ to denote the maximal radius of balls contained in $K$. Since $o\in K$, Steinhagen's theorem yields the existence of $w\in S^{n-1}$ such that
\begin{equation}
\label{Steinhagen}
|\langle x,w\rangle|\leq 2nr(K) \mbox{ \ for $x\in K$.}
\end{equation}

\begin{lemma}
\label{CpqVqrK}
For $n\geq 2$, $p>1$ and $q>0$, there exists {a} constant $c>0$ depending
only on $p,q,n$ such that
if $K\in \mathcal{K}^n_{(o)}$, then 
$$
\widetilde{C}_{p,q}(K,S^{n-1})\geq c\cdot r(K)^{-p}\cdot \widetilde{V}_{q}(K).
$$
 \end{lemma} 
\begin{proof}
We may assume that $r(K)=1$, and hence \eqref{Steinhagen} yields
the existence of $w\in S^{n-1}$ such that
\begin{equation}
\label{SteinhagenCpq}
|\langle x,w\rangle|\leq 2n \mbox{ \ for $x\in K$.}
\end{equation}

Let $\widetilde{K}=K|w^\bot$ {be the orthogonal projection of $K$ to the hyperplane $w^\perp$}, and hence the radial function $\varrho_{\widetilde{K}}$ is positive and continuous on 
$w^\bot\cap S^{n-1}$. We consider the concave function $f$ and the convex function $g$ on $\widetilde{K}=K|w^\bot$
such that
$$
K=\left\{y+tw:\,y\in \widetilde{K}\mbox{ and }
g(y)\leq t\leq f(y)\right\}.
$$
We divide $w^\bot\cap S^{n-1}$ into {pairwise} disjoint Borel sets $\widetilde{\Omega}_1,\ldots,\widetilde{\Omega}_m$ of positive $\HH^{n-2}$ measure such that for each 
{$\widetilde{\Omega}_i$}, there exists {a $\varrho_i>0$} satisfying
\begin{equation}
\label{rhoi}
\varrho_i/2\leq \varrho_{\widetilde{K}}(u)\leq \varrho_i\mbox{ \ for $u\in \widetilde{\Omega}_i$.}
\end{equation}
For any $i=1,\ldots,m$, we consider
\begin{eqnarray*}
\Omega_i&=&\left\{u\cos \alpha +w\sin\alpha:\,u\in \widetilde{\Omega}_i\mbox{ and }
\alpha\in\left(-\frac{\pi}2,\frac{\pi}2\right)\right\}\subset S^{n-1},\\
\Psi_i&=&\left\{\varrho_K(u)u:\,u\in \Omega_i \right\}\subset {\partial}K.
\end{eqnarray*}
It follows that $S^{n-1}\backslash\{w,-w\}$ is divided into the {pairwise} disjoint Borel sets $\Omega_1,\ldots,\Omega_m$,
and ${\partial} K\backslash\{f(o)w,g(o)w\}$  is divided into the {pairwise} disjoint Borel sets $\Psi_1,\ldots,\Psi_m$.

According to \eqref{CpqbdK} and Lemma~\ref{intgCqo}, to verify Lemma~\ref{CpqVqrK}, it is sufficient to prove that
there exists {a} constant $c>0$ depending only on $n,p,q$ such that if $i=1,\ldots,m$, then
\begin{equation}
\label{Psii}
\int_{\partial' K\cap \Psi_i} \langle \nu_K(x),x\rangle^{1-p}\|x\|^{q-n}\,d\HH^{n-1}(x)\geq
c\int_{\partial' K\cap \Psi_i} \langle \nu_K(x),x\rangle\|x\|^{q-n}\,d\HH^{n-1}(x).
\end{equation}

We define
\begin{equation}
\label{CpqVqrKRdef}
R=4(2n)^2.
\end{equation}
\noindent{\em {Case 1.}}\; If $\varrho_i\leq R$, then \eqref{SteinhagenCpq} yields that 
$$
\langle \nu_K(x),x\rangle\leq \|x\|\leq R+2n\mbox{ \ for $x\in \Psi_i$},
$$
and hence $\langle \nu_K(x),x\rangle^{1-p}\geq\langle \nu_K(x),x\rangle(R+2n)^{-p}$.
Therefore we may choose $c=(R+2n)^{-p}$ in \eqref{Psii}.

\smallskip

\noindent{\em {Case 2.}}\; {If $\varrho_i> R$, then consider the set}
$$
\Phi_i=\left\{t{u}:\,u\in \widetilde{\Omega}_i \mbox{ \ and \ }0<t\leq\varrho_i/4\right\}\subset \Psi_i|w^\bot,
$$
and subdivide $\Psi_i$ into
\begin{eqnarray*}
\Psi_i^0&=&\left\{y+f(y)w:\,y\in \Phi_i \right\}\cup\left\{y+g(y)w:\,y\in \Phi_i \right\}\subset 
\Psi_i\cap\left(\mbox{$\frac{\varrho_i}4+2n$}\right)B^n,\; \text{{and}}\\
\Psi_i^1&=&\Psi_i\backslash \Psi_i^0 \subset 
\Psi_i\backslash\left(\mbox{$\frac{\varrho_i}4$}\,B^n\right).
\end{eqnarray*}

We claim that 
\begin{equation}
\label{xnuxPsi0}
\langle \nu_K(x),x\rangle\leq 6n \mbox{ \ for $x\in \Psi_i^0$}.
\end{equation}
We observe that $x=y+tw$ for some $y\in\Phi_i$ and $t\in[-2n,2n]$,
and $s=f(2y)$ satisfies $s\in[-2n,2n]$ and $2y+sw\in\Psi_i$. It follows that
$$
\langle \nu_K(x),2y+sw\rangle\leq \langle \nu_K(x),x\rangle=\langle \nu_K(x),y+t{w}\rangle,
$$
and hence 
$$
\langle \nu_K(x),y\rangle\leq \langle \nu_K(x),t{w}\rangle-{\langle}\nu_K(x),s{w}\rangle\leq 4n.
$$
We conclude that $\langle \nu_K(x),y+t{w}\rangle=\langle \nu_K(x),y\rangle+\langle \nu_K(x),t{w}\rangle\leq 6n$,
in accordance with \eqref{xnuxPsi0}. 

In turn, \eqref{xnuxPsi0} yields that
${\langle}\nu_K(x),x\rangle^{1-p}\geq {\langle}\nu_K(x),x\rangle(6n)^{{-p}}$ for $x\in \partial' K\cap \Psi_i^0$, and hence
\begin{equation}
\label{Psii0}
\int_{\partial' K\cap \Psi_i^0} \langle \nu_K(x),x\rangle^{1-p}\|x\|^{q-n}\,d\HH^{n-1}(x)\geq
(6n)^{-p}\int_{\partial' K\cap \Psi_i^0} \langle \nu_K(x),x\rangle\|x\|^{q-n}\,d\HH^{n-1}(x).
\end{equation}

Next, we prove the existence of $\gamma_1>0$ depending on $n,p,q$ such that
\begin{equation}
\label{Psii0est}
\int_{\partial' K\cap {\Psi_i^0}} \langle \nu_K(x),x\rangle^{1-p}\|x\|^{q-n}\,d\HH^{n-1}(x)\geq
\left\{
\begin{array}{lcl}
\gamma_1\HH^{n-2}(\widetilde{\Omega}_i)\varrho_i^{q-1}&\mbox{ if }&q>1\\[0.5ex]
\gamma_1\HH^{n-2}(\widetilde{\Omega}_i)&\mbox{ if }&{q\in(0,1]}
\end{array}
\right. .
\end{equation}
Let us consider $x=y+f(y)w\in \Psi_i^0\cap \partial' K$ for some
$y\in\Phi_i\backslash (2nB^n)$. Since $\|y\|\leq\|x\|\leq 2\|y\|$ by  \eqref{SteinhagenCpq},
it follows from \eqref{xnuxPsi0} that
$$
\langle\nu_K(x),x\rangle^{1-p}\|x\|^{q-n}\geq (6n)^{1-p}\min\{1,2^{q-n} \}\,\|y\|^{q-n}.
$$
Therefore there exists $\gamma_2>0$ depending on $n,p,q$ such that
\begin{eqnarray*}
\int_{\partial' K\cap {\Psi_i^0}} \langle \nu_K(x),x\rangle^{1-p}\|x\|^{q-n}\,d\HH^{n-1}(x)&\geq&
\gamma_2\int_{\Phi_i\backslash (2nB^n)}\|y\|^{q-n}\,d\HH^{n-1}(x)\\
&=&\gamma_2\HH^{n-2}(\widetilde{\Omega}_i)\int_{2n}^{\varrho_i/4}t^{q-n}t^{n-2}\,dt\\
&=&{\gamma_2\HH^{n-2}(\widetilde{\Omega}_i)\int_{2n}^{\varrho_i/4}t^{q-2}\,dt},
\end{eqnarray*}
and in turn we conclude \eqref{Psii0est}.

The final {part of the argument} is the estimate
\begin{equation}
\label{Psii1est}
\int_{\partial' K\cap \Psi_i^1} \langle \nu_K(x),x\rangle\|x\|^{q-n}\,d\HH^{n-1}(x)\leq
2^q16n\cdot \HH^{n-2}(\widetilde{\Omega}_i)\cdot \varrho_i^{q-1}.
\end{equation}
{Let} $\Omega_i^1=\pi_K(\Psi_i^1)$.
If $x=y+sw\in \Psi_i^1$ for $y\in {(\Psi_i|w^\perp)\setminus \Phi_i}$, then $y\in {(\Psi_i|w^\perp)\backslash (\frac{\varrho_i}4\,B^n)}$
and $|s|\leq 2n$. It follows that
$|\tan\alpha|\leq \frac{2n}{\varrho_i/4}=\frac{8n}{\varrho_i}$ for the angle $\alpha$ of $x$ and $y$.
In particular, 
$$
\Omega_i^1\subset{\pi_K}\left(\widetilde{\Omega}_i+
\left[\frac{-8n}{\varrho_i},\frac{8n}{\varrho_i}\right]\cdot w\right)
$$ 
which, in turn, yields that
$$
\HH^{n-1}(\Omega_i^1)\leq \frac{16n}{\varrho_i}\,\HH^{n-2}(\widetilde{\Omega}_i).
$$
We deduce from \eqref{piKGaussian} and from the fact that
$\|x\|\leq \varrho_i+2n\leq 2\varrho_i$ for $x\in\Psi_i^1$ that
$$
\int_{\partial' K\cap \Psi_i^1} \langle \nu_K(x),x\rangle\|x\|^{q-n}\,d\HH^{n-1}(x)=
\int_{\Omega_i^1}\varrho_K(u)^q\,d\HH^{n-1}(u)
\leq\frac{16n}{\varrho_i}\,\HH^{n-2}(\widetilde{\Omega}_i)\cdot (2\varrho_i)^q,
$$
yielding \eqref{Psii1est}.

{
We deduce from \eqref{Psii0est} and \eqref{Psii1est} 
the existence of $\gamma_3>0$ depending on $n,p,q$ such that
\begin{equation}
\label{gamma3}
\int_{\partial' K\cap \Psi_i^0} \langle \nu_K(x),x\rangle^{1-p}\|x\|^{q-n}\,d\HH^{n-1}(x)\geq
\gamma_3\int_{\partial' K\cap \Psi_i^1} \langle \nu_K(x),x\rangle\|x\|^{q-n}\,d\HH^{n-1}(x).
\end{equation}
Combining \eqref{Psii0} and \eqref{gamma3} implies
\eqref{Psii} if $\varrho_i>R$, as well, completing the proof of Lemma~\ref{CpqVqrK}.}
\end{proof}

Next we investigate the limit of convex bodies with bounded $L_p$ dual curvature measure in
Lemmas~\ref{p>q>0} and \ref{KmXiK}.

\begin{lemma}
\label{p>q>0}
If $p>1$, $0<q\leq p$ and $K_m\in \mathcal{K}^n_{(o)}$ for $m\in\N$ tend to 
$K\in \mathcal{K}^n_o$ with ${\rm int}K\neq \emptyset$ such that
$\widetilde{C}_{p,q}(K_m,S^{n-1})$ stays bounded, then $K\in \mathcal{K}^n_{(o)}$.
 \end{lemma} 
\begin{proof} Let us suppose that $o\in {\partial}K$, and seek a contradiction. We claim that there exists {a vector} 
$w\in{\rm int}N(K,o)^*$ such that $-w\in N(K,o)\cap S^{n-1}$.
 If this property fails, then $(-N(K,o))\cap {\rm int}N(K,o)^*=\emptyset$, and hence
the Hahn-Banach theorem yields {the existence of a vector} $v\in S^{n-1}$ such that 
$\langle v,u\rangle\leq 0$ if $u\in N(K,o)^*$, and 
$\langle v,u\rangle\geq 0$ if $u\in -N(K,o)$, and hence $v\in N(K,o)^*$. Since $\langle v,v\rangle=1>0$ 
contradicts $\langle v,u\rangle\leq 0$ if $u\in N(K,o)^*$, we conclude the existence of the required $w$.

To simplify notation, we set $B(r)=w^\bot\cap(r B^n)$ for $r>0$.
The conditions in Lemma~\ref{p>q>0} and \eqref{CpqbdK} yield the existence of some $M>0$ such that for each $K_m$, we have {that}
\begin{eqnarray}
\label{CpqbdKmM}
M&>&\widetilde{C}_{p,q}(K_m,S^{n-1})=
\frac1n\int_{\partial' K_m} \langle \nu(K_m,x),x\rangle^{1-p}\|x\|^{q-n}\,d\HH^{n-1}(x)\\
\nonumber
&\geq&
\frac1n\int_{\partial' K_m\cap B^n} \|x\|^{1-n+q-p}\,d\HH^{n-1}(x)\geq
\frac1n\int_{\partial' K_m\cap B^n} \|x\|^{1-n}\,d\HH^{n-1}(x).
\end{eqnarray}
{We note that since $K_m\to K$ and $o\in\partial K$, for sufficiently large $m$,  $\partial'K_m\cap B^n\neq\emptyset$ and the right-hand side of \eqref{CpqbdKmM} is greater than zero.}
As $w\in{\rm int}N(K,o)^*$ and $w\in N(K,o)$, there exist a $\varrho\in(0,1)$ and a non-negativ convex function $f$ on $B(2\varrho)$ 
with $f(o)=0$ such that
$$
U=\{z+f(z)w:\,z\in B(2\varrho)\}{\partial} K.
$$
 In particular, there exist {an} $\eta>0$ such that
\begin{equation}
\label{xprojlength}
\|x|w^\bot\|\geq 2\eta \|x\|\mbox{ \ for $x\in U$.}
\end{equation}
We may assume that $\varrho\in(0,1)$ is small enough to ensure that $U\subset {\rm int}B^n$.

Since $\int_{B(\varrho)}\|z\|^{1-n}d\HH^{n-1}(z)=\infty$,  there exists some $\delta\in(0,\varrho)$ such that
\begin{equation}
\label{projintlarge}
\frac1n\int_{B(\varrho)\backslash B(\delta)}\left(\frac{\|z\|}{\eta}\right)^{1-n}d\HH^{n-1}(z)> M.
\end{equation}
There exist and {an} $m_0$ such that if $m>m_0$, then for some convex function $f_m$ on $B(\varrho)$, we have 
\begin{equation}
\label{UmfmKm}
U_m=\left\{z+f_m(z)w:\,z\in B(\varrho)\backslash B(\delta)\right\}
\subset({\partial} K_m)\cap({\rm int}B^n),
\end{equation}
and (compare \eqref{xprojlength})
\begin{equation}
\label{xprojlengthm}
\|z\|\geq \eta \|z+f_m(z)w\|\mbox{ \ for $z\in B(\varrho)\backslash B(\delta)$.}
\end{equation}
We deduce from \eqref{CpqbdKmM}, \eqref{UmfmKm} and \eqref{xprojlengthm},
and finally from \eqref{projintlarge} that
\begin{eqnarray*}
M&>&\frac1n\int_{U_m} \|x\|^{1-n}\,d\HH^{n-1}(x)\geq
\frac1n\int_{B(\varrho)\backslash B(\delta)}\|z+f_m(z)w\|^{1-n}d\HH^{n-1}(z)\\
&\geq & 
\frac1n\int_{B(\varrho)\backslash B(\delta)}\left(\frac{\|z\|}{\eta}\right)^{1-n}d\HH^{n-1}(z)> M.
\end{eqnarray*}
This is {a contradiction}, and in turn we conclude Lemma~\ref{p>q>0}.
\end{proof}

\begin{lemma}
\label{KmXiK}
If $p>1$, $q>0$ and $K_m\in \mathcal{K}^n_{(o)}$ for $m\in\N$ tend to 
$K\in \mathcal{K}^n_o$ with ${\rm int}K\neq \emptyset$ such that
$\widetilde{C}_{p,q}(K_m,S^{n-1})$ stays bounded, then $\HH^{n-1}(\Xi_K)=0$.
 \end{lemma} 
\begin{proof} We fix a {point} $z\in {\rm int}K$, and for any bounded $X\subset \R^n\backslash\{z\}$, we {define the} set
$$
\sigma(X)=\{z+\lambda(x-z):\,x\in X\mbox{ and }\lambda>0\}.
$$
We observe that $\sigma(X)$ is open if $X\subset{\partial}K$ is relatively open,  
and $\sigma(X)\cup\{o\}$ is closed if $X$ is compact.

We will use the weak continuity of the $(n-1)$th curvature measure. In particular, 
according to Theorem 4.2.1 and Theorem 4.2.3 in Schneider \cite{Sch14}, if $\beta\subset\R^n$ is open, then
\begin{equation}
\label{curvatureweakcont}
\liminf_{m\to\infty}\HH^{n-1}(\beta\cap {\partial}\,K_m)\geq \HH^{n-1}(\beta\cap {\partial}\,K).
\end{equation}

Let us suppose, {on the contrary,} that $\HH^{n-1}(\Xi_K)>0$, and hence $o\in {\partial}K$, and seek a contradiction.
Choose some large $M,R>0$, and {a compact set} $\widetilde{\Xi}\subset \Xi_K\backslash \{o\}$ such that 
\begin{eqnarray*}
K_m&\subset&RB^n,\\
\widetilde{C}_{p,q}(K_m,S^{n-1})&\leq &M\mbox{ \ for $m\in \N$},\\
\HH^{n-1}(\widetilde{\Xi})&=&\omega>0.
\end{eqnarray*}
Now there exists some $\eta>0$ such that
\begin{description}
\item[(i)] $(\eta B^n)\cap \sigma(\widetilde{\Xi}+\eta B^n)=\emptyset$.
\end{description}
{Since $p>1$, we may} choose $\varepsilon>0$ such that
\begin{equation}
\label{epsilonetaomega}
\frac{(2\varepsilon)^{1-p}}n\cdot\min\{\eta^{q-n},R^{q-n}\}\cdot(\omega/2)>M.
\end{equation}
We have $\HH^{n-1}(\widetilde{\Xi}\cap\partial' K)=\omega$.
For any $x\in  \widetilde{\Xi}\cap\partial' K$, there exists $r_x\in(0,\eta)$ such that
\begin{equation}
\label{Bxcond}
h_K(u)\leq \varepsilon \mbox{ \ if $u\in S^{n-1}$ is exterior normal at $y\in {\partial}K\cap (x+r_xB^n)$, }
\end{equation}
and we define $B_x={\rm int}(x+r_xB^n)$. Let
$$
\mathcal{U}=\bigcup_{x\in \widetilde{\Xi}\cap\partial' K}(B_x\cap {\partial}K),
$$
which is a relatively open subset of ${\partial}K$ satisfying
\begin{description}
\item[(a)] $(\eta B^n)\cap \sigma(\mathcal{U})=\emptyset$,
\item[(b)] $\HH^{n-1}(\mathcal{U})\geq\omega$,
\item[(c)] $h_K(u)\leq \varepsilon$ if $u\in S^{n-1}$ is exterior normal at $x\in {\rm cl}\,\mathcal{U}$.
\end{description}
It follows that (applying \eqref{curvatureweakcont} in the case (b')) that
 there exists $m_0$ such that if $m\geq m_0$, then
\begin{description}
\item[(a')] $\|x\|\geq \eta$ if $x\in \sigma(\mathcal{U})\cap{\partial}K_m$,
\item[(b')] $\HH^{n-1}(\sigma(\mathcal{U})\cap{\partial}K_m)\geq\omega/2$,
\item[(c')] $h_K(u)\leq 2\varepsilon$ if $u\in S^{n-1}$ is exterior normal at 
$x\in \sigma(\mathcal{U})\cap{\partial}K_m$.
\end{description}

For any $x\in \sigma(\mathcal{U})\cap{\partial}K_m$, (a') and $K_m\subset RB^n$ yield that
$$
\|x\|^{q-n}\geq \min\{\eta^{q-n},R^{q-n}\}.
$$
It follows {first} by \eqref{CpqbdK}, then by (b'), (c') and \eqref{epsilonetaomega}, that
$$
M\geq \widetilde{C}_{p,q}(K_m,S^{n-1})\geq
\frac1n\int_{\sigma(\mathcal{U})\cap\partial'K_m} \langle \nu_K(x),x\rangle^{1-p}\|x\|^{q-n}\,d\HH^{n-1}(x)
>M.
$$
This contradiction proves Lemma~\ref{KmXiK}.
\end{proof}

\section{Theorem~\ref{main} for general convex bodies if $Q=B^n$}
\label{secmainBn}

For $w\in S^{n-1}$ and $\alpha\in(-1,1)$, we write
$$
\Omega(w,\alpha)=\{u\in S^{n-1}:\langle u,w\rangle >\alpha\}.
$$
The following is a simple but useful observation.

\begin{lemma}
\label{sphericalstrips}
For a finite Borel measure $\mu$ on $S^{n-1}$ not concentrated on a  closed hemi-sphere, there exists $t\in(0,1)$ such that
for any $w\in S^{n-1}$, we have $\mu(\Omega(w,t))>t$.
\end{lemma}

First we prove the following variant of Theorem~\ref{main} involving the dual intrinsic volume.

\begin{theorem}
\label{withvolume}
For $p>1$ and $q>0$, and finite Borel measure $\mu$ on $S^{n-1}$ not concentrated on a  closed hemi-sphere, there exists a convex body $K\in\mathcal{K}^n_o$ with ${\rm int}K\neq 0$ and $\HH^{n-1}(\Xi_K)=0$ such that
$$
 \widetilde{V}_q(K)h_K^{p}d\mu= d\widetilde{C}_{q}(K,\cdot),
$$
and in addition, $K\in\mathcal{K}^n_{(o)}$ if $p\geq q$.
\end{theorem}
\begin{proof} 
We choose a sequence of discrete measures $\mu_m$ tending to $\mu$ that are not concentrated on any closed hemispheres.  It follows from Theorem~\ref{polytopecor}, 
that there exists polytope $P_m\in\mathcal{K}^n_{(o)}$ such that
\begin{equation}
\label{mumPm}
d\mu_m= \frac1{ \widetilde{V}_q(P_m)}\,d\widetilde{C}_{p,q}(P_m,\cdot)=
\frac{h_{P_m}^{-p}}{ \widetilde{V}_q(P_m)}\,d\widetilde{C}_{q}(P_m,\cdot)
\end{equation}
for each $m$, and hence we may assume that
\begin{equation}
\label{mumPmupper}
\frac{\widetilde{C}_{p,q}(P_m,S^{n-1})}{ \widetilde{V}_q(P_m)}<2\mu(S^{n-1}).
\end{equation}

We claim that there exists $R>0$ such that
\begin{equation}
\label{Pmbounded}
P_m\subset RB^n.
\end{equation}
We prove \eqref{Pmbounded} by contradiction, thus we suppose that
$R_m=\max_{x\in P_m}\|x\|$ tends to infinity.
We choose $v_m\in S^{n-1}$ such that $R_mv_m\in P_m$, and we may assume by possibly taking a subsequence that $v_m$ tends to $v\in S^{n-1}$. We deduce from Lemma~\ref{sphericalstrips} that there exist $s,t>0$ such that
$\mu(\Omega(v,2t))>2s$. As $v_m$ tends to $v\in S^{n-1}$ and $\mu_m$ tends weakly to $\mu$, we may also assume that $\Omega(v,2t)\subset \Omega(v_m,t)$ and
$\mu_m(\Omega(v,2t))>s$, therefore $\mu_m(\Omega(v_m,t))>s$ for each $m$.
Since $h_{P_m}(u)\geq \langle R_mv_m,u\rangle\geq R_mt$ for $u\in \Omega(v_m,t)$, we deduce from
\eqref{mumPm} that
$$
s<\mu_m(\Omega(v_m,t))=
\int_{\Omega(v_m,t)}\frac{h_{P_m}^{-p}(u)}{ \widetilde{V}_q(P_m)}\,d\widetilde{C}_{q}(P_m,u)\leq
R_m^{-p}t^{-p}\frac{\widetilde{C}_{q}(P_m,S^{n-1})}{ \widetilde{V}_q(P_m)}
\leq R_m^{-p}t^{-p}.
$$
In particular, $R_m^p\leq s^{-1}t^{-p}$, contradicting the fact that $R_m$ tends to infinity, and in turn proving 
\eqref{Pmbounded}.

It follows from \eqref{Pmbounded} that $P_m$ tends to a compact convex set $K\in\mathcal{K}^n_{o}$ with 
$K\subset RB^n$. We deduce from \eqref{mumPmupper} and Lemma~\ref{CpqVqrK} that
$r(K)>0$.

We observe that  $h_{P_m}^{p}$ tends uniformly to $h_{K}^{p}$, and hence
also  $\widetilde{V}_q(P_m)h_{P_m}^{p-1}$ tends uniformly to $\widetilde{V}_q(K)h_{K}^{p-1}$
by Lemma~\ref{Vqcont}. Therefore given any continous function $f$, we have
$$
\lim_{m\to \infty}\int_{S^{n-1}}f(u)\widetilde{V}_q(P_m)h_{P_m}^{p-1}(u)\,d\mu_m=
\int_{S^{n-1}}f(u)\widetilde{V}_q(K)h_{K}^{p-1}(u)\,d\mu.
$$
It follows from
Proposition~\ref{Cqcont}
that the dual curvature measure $\widetilde{C}_{q}(P_m,\cdot)$ tends
weakly to $\widetilde{C}_{q}(K,\cdot)$, thus \eqref{mumPm} yields
$$
\int_{S^{n-1}}f(u)\widetilde{V}_q(K)h_{K}^{p}(u)\,d\mu=
\int_{S^{n-1}}f(u)\,d\widetilde{C}_{q}(K,u).
$$
Since the last property holds for all continuos function $f$, we conclude that
$$
\widetilde{V}_q(K)h_{K}^{p}\,d\mu=
d\widetilde{C}_{q}(K,\cdot),
$$
as it is required.

Having \eqref{mumPmupper} at hand, Lemma~\ref{KmXiK} yields that $\HH^{n-1}(\Xi_K)=0$,
and Lemma~\ref{p>q>0} implies that  if $p\geq q$, then
$K\in\mathcal{K}^n_{(o)}$.
\end{proof}

\noindent{\bf Proof of Theorem~\ref{main} in the case of $Q=B^n$ }
Let $p>1$, $q>0$ and $p\neq q$.
According to Theorem~\ref{withvolume}, there exists a  
$K_0\in\mathcal{K}^n_{(o)}$ with ${\rm int}K_0\neq\emptyset$ and $\HH^{n-1}(\Xi_{K_0})=0$
such that 
$\widetilde{V}_{q}(K_0)^{-1}\widetilde{C}_{p,q}(K_0,\cdot)=\mu$.
 For $\lambda=\widetilde{V}_{q}(K_0)^{\frac{-1}{q-p}}$ and $K=\lambda K_0$, we have
$$
\widetilde{C}_{p,q}(K,\cdot)=
\lambda^{q-p}\widetilde{C}_{p,q}(K_0,\cdot)=
\widetilde{V}_{q}(K_0)^{-1}\widetilde{C}_{p,q}(K_0,\cdot)=\mu.
$$
 It follows from Theorem~\ref{withvolume} that $o\in{\rm int}K$ if $p>q$. $\Box$

\section{The $L_p$ dual curvature measure involving the star body $Q$}
\label{secQincluded}

In this section, we discuss how to extend the results 
of Sections~\ref{secdualcurvature} to \ref{secmainBn}
about {dual curvature measures
$\widetilde{C}_{q}(K,\cdot)$ and $L_p$ dual curvature measures $\widetilde{C}_{p,q}(K,\cdot)$
to $\widetilde{C}_{q}(K,Q,\cdot)$ and $\widetilde{C}_{p,q}(K,Q,\cdot)$, where $Q$ is a star body.}
We recall that if $q>0$, $Q\in\mathcal{S}_{(o)}^n$ and $K\in \mathcal{K}^n_{o}$, then
\begin{equation}
\label{dualintrinsicQ0} 
\widetilde{V}_q(K,Q)=
\frac 1n\int_{S^{n-1}}\varrho_K^{q}(u)\varrho_Q^{n-q}(u)\,{d}\HH^{n-1}(u),
\end{equation} 
and if, in addition, $\eta\subset S^{n-1}$ is a Borel set, then
\begin{equation}
\label{dualcurvmeasureQ0} 
\widetilde{C}_q(K,Q,\eta)=
\frac 1n\int_{\alpha^*_K(\eta)}\varrho_K^{q}(u)\varrho_Q^{n-q}(u)\,{d}\HH^{n-1}(u).
\end{equation} 

Since for $Q\in \mathcal{S}^n_{(o)}$, $\varrho_{Q}$ is a positive continuous function on $S^{n-1}$,
essentially the same arguments as in Section~\ref{secdualcurvature} yield the analogues 
Lemmas~\ref{intgCqoQ}, \ref{VqcontQ} and {Proposition~\ref{CqcontQ}}
of Lemmas~\ref{intgCqo}, \ref{Vqcont} and \ref{Cqcont}. We note that
\begin{equation}
\label{N(K,o)Q}
\widetilde{C}_{q}(K,Q,S^{n-1}\cap N(K,o))=0
\end{equation}
as $\varrho_K(u)=0$ if $u\in\alpha^*_K({\rm int}N(K,o))$, and 
$$
\alpha^*_K(S^{n-1}\cap N(K,o))\backslash \alpha^*_K(S^{n-1}\cap {\rm int}N(K,o)) \subset
S^{n-1}\cap  {\partial}N(K,o)^*
$$
and $\mathcal{H}^{n-1}( S^{n-1}\cap  {\partial}N(K,o)^*)=0$.

For Lemma~\ref{intgCqo}, the only additional observation needed is that if $u\in S^{n-1}$
and $x=\varrho_K(u) u\in{\partial}K$, then
$\|x\|_Q=\varrho_K(u)/\varrho_Q(u)$.

\begin{lemma}
\label{intgCqoQ}
If $q>0$, $Q\in \mathcal{S}^n_{{(0)}}$, $K\in \mathcal{K}^n_o$ with ${\rm int}K\neq \emptyset$, and the Borel function $g:\,S^{n-1}\to \R$ is bounded, then
\begin{eqnarray}
\label{intgCqo1Q}
\int_{S^{n-1}}g(u)\,d\widetilde{C}_{q}(K,Q,u)&=&\frac1n
\int_{S^{n-1}\cap ({\rm int}N(K,o)^*)}g(\alpha_K(u))\varrho_K(u)^q\varrho_Q(u)^{n-q}\,d\HH^{n-1}(u)\\ 
\label{intgCqo2Q}
&=&\frac1n\int_{\partial' K\backslash \Xi_K} g(\nu_K(x))\langle \nu_K(x),x\rangle\|x\|_Q^{q-n}\,d\HH^{n-1}(x),\\
\label{intgCqo3Q}
&=&\frac1n\int_{\partial' K} g(\nu_K(x))\langle \nu_K(x),x\rangle\|x\|_Q^{q-n}\,d\HH^{n-1}(x)
\end{eqnarray}
\end{lemma}

{From \eqref{intgCqo3Q}} we deduce the following.

\begin{coro}
\label{SKCqQ}
If $q>0$, $Q\in \mathcal{S}^n_{{(0)}}$, $K\in \mathcal{K}^n_o$ with ${\rm int}K\neq \emptyset$
and $\HH^{n-1}(\Xi_K)=0$, then
the surface area measure $S(K,\cdot)$ is absolutely continuous with respect to 
$\widetilde{C}_{q}(K,Q,\cdot)$.
\end{coro}

Corollary~\ref{SKCqQ} will be useful {for} the differential equation representing the 
$L_p$ dual Minkowski problem in Section~\ref{secregularity}.

Now, arguments verifying Lemma~\ref{VqcontQ} and {Proposition~\ref{CqcontQ}}
use  \eqref{intgCqo1Q} {in a similar way} as the proofs of
Lemmas~\ref{Vqcont} and \ref{Cqcont} are based on \eqref{intgCqo1}.

\begin{lemma}
\label{VqcontQ}
For $q>0$ and $Q\in \mathcal{S}^n_{{(0)}}$, $\widetilde{V}_q(K)$ is a continuous function of $K\in \mathcal{K}^n_o$ with respect to the Hausdorff distance.
\end{lemma}

\begin{prop}
\label{CqcontQ}
If $q>0$, $Q\in \mathcal{S}^n_{{(0)}}$ and $\{K_m\}$, $m\in\N$, tends to $K$ for $K_m,K\in \mathcal{K}^n_o$, then 
$\widetilde{C}_q(K_m,Q,\cdot)$ tends weakly to $\widetilde{C}_q(K,Q,\cdot)$.
\end{prop}

For $q>0$, we extend Theorem~6.8 in \cite{LYZ18} (see \eqref{Cqaffineinv0}) to any convex body containing the origin on {its} boundary. For $Q\in \mathcal{S}^n_{(o)}$, we observe that if  $P\in\mathcal{K}^n_{o}$ is a polytope 
with ${\rm int}\,P\neq\emptyset$ and 
$v_1,\ldots,v_l\in S^{n-1}$ are the exterior normals of the facets of $P$ not containing the origin, 
then Lemma~\ref{intgCqoQ} yields
\begin{equation}
\label{CqP}
\begin{array}{rcll}
{\rm supp} \,\widetilde{C}_q(P,Q,\cdot)&=&\{v_1,\ldots,v_l\}, \quad\text{{ and}}&\\[1ex]
\widetilde{C}_q(P,Q,\{v_i\})&=&\displaystyle{\frac 1n}\int_{\tilde{\pi}(F(P,v_i))}
\varrho_{P}^q(u)\varrho_{Q}^{n-q}(u)\, d\HH^{n-1}(u)&\\[2ex]
&=&\displaystyle{\frac 1n}\int_{F(P,v_i))}h_{P}(v_i)\|x\|_{Q}^{q-n}\,d\HH^{n-1}(x) &
\mbox{ \ for $i=1,\ldots,l$}.
\end{array}
\end{equation}

\begin{lemma}
\label{CqslnQ}
If $q>0$,  $K\in \mathcal{K}^n_o$, $Q\in \mathcal{S}^n_{(o)}$, $g$ is a bounded {real} Borel function on $S^{n-1}$ and 
$\varphi\in{\rm SL}(n,\R)$, then 
$$
\int_{S^{n-1}}g(u)\,d\widetilde{C}_{q}(\varphi K,\varphi Q,u)=
\int_{S^{n-1}} g\left(\frac{\varphi^{-t} u}{\|\varphi^{-t} u\|}\right)d\widetilde{C}_{q}( K,Q,u).
$$
\end{lemma}
\begin{proof} It is suficient to prove Lemma~\ref{CqslnQ} {for the case when} $g$ is continuous. Therefore,
it follows from Proposition~\ref{CqcontQ} and polytopal approximation that
we may assume that $K$ is an $n$-dimensional polytope. We write  $u_1,\ldots,u_k$ to denote the exterior unit normals of $K$, and set $F_i=F(K,u_i)$ for $i=1,\ldots,k$. It follows that the exterior unit normal at the facet $\varphi F_i$ of
$\varphi K$ is $v_i=\frac{\varphi^{-t}u_i}{\|\varphi^{-t}u_i\|}$ for $i=1,\ldots,k$.

For any $i=1,\ldots,k$, $\det \varphi=1$ yields that the volumes of the cones over the facets do not change,
and hence 
$\frac1n h_{\varphi K}(v_i)\cdot \HH^{n-1}(\varphi F_i)=\frac1nh_{K}(u_i)\cdot \HH^{n-1}(F_i)$,
which in turn implies that
\begin{equation}
\label{phionuibot}
\det\left(\varphi|_{u_i^\bot}\right)=\frac{h_{K}(u_i)}{h_{\varphi K}(v_i)}. 
\end{equation}
We note that  the linearity of $\varphi$ yields 
$\|\varphi y\|_{\varphi Q}=\|y\|_{Q}$ for any $y\in\R^n$.
We deduce first from \eqref{CqP} and later from \eqref{phionuibot} that
\begin{eqnarray*}
\int_{S^{n-1}}g(u)\,d\widetilde{C}_{q}(\varphi K,\varphi Q,u)&=&
\frac1n\sum_{i=1}^k\int_{\varphi F_i} g(v_i)\|x\|_{\varphi Q}^{q-n}h_{\varphi K}(v_i)\,d\HH^{n-1}(x)\\
&=&\frac1n\sum_{i=1}^k\int_{F_i} g(v_i)\|y\|_Q^{q-n}h_{\varphi K}(v_i)\det\left(\varphi|_{u_i^\bot}\right)\,d\HH^{n-1}(y)\\
&=& \frac1n\sum_{i=1}^k\int_{F_i} g\left(\frac{\varphi^{-t}u_i}{\|\varphi^{-t}u_i\|}\right)\|y\|_Q^{q-n}
h_{K}(u_i)\,d\HH^{n-1}(y),
\end{eqnarray*}
which in turn implies  Lemma~\ref{CqslnQ} by \eqref{CqP}.
\end{proof}

bubu

For $w\in S^{n-1}$ and $\alpha\in(-1,1)$, we define
$$
\Gamma(w,\alpha)=\{u\in S^{n-1}:|\langle u,w\rangle| {<}\alpha\}.
$$
Since the restriction of the radial projection $\tilde{\pi}$ satisfies that
$\|\tilde{\pi}(x_1)- \tilde{\pi}(x_2)\|\leq \|x_1-x_2\|$ for $x_1,x_2\in(w^\bot\cap S^{n-1})+{\rm lin}\,w$, we have

\begin{lemma}
\label{sphericalstripsest}
If $w\in S^{n-1}$ and $\alpha\in(-1,1)$, then
$$
\mathcal{H}^{n-1}(\Gamma(w,\alpha))\leq 
{(n-2)\kappa_{n-2}}\cdot 
\frac{2\alpha}{\sqrt{1-\alpha^2}}.
$$
\end{lemma}

For Lemma~\ref{varyz0Q}, we start with $u_1,\ldots,u_k\in S^{n-1}$ that are not contained in a closed hemi-sphere.
For $z=(z_1,\ldots,z_k)\in\R_+^k$, we define
\begin{equation}
\label{P(z)Q}
P(z)=\{x\in\R^n:\,\langle x,u_i\rangle\leq z _i \mbox{ for }i=1,\ldots,k\}.
\end{equation}
We observe that $P(z)$ is an $n$-dimensional polytope with $o\in{\rm int}P(z)$, and any facet exterior unit normal is among $u_1,\ldots,u_k$. The following is the special case of polytopes of Theorem~6.2 in Lutwak, Yang, Zhang \cite{LYZ18}. For the sake of completeness, we provide the proof in this special case.

\begin{lemma}[Lutwak, Yang, Zhang \cite{LYZ18}]
\label{varyz0Q}
 If $q\neq 0$, $Q\in \mathcal{S}^n_{(o)}$, $\eta\in(0,1)$ and $z_t=(z_1(t),\ldots,z_k(t))\in\R_+^k$ for 
$t\in(-\eta,\eta)$ are such that
 $\lim_{t\to 0^+}\frac{z_i(t)-z(0)}{t}=z'_i(0)\in\R$ 
for $i=1,\ldots,k$ exists, then the $P(z_t)$ defined in \eqref{P(z)Q} satisfies that
$$
\lim_{t\to 0^{{+}}}\frac{\widetilde{V}_q(P(z_t),Q)-\widetilde{V}_q(P(z_0),Q)}{t}
=q\sum_{i=1}^k\frac{z'_i(0)}{h_{P(z_0)}(u_i)}\cdot \widetilde{C}_q(P(z_0),Q,\{u_i\}).
$$
 \end{lemma}
\begin{proof}
We set $P_0=P(z_0)$. We may assume that $F(P_0,u_i)$ is an $(n-1)$-dimensional facet of $P_0$ if and only if $i\leq l$ where $l\leq k$. 

For a {point} $x\in \R^n$ and affine $d$-plane $A$, $1\leq d\leq n-1$, we write $\delta(x,A)$ {for} the distance of $x$ from $A$.
For $i=1,\ldots,k$, let $H_i$ be the hyperplane $\{x\in\R^n:\, \langle u_i,x\rangle =z_i(0)\}$, and
for $i,j\in\{1,\ldots,k\}$ with $u_i\neq \pm u_j$, let $A_{ij}=H_i\cap H_j$, {which} is an affine
$(n-2)$-plane not containing the origin. Therefore ${\rm lin}A_{ij}$ is $(n-1)$-dimensional, and let $w_{ij}\in S^{n-1}$ be {orthogonal to ${\rm lin\,}A_{ij}$}. Choosing { the number $\Delta$ in such a way} that for any  $i,j\in\{1,\ldots,k\}$ with $u_i\neq \pm u_j$, 
we have $(1-\langle u_i,u_j\rangle^2)^{{-1/2}}\leq \Delta$, we deduce that
 if $s>0$ and $i,j\in\{1,\ldots,k\}$ with $u_i\neq \pm u_j$, then
\begin{equation}
\label{Aijdist}
y\in H_i\mbox{ and }d(y,H_j)\leq s \mbox{ yield }d(y,A_{ij})\leq \Delta s.
\end{equation}

Possibly decreasing $\eta>0$, we may assume that
there exist $r,R,Z>0$ such that if $t\in(-\eta,\eta)$, then
$$
\begin{array}{rcll}
rB^n\subset &P(z_t)&\subset RB^n,\quad \text{{and}}& \\
|z_i(t)-z_i(0)|&\leq & Z|t|&\mbox{ for $i=1,\ldots,k$.}
\end{array}
$$
If $u\in \tilde{\pi}(F(P(z_t),u_i)$ for {$i\in\{1,\ldots,l\}$} and $t\in(-\eta,\eta)$, then
$\langle \varrho_{P(z_t)}(u)\,u,u_i\rangle=z_i(t)\geq r $, therefore
\begin{equation}
\label{uuir}
{\langle u,u_i\rangle\geq 
\frac{r}{R}.}
\end{equation}
In addition, $\varrho_{P(z_t)}(u)\,u\in P(z_t)$, thus
\begin{equation}
\label{ziPzt}
\varrho_{P(z_t)}(u)\leq \frac{z_i(t)}{\langle u_i,u\rangle}.
\end{equation}
Now if $u\in \tilde{\pi}(F(P_0,u_i))$ for ${i\in\{1,\ldots,l\}}$
and $\varrho_{P(z_t)}(u)<\frac{z_i(t)}{\langle u_i,u\rangle}$, then
there exists $j\in\{1,\ldots,k\}$ with $u_j\neq \pm u_i$ satisfying
$\varrho_{P(z_t)}(u)\,u\in F(P(z_t),u_j)$, or in other words,
$$
\varrho_{P(z_t)}(u)=\frac{z_j(t)}{\langle u_j,u\rangle};
$$
and we claim that
\begin{eqnarray}
\label{nozitP1}
u\in&\Gamma(w_{ij}, c_1\cdot |t|),&\mbox{ where $c_1=\frac{\Delta RZ}{r^2}$}, \quad\text{{and}}\\
\label{nozitP2}
|\varrho_{P(z_t)}(u)-\varrho_{P_{{0}}}(u)|\leq&  c_2\cdot |t|, &\mbox{ where $c_2=\frac{R^2Z}{r^2}$}
\end{eqnarray}
On the one hand, \eqref{uuir} yields {that}
\begin{equation}
\label{rhouz0compare}
\left\|\varrho_{P_0}(u)\,u-\frac{z_i(t)}{\langle u_i,u\rangle}\,u\right\|=
\frac{|z_i(0)-z_i(t)|}{\langle u_i,u\rangle}\leq \frac{RZ}{r}\cdot |t|.
\end{equation}
On the other hand, since $\frac{z_i(t)}{\langle u_i,u\rangle}\,u\not\in P(z_t)$,
there exists $j\in\{1,\ldots,k\}$ with $u_j\neq \pm u_i$ such that
\begin{equation}
\label{rhouzijt}
\left\langle u_j, \frac{z_i(t)}{\langle u_i,u\rangle}\,u\right\rangle>z_j(t),
\end{equation}
and hence $u_j\neq \pm u_i$.
In turn it follows from \eqref{rhouz0compare} that
$$
d(\varrho_{P_0}(u)\,u,H_j)= z_j(t)-\langle u_j,\varrho_{P_0}(u)\,u\rangle <
\left\langle u_j, \frac{z_i(t)}{\langle u_i,u\rangle}\,u\right\rangle-\langle u_j,\varrho_{P_0}(u)\,u\rangle
\leq \frac{RZ}{r}\cdot |t|.
$$
We deduce from \eqref{Aijdist} that $d({\varrho_{P_0}(u)u},A_{ij})\leq \frac{\Delta RZ}{r}\cdot |t|$, and hence
$$
{|}\langle w_{ij},\varrho_{P_0}(u)\,u\rangle{|}\leq \frac{\Delta RZ}{r}\cdot |t|.
$$
Finally, $\varrho_{P_0}(u)\geq r$ yields \eqref{nozitP1}.

For \eqref{nozitP2}, we deduce from $\varrho_{P(z_t)}(u)=\frac{z_j(t)}{\langle u_j,u\rangle}$, 
\eqref{rhouz0compare} and \eqref{rhouzijt} that
$$
\langle u_j, \varrho_{P_0}(u)\,u\rangle>z_j(t)-\frac{RZ}{r}\cdot |t|=
\varrho_{P(z_t)}(u)\langle u_j, u\rangle-
\frac{RZ}{r}\cdot |t|.
$$
On the other hand, since $\varrho_{P_0}(u)\,u\in P_0$, we have
$$
\langle u_j, \varrho_{P_0}(u)\,u\rangle\leq z_j(0) \leq z_j(t)+Z|t|\leq \varrho_{P(z_t)}(u)\langle u_j, u\rangle+Z|t|,
$$
which in turn yields
$$
\langle u_j,u\rangle|\varrho_{P(z_t)}(u)-\varrho_{P_0}(u)|\leq \frac{2RZ}{r}\cdot |t|.
$$
Since $\langle u_j,u\rangle\geq \frac{r}{R}$ according to
\eqref{uuir} for $j$ instead of $i$, we conclude \eqref{nozitP2}.

For $i=1,\ldots,k$, we write $X_i$ to denote the set of all $u\in \tilde{\pi}(F(P_0,u_i))$ such that
$u\in\Gamma(w_{ij}, c_1\cdot |t|)$ for some 
 $j\in\{1,\ldots,k\}$ with $u_j\neq \pm u_i$. 
Using $\varrho_{P_0}(u)=\frac{z_i(0)}{\langle u_i,u\rangle}$
for $i=1,\ldots,l$ and $u\in \tilde{\pi}(F(P_0,u_i))$, \eqref{ziPzt} and \eqref{nozitP1}, 
it follows that $F(t)=\frac{\widetilde{V}_q(P(z_t),Q)-\widetilde{V}_q(P_0,Q)}{t}$ satisfies
\begin{eqnarray*}
F(t)&=&
\frac1n\sum_{i=1}^l\int_{\tilde{\pi}(F(P_0,u_i))}\frac{\varrho_{P(z_t)}(u)^q-\varrho_{P_0}(u)^q}t
\cdot \varrho_{Q}(u)^{n-q}\,d\HH^{n-1}(u)\\
&=&
\frac1n\sum_{i=1}^k\int_{X_i}\left(\frac{\varrho_{P(z_t)}(u)^q-\varrho_{P_0}(u)^q}t
+\frac{z_i(0)^q-z_i(t)^q}{\langle u,u_i\rangle^qt}\right)
\cdot \varrho_{Q}(u)^{n-q}\,d\HH^{n-1}(u)+\\
&&+\frac1n\sum_{i=1}^l\int_{\tilde{\pi}(F(P_0,u_i))}
\frac{z_i(t)^q-z_i(0)^q}{\langle u,u_i\rangle^qt}
\cdot \varrho_{Q}(u)^{n-q}\,d\HH^{n-1}(u)
\end{eqnarray*}
We deduce from \eqref{uuir}, \eqref{nozitP2} and $|z_i(t)-z_i(0)|\leq Z|t|$ that
$$
\frac{\varrho_{P(z_t)}(u)^q-\varrho_{P_0}(u)^q}t
+\frac{z_i(0)^q-z_i(t)^q}{\langle u,u_i\rangle^qt}
$$
{is} uniformly bounded on $\tilde{\pi}(F(P_0,u_i))$ as $t$ tends to $0$. Since $h_{P_0}(u_i)=z_i(0)$ for $i=1,\ldots,l$ and $\HH^{n-1}(X_i)=O(t)$
according to  Lemma~\ref{sphericalstripsest}, we deduce
$$
\lim_{t\to 0}F(t)=
\frac{q}n\sum_{i=1}^l\int_{\tilde{\pi}(F(P_0,u_i))}
\frac{z_i(0)^{q-1}z'_i(0)}{\langle u,u_i\rangle^qt}
\cdot \varrho_{Q}(u)^{n-q}\,d\HH^{n-1}(u)=q\sum_{i=1}^l\frac{z'_i(0)}{h_{P_0}(u_i)}\cdot \widetilde{C}_q(P_0,Q,\{u_i\}).
$$
As $\widetilde{C}_q(P_0,Q,\{u_i\})=0$ for $i>l$, we conclude Lemma~\ref{varyz0Q}.
\end{proof}

Now we sketch the necessary changes needed to
extend Theorem~\ref{polytopecor} to the case when $Q$ {is a} star body.

\begin{theorem}
\label{polytopecorQ}
Let $p>1$, $q>0$ and $Q\in \mathcal{S}^n_{(o)}$, and let $\mu$ be a discrete measure on $S^{n-1}$ that is not concentrated on any closed hemisphere. Then there exists a polytope $P\in\mathcal{K}^n_{(o)}$ such that 
$\widetilde{V}_{q}(P,Q)^{-1}\widetilde{C}_{p,q}(P,Q,\cdot)=\mu$.
\end{theorem}
{\bf Sketch of the proof {Theorem~\ref{polytopecorQ}.}}\,
Let $p>1$, $q>0$ and $\mu$ a discrete measure on $S^{n-1}$ that is not concentrated on any closed {hemisphere}.
Let ${\rm supp}\,\mu=\{u_1,\ldots,u_k\}$, and let $\mu(\{u_i\})=\alpha_i>0$, $i=1,\ldots,k$. For any 
$z=(t_1,\ldots,t_k)\in (\R_{\geq 0})^k$, we define
\begin{eqnarray}
\nonumber
\Phi(z)&=&\sum_{i=1}^k\alpha_it_i^p,\\
\label{P(z)}
P(z)&=&\{x\in\R^n:\,\langle x,u_i\rangle\leq t_i\,\forall i=1,\ldots,k\},\\
\nonumber
\Psi(z)&=&\widetilde V_q(P(z),Q).
\end{eqnarray}
Since $\alpha_i>0$ for $i=1,\ldots,k$, the set $Z=\{z\in (\R_{\geq 0})^k:\,\Phi(z)=1\}$ is compact, and hence Lemma~\ref{VqcontQ} yields the existence of 
$z_0\in Z$ such that
$$
\Psi(z_0)=\max\{\Psi(z):\,z\in Z\}.
$$
Now, similarly to the proof of Theorem~\ref{polytopecor},
only using Lemma~\ref{varyz0Q} in place of 
{Lemma~\ref{varyz0}}, 
we prove that $o\in {\rm int}\, P(z_0)$ and {that} there exists {a} $\lambda_0>0$ such that 
$$
\widetilde{V}_{q}(\lambda_0P(z_0),Q)^{-1}\widetilde{C}_{p,q}(\lambda_0P(z_0),Q,\cdot)=\mu.
$$
Therefore we can choose $P=\lambda_0P(z_0)$ in Theorem~\ref{polytopecorQ}.
\hfill $\Box$\\

\noindent{\bf Proof of Theorem~\ref{polytope}. }
Theorem~\ref{polytopecorQ} yields Theorem~\ref{polytope}
using the {same} argument {as the one} at the end of Section~\ref{secpolytopeBn}.
\hfill $\Box$\\

Next, we extend the results {of} Section~\ref{secLpdual} {on} the $L_p$ dual curvature measure{s}
to the case when {a star body $Q\in \mathcal{S}^n_{(o)}$ is involved.} The first {of these extensions} can be obtained as
Corollary~\ref{intgCqocor}.

\begin{coro}
\label{intgCqocorQ}
If $p>1$, $q>0$, $Q\in \mathcal{S}^n_{{(o)}}$, $K\in \mathcal{K}^n_o$ with ${\rm int}K\neq \emptyset$,
$\widetilde{C}_{p,q}(K,S^{n-1})<\infty$ and $\HH^{n-1}(\Xi_K)=0$, and the Borel function $g:\,S^{n-1}\to \R$ is bounded, then
$$
\int_{S^{n-1}}g(u)\,d\widetilde{C}_{p,q}(K,Q,u)=
\frac1n\int_{\partial' K} g(\nu_K(x))\langle \nu_K(x),x\rangle^{1-p}\|x\|_Q^{q-n}\,d\HH^{n-1}(x).
$$
\end{coro}

Lemmas~\ref{CpqVqrKQ} and \ref{p>q>0Q} can be proved 
{essentially the same way} as Lemmas~\ref{CpqVqrK} and \ref{p>q>0}.

\begin{lemma}
\label{CpqVqrKQ}
For $n\geq 2$, $p>1$, $q>0$ and $Q\in \mathcal{S}^n_{{(o)}}$, there exists constant $c>0$ depending
only on $p,q,n,Q$ such that
if $K\in \mathcal{K}^n_{(o)}$, then 
$$
\widetilde{C}_{p,q}(K,Q,S^{n-1})\geq c\cdot r(K)^{-p}\cdot \widetilde{V}_{q}(K,Q).
$$
 \end{lemma}

\begin{lemma}
\label{p>q>0Q}
If $p>1$, $0<q\leq p$, $Q\in \mathcal{S}^n_{{(o)}}$ and $K_m\in \mathcal{K}^n_{(o)}$ for $m\in\N$ tend to 
$K\in \mathcal{K}^n_o$ with ${\rm int}K\neq \emptyset$ such that
$\widetilde{C}_{p,q}(K_m,Q,S^{n-1})$ stays bounded, then $K\in \mathcal{K}^n_{(o)}$.
 \end{lemma} 

Since the sequence $\{\widetilde{C}_{p,q}(K_m,Q,S^{n-1})\}$ in Lemma~\ref{KmXiKQ} is bounded if and only if
$\{\widetilde{C}_{p,q}(K_m,S^{n-1})\}$ {is bounded},
Lemma~\ref{KmXiK} directly yields Lemma~\ref{KmXiKQ}.

\begin{lemma}
\label{KmXiKQ}
If $p>1$, $q>0$, $Q\in \mathcal{S}^n_{{(o)}}$ and $K_m\in \mathcal{K}^n_{(o)}$ for $m\in\N$ tend to 
$K\in \mathcal{K}^n_o$ with ${\rm int}K\neq \emptyset$ such that
$\widetilde{C}_{p,q}(K_m,Q,S^{n-1})$ stays bounded, then $\HH^{n-1}(\Xi_K)=0$.
 \end{lemma} 

Using Theorem~\ref{polytopecorQ}, Proposition~\ref{CqcontQ} and
Lemmas~\ref{CpqVqrKQ}, \ref{p>q>0Q} and \ref{KmXiKQ},
an argument similar to the one leading to Theorem~\ref{withvolume} implies

\begin{theorem}
\label{withvolumeQ}
For $p>1$, $q>0$, $Q\in{\mathcal{S}^n_{(o)}}$ and {a} finite Borel measure $\mu$ on $S^{n-1}$ not concentrated on a  closed {hemisphere}, there exists a convex body $K\in\mathcal{K}^n_o$ with ${\rm int}K\neq 0$ and $\HH^{n-1}(\Xi_K)=0$ such that
$$
 \widetilde{V}_q(K,Q)h_K^{p}d\mu= d\widetilde{C}_{q}(K,Q,\cdot),
$$
and, in addition, $K\in\mathcal{K}^n_{(o)}$ if $p\geq q$.
\end{theorem}

\noindent{\bf Proof of Theorem~\ref{main}. }
Theorem~\ref{withvolumeQ} yields Theorem~\ref{main}
using  {essentialy} the {same}  argument {as the one} at the end of Section~\ref{secmainBn}.
\hfill $\Box$

\section{The regularity of the solution}
\label{secregularity}

Given $p>1$, $q>0$, and a finite non-trivial Borel measure $\mu$ on $S^{n-1}$ not concentrated on any closed hemisphere, 
the $L_p$ dual Minkowski problem asks for a convex body $K\in\mathcal{K}_{o}^n$ such that
$\mathcal{H}^{n-1}(\Xi_K)=0$ and
\begin{equation}
\label{LpMinkowski}
h_K^{p-1}\,d\mu=d\widetilde{C}_q(K,\cdot).
\end{equation}
First we discuss why the condition $\mathcal{H}^{n-1}(\Xi_K)=0$ is natural.

\begin{example}
\label{expolytope}
For $p>1$ and $q>0$ with $p\neq q$, there {exists a} discrete measure $\mu$ on $S^{n-1}$ and polytopes $P_0$ and $P$ such that
$$
h_P^{p-1}d\mu=d\widetilde{C}_q(P,\cdot)\mbox{ \ and \ }h_{P_0}^{p-1}d\mu=d\widetilde{C}_q(P_0,\cdot)
$$
{with} $o\in{\rm int}P$ and $\HH^{n-1}(\Xi_{P_0})>0$.
\end{example}
\begin{proof}
Let $P_0$ be any polytope in $\R^n$ such that $u_0,\ldots,u_k$ denote the exterior unit nomals to its facets,
$h_{P_0}(u_0)=0$, $h_{P_0}(u_i)>0$ for $i=1,\ldots,k$, and no closed hemisphere contains $u_1,\ldots,u_k$.
Let ${\rm supp}\,\mu=\{u_1,\ldots,u_k\}$, and let $\mu(\{u_i\})=\widetilde{C}_{p,q}(P_0,\{u_i\})$ for $i=1,\ldots,k$.
According to Theorem~\ref{polytope}, there exists a polytope $P\in\mathcal{K}_{(o)}^n$ such that 
$\widetilde{C}_{p,q}(P,\cdot)=\mu$.
\end{proof}

We recall that according to Hug, Lutwak, Yang, Zhang \cite{HLYZ05}, if $p>1$ and $q=n$, then there is a unique solution $P$ to the $L_p$ dual Minkowski problem \eqref{LpMinkowski} for any 
measure $\mu$ on $S^{n-1}$ not concentrated on any closed hemisphere with $\mathcal{H}^{n-1}(\Xi_p)=0$; namely,
$P\in\mathcal{K}_{(o)}^n$.

We now turn to absolute continuous measures on $S^{n-1}$. 
We write $D$ and $D^2$ to denote {the} derivative and {the} Hessian {of} real functions on Euclidean spaces,
and $\nabla$ and  $\nabla^2$ to denote the gradient and the Hessian of 
real functions on $S^{n-1}$ with respect to a moving orthonormal frame on $S^{n-1}$.

First, let us discuss some relation between the support function and the boundary of a convex body.
Let $C\in\mathcal{K}_{(o)}^n$. If $y\in\R^n\backslash\{o\}$, then it is well-known (see Schneider \cite{Sch14})  that
the face with exterior normal $y$ is the set of derivatives of {the support functions} $h_C$ at $y$; namely,
\begin{equation}
\label{hderF}
F(C,y)=\partial h_C(y)=\{z\in\R^n : h_C(x)\geq h_C(y)+\langle z,x-y\rangle \text{ for each $x\in\R^n$}\}.
\end{equation}
We note {that} $h_C$ is differentiable at $\mathcal{H}^n$ almost all points of $\R^n$ being convex,
and $\mathcal{H}^{n-1}$ almost all points of $S^{n-1}$ being, in addition, $1$-homogeneous. 
It follows that whenever $h_{{C}}$ is differentiable at $u\in S^{n-1}$
(and hence for $\mathcal{H}^{n-1}$ almost every $u\in S^{n-1}$), we have
\begin{eqnarray}
\label{hder}
Dh_C(u)&=&x \mbox{ \ where $u$ is an exterior normal at $x\in{{\partial}}C$;}\\
\label{hderh}
\langle Dh_C(u),u\rangle&=&h_C(u).
\end{eqnarray}
In addition, \eqref{hderh} yields 
\begin{eqnarray}
\label{hdersum}
Dh_C(u)&=&\nabla h_K(u)+h_K(u)\,u, \quad \text{{and}}\\
\label{hderx}
\|x\|&=&\|Dh_C(u)\|=\sqrt{h(u)^2+\|\nabla h_C(u)\|^2}.
\end{eqnarray}

According to Corollary~\ref{SKCqQ}, if $q>0$ and $\HH^{n-1}(\Xi_K)=0$
for $K\in\mathcal{K}^n_{o}$, then
the surface area measure $S(K,\cdot)$ is absolutely continuous with respect to 
$\widetilde{C}_{q}(K,\cdot)$.
We deduce from Lemma~\ref{intgCqoQ}, \eqref{hder} and \eqref{hderx}
 that if 
$d\widetilde{C}_{p,q}(K,\cdot)=f\,d\mathcal{H}^{n-1}$ for a non-negative $L_1$ functon $f$ on $S^{n-1}$, then
the Monge-Amp\`ere equation for the $L_p$ dual curvature measure is
\begin{equation}
\label{Monge-Ampere0}
\det(\nabla^2 h+h\,{\rm Id})=\mbox{$\frac1n$}\,h^{p-1}
\left(\|\nabla h\|^2+h^2\right)^{\frac{n-q}2}\cdot f.
\end{equation}
In the case when a star body $Q\in\mathcal{S}^n_{(o)}$ is involved, we deduce from
Lemma~\ref{intgCqoQ}, \eqref{hder} and \eqref{hdersum} that
the Monge-Amp\`ere equation for the $L_p$ dual curvature measure is
\begin{equation}
\label{Monge-AmpereQ0}
\det(\nabla^2 h(u)+h(u)\,{\rm Id})=\mbox{$\frac1n$}\,h(u)^{p-1}
\left\|\nabla h_K(u)+h_K(u)\,u\right\|_Q^{n-q}\cdot f(u).
\end{equation}

{In} the rest of {this} section, we consider solutions to \eqref{Monge-Ampere0} in the case when there exist $c_2>c_1>0$
satisfying
\begin{equation}
\label{Monge-Amperec_1c_2}
c_1<f(u)<c_2\mbox{ \ for $u\in S^{n-1}$}.
\end{equation}

\begin{example}
\label{1<p<q}
Given $q>p>1$, there exists a $K\in\mathcal{K}_{o}^n$ such that ${\rm int}\,K\neq \emptyset$,
$o\in{{\partial}}K$ is not a smooth point,  $\Xi_K=\{o\}$
and $h_K$ satisfies both \eqref{Monge-Ampere0} and \eqref{Monge-Amperec_1c_2} in the sense of measure where actually $f$ is positive and continuous on $S^{n-1}$.
\end{example}
\begin{proof} For positive functions $g_1$ and $g_2$ on $B^{n-1}$, we write
$$
g_1\approx g_2 \mbox{ \ if $\alpha_1g_1(x)\leq g_2(x)\leq \alpha_2g_2(x)$ for $x\in B^{n-1}\backslash\{o\}$},
$$
where $\alpha_2>\alpha_1>0$ {are constants} depending only on $n,p,q$.

We define $g:\,\R^{n-1}\to\R$ by the formula
$$
g(x)=\|x\|+\|x\|^\theta \mbox{ \ for $\theta=q/p>1$},
$$
and we consider a convex body $K\in\mathcal{K}_{o}^n$ such that the graph of $g$ above $B^{n-1}$
 is part of ${\partial} K$ and
${\partial} K\backslash \{o\}$ is $C^2_+$. We observe that
$$
N(K,o)=\{(x,t):\,x\in\R^{n-1}\mbox{ and }t\leq -\|x\|\}.
$$

For $x\in B^n \backslash \{o\}$,
\begin{eqnarray*}
Dg(x)&=&x\left(\|x\|^{-1}+\theta\|x\|^{\theta-2}\right),\\
\|Dg(x)\|&=&1+\theta\|x\|^{\theta-1}\approx 1.
\end{eqnarray*}
 For $y=(x,g(x))\in {\partial} K$ { and $x\in B^{n-1}\setminus\{o\}$} we have
\begin{eqnarray*}
\nu_K(y)&=&(1+\|Dg(x)\|^2)^{-1}(Dg(x),-1),\\
\langle \nu_K(y),y\rangle &= &(1+\|Dg(x)\|^2)^{-1}\langle (Dg(x),-1),(x,g(x))\rangle=
(\theta-1)\|x\|^\theta\approx \|x\|^\theta,\\
\|y\|&=&\sqrt{\|x\|^2+(\|x\|+\|x\|^\theta)^2}=\|x\|\sqrt{2+2\|x\|^{\theta-1}+\|x\|^{2\theta-2}}
\approx \|x\|.
\end{eqnarray*}
At $x\in B^{n-1} \backslash \{o\}$, we have
$$
\det D^2g(x)=\theta(\theta-1)\|x\|^{\theta-2}(\|x\|^{-1}+\theta\|x\|^{\theta-2})^{n-2}\approx
\|x\|^{\theta-2}\|x\|^{-n+2}=\|x\|^{\theta-n}.
$$
Setting $u=\nu_K(y)$ and writing $\kappa(y)$ to denote the Gaussian curvature at $y$, we have
$$
\det\left(\nabla^2 h_K(u)+h_K(u)\,{\rm Id}\right)=\kappa(y)^{-1}=
\frac{(1+\|Dg(x)\|^2)^{\frac{1+n}2}}{\det D^2g(x)}\approx \|x\|^{n-\theta}.
$$

Let us consider the spherically open set
$$
\mathcal{U}=\{\nu_K(y):\, y=(x,g(x))\mbox{ and }x\in{\rm int}\,B^{n-1} \backslash \{o\}\}.
$$
Since ${\partial} K\backslash \{o\}$ is $C^2_+$, we deduce that there exists some continuous
function $f$ on $S^{n-1}\backslash N(K,o)$ such that
$d\widetilde{C}_{p,q}(K,\cdot)=f\,d\mathcal{H}^{n-1}$ on $S^{n-1}\backslash N(K,o)$.
It follows from \eqref{Monge-Ampere0} and the considerations above that if $u\in \mathcal{U}$
with $u=\nu_K(y)$ and $y=(x,g(x))$ for $x\in{\rm int}\,B^{n-1} \backslash \{o\}$, then
$\|y\|^2=\|\nabla h_K(u)\|^2+h_K(u)^2$ and
\begin{eqnarray}
\nonumber
f(u)&=&n\det\left(\nabla^2 h_K(u)+h_K(u)\,{\rm Id}\right)
h_K(u)^{1-p}
\left(\|\nabla h_K(u)\|^2+h_K(u)^2\right)^{\frac{q-n}2}\\
\label{formulax}
&=&\frac{(1+\|Dg(x)\|^2)^{\frac{1+n}2}}{\det D^2g(x)}\left[(\theta-1)\|x\|^\theta\right]^{1-p}
\left(\|x\|^2+(\|x\|+\|x\|^\theta)^2\right)^{\frac{q-n}2}\\
\label{fformulax}
&\approx&\|x\|^{n-\theta}\|x\|^{\theta(1-p)}\|x\|^{q-n}=\|x\|^{q-\theta p}=1.
\end{eqnarray}
The expression \eqref{formulax} has some limit $F>0$ as $x\in B^{n-1} \backslash \{o\}$ tends to $o$
according to the formulas above, therefore defining $f(u)=F$ for $u\in N(K,o)\cap S^{n-1}$, 
\eqref{fformulax}  yields that $f$
is a positive continuous function on $S^{n-1}$ satisfying \eqref{Monge-Ampere0} in the sense of measure.
\end{proof}

Let us {recall} some fundamental properties of Monge-Amp\`ere equations based on
 the survey Trudinger, Wang~\cite{TrW08}.
Given a convex function $v$ defined in an open convex set $\Omega$ of $\R^{{n}}$, $D v$ and $D^2 v$ denote its gradient and its Hessian, respectively. When $v$ is a convex function defined in an open convex set $\Omega\subset\R^{{n}}$, the subgradient $\partial v(x)$ of $v$ at $x\in\Omega$ is defined as
\[
\partial v(x) =\{z\in\R^{{n}} : v(y)\geq v(x)+\langle z,y-x\rangle \text{ for each $y\in\Omega$}\},
\]
which is a compact convex set.
If $\omega\subset\Omega$ is a Borel set, then we denote by $N_v(\omega)$ the image of $\omega$ {via} the gradient map of $v$, i.e.
\[
N_v(\omega)=\bigcup_{x\in\omega}\partial v (x). 
\]
The associated Monge-Amp\`ere measure is defined by
\begin{equation}\label{Monge-Ampere-measure}
\mu_v(\omega)=\mathcal{H}^n\Big(N_v\big(\omega\big)\Big).
\end{equation}
We observe that if $v$ is $C^2$, then
$$
\mu_v(\omega)=\int_\omega \det( D^2 v ).
$$
We say that {a} convex {function} $v$ is the solution of a Monge-Amp\`ere equation in the sense of measure (or in the Alexandrov sense), if it solves the corresponding integral formula for $\mu_v$ instead of 
the original formula for $\det( D^2 v )$.

If $K$ is any convex body in $\R^n$, then
\begin{equation}
\label{duality_body_support2}
\partial h_K(u)=F(K,u),
\end{equation}
{see Schneider~\cite[Thm.~1.7.4]{Sch14}.}
In particular, for any Borel $\omega\subset S^{n-1}$, the surface area measure $S_K$ satisfies
$$
S_K(\omega)=\mathcal{H}^{n-1}\big(\cup_{u\in\omega} F(K,u)\big)
=\mathcal{H}^{n-1}\big(\cup_{u\in\omega} \partial h_K(u)\big),
$$
and hence $S_K$ is the analogue of the Monge-Amp\`ere measure for the restriction of $h_K$ to $S^{n-1}$.

We use Lemma~\ref{MongeAmpereRn-lemma}  to transfer the $L_p$ dual Minkowski Monge-Amp\`ere equation (\ref{Monge-Ampere0})
on $S^{n-1}$ to a Euclidean Monge-Amp\`ere equation on $\R^{n-1}$.
For $e\in S^{n-1}$, we consider the restriction of a solution $h$ of \eqref{Monge-Ampere0} to the  hyperplane tangent to $S^{n-1}$ at  $e$. 

\begin{lemma}
\label{MongeAmpereRn-lemma}
For $p>1$, $q>0$, $Q\in\mathcal{S}_{{(o)}}^n$, $e\in S^{n-1}$ and $K\in\mathcal{K}_o^n$ with
 $\mathcal{H}^{n-1}(\Xi_K)=0$,
if $h=h_K$ is a solution of \eqref{Monge-AmpereQ0} for {a} non-negative {function} $f$,  and $v(y)=h_{{K}}(y+e)$ holds for
 $v:\,e^\bot\to \R$, then 
$v$ satisfies
\begin{equation}
\label{MongeAmpereRn}
\det( D^2 v(y) )=v(y)^{p-1}\left\|Dv(y)+\left(\langle Dv(y),y\rangle-v(y)\right)\cdot e\right\|_Q^{n-q}\,g(y) \quad 
\text{ on $e^\bot$}
\end{equation}
in the sense of measure, where  for $y\in  e^\bot$, we have
\[
g(y)=\left(1+\|y\|^2\right)^{-\frac{n+p}2} f\left(\frac{e+y}{\sqrt{1+\|y\|^2}}\right).
\]
\end{lemma}
\noindent{\bf Remark. } 
$\|Dv(y)+(\langle Dv(y),y\rangle-v(y))\cdot e\|_{{Q}}^{n-q}=
(\|Dv(y)\|^2+(\langle Dv(y),y\rangle-v(y))^2)^{\frac{n-q}2}$ if $Q=B^n$.
\begin{proof} 
Thus using to Corollary~\ref{intgCqocorQ} and \eqref{hderh}, the Monge-Amp\`ere equation for $h_K$ can be written in the form
\begin{equation}
\label{hKnozero}
 dS_K=h_K^{p-1}\|Dh_K\|_Q^{q-n}f\,d \mathcal{H}^{n-1}\mbox{ \ \ on $S^{n-1}$}.
\end{equation}

We consider $\pi:e^\bot\to S^{n-1}$ defined by
\begin{equation}
\label{pidef}
\pi(y)=(1+\|y\|^2)^{\frac{-1}2}(y+e),
\end{equation}
which is induced by the radial projection from the tangent hyperplane $e+e^\bot$ to $S^{n-1}$. 
Since $\langle \pi(x),e\rangle=(1+\|x\|^2)^{\frac{-1}2}$, the Jacobian of $\pi$ is
\begin{equation}
\label{Dpi}
\det D\pi(y)=(1+\|y\|^2)^{\frac{-n}2}.
\end{equation}

For $y\in e^\bot$, (\ref{duality_body_support2}) and writing $h_K$ in terms of an orthonormal basis of $\R^n$ containing $e$,  
yield that $v$ satisfies
\begin{equation}
\label{partialvF}
\partial v(y)=\partial h_K(y+e)|e^\bot=F(K,y+e)|e^\bot=F(K,\pi(y))|e^\bot.
\end{equation}

Let $u=\pi(y)$ for some $y\in e^\bot$, where $v$ is differentiable. As $h_K$ is homogeneous {of degree one}, we have $Dh_K(y+e)=Dh_K(u)$, {and} therefore
$$
Dv(y)=Dh_K(y+e)|e^\bot=Dh_K(u)|e^\bot,
$$
and hence $Dh_K(u)=Dv(y)-te$ for some $t\in \R$. 
Now 
 $\langle Dh_K(u),u\rangle=h_K(u)$ according to \eqref{hderh}, which in turn yields 
by $u=(1+\|y\|^2)^{-1}(y+e)$ and $h_K(u)=(1+\|x\|^2)^{-1}v(y)$ that
$t=\langle Dv(y),y\rangle-v(y)$. In other words,
if $v$ is differentiable at  $y\in e^\bot$ and $u=\pi(y)$, then
\begin{equation} 
\label{DhDv}
Dh_{{K}}(u)=Dh_{{K}}(e+y)=Dv(y)+\left(\langle Dv(y),y\rangle-v(y)\right)\cdot e.
\end{equation}

For a bounded Borel set $\omega\subset e^\bot$, we have
using \eqref{DhDv} that
\begin{eqnarray*}
\mathcal{H}^{n-1}(N_v(\omega))&=&
\mathcal{H}^{n-1}\left(\cup_{y\in\omega}\partial v(y)\right)\\
&=&\mathcal{H}^{n-1}\left(\cup_{u\in\pi(\omega)}\left(F(K,u)|e^\bot\right)\right)
=\int_{\pi(\omega)}\langle u,e\rangle\,dS_K(u)\\
&=&\int_{\pi(\omega)}\langle u,e\rangle h_K^{p-1}(u)\|Dh_K(u)\|^{q-n}f(u)\,d \mathcal{H}^{n-1}(u)\\
&=&\int_\omega(1+\|y\|^2)^{\frac{-n-p}2} v(y)^{p-1}
\left\|Dv(y)+\left(\langle Dv(y),y\rangle-v(y)\right)\cdot e\right\|^{n-q}
f(\pi(y))\,d \mathcal{H}^{n-1}(y)
\end{eqnarray*}
where we used {in} the last step that
$$
v(y)=h_K(y+e)=(1+\|y\|^2)^{\frac{1}2}h_K(\pi(y)).
$$
In particular, $v$ satisfies the Monge-Amp\`ere type differential equation
$$
\det D^2v(y)=(1+\|y\|^2)^{-\frac{n+p}2} v(y)^{p-1}  \left\|Dv(y)+\left(\langle Dv(y),y\rangle-v(y)\right)\cdot e\right\|^{n-q}  f(\pi(y))
\mbox{ \ \ on $e^\bot$},
$$
or in other words, $v$ satisfies (\ref{MongeAmpereRn}) on $e^\bot$.
\end{proof}

The following results by Caffarelli (see Theorem~1 and Corollary~1 in \cite{Caf90a} for (i) and (ii), and 
for (iii)) are the core of the discussion of the part of the boundary 
of a convex body $K$, where the support function at some normal vector is positive.

\begin{theorem}[Caffarelli]
\label{Caffarelli-smooth}
Let $\lambda_2>\lambda_1>0$,  and let $v$ be a convex function on an open bounded convex set $\Omega\subset \R^{{n}}$ such that 
\begin{equation}\label{aggiunta}
\lambda_1\leq \det D^2v\leq \lambda_2
\end{equation}
in the sense of measure.
\begin{description}
\item[(i)] If $v$ is non-negative and
$S=\{y\in\Omega:\,v(y)=0\}$ is not a point, then $S$ has no extremal point in $\Omega$.
\item[(ii)] If $v$ is strictly convex, then $v$ is $C^1$.
\end{description}
\end{theorem}
We recall that \eqref{aggiunta} is equivalent to saying  that for each bounded Borel set $\omega\subset \Omega$, we have
$$
 \lambda_1\mathcal{H}^{{n}}(\omega)\leq \mu_v(\omega)\leq \lambda_2\mathcal{H}^{{n}}(\omega),
$$
where $\mu_v$ has been defined in~\eqref{Monge-Ampere-measure}.

Caffarelli \cite{Caf90b} strengthened Theorem~\ref{Caffarelli-smooth} to have some estimates on H\"older continuity
 under some additional {assumptions} on $v$.

\begin{theorem}[Caffarelli]
\label{Caffarelli-smooth-alpha}
For real functions $v$ and $f$ on an open bounded convex set $\Omega\subset \R^{{n}}$, let $v$ be strictly convex, and let $f$ be positive and continuous
such that
\begin{equation}\label{aggiuntaf}
 \det D^2v=f
\end{equation}
in the sense of measure.
\begin{description}
\item[(i)] Each $z\in\Omega$ has an open ball $B\subset\Omega$ around $z$ 
such that the restriction of $v$ to $B$ is in
$C^{1,\alpha}(B)$ for any $\alpha\in(0,1)$.
\item[(ii)] If $f$ is in $C^\alpha(\Omega)$ for some $\alpha\in(0,1)$, then 
each $z\in\Omega$ has an open ball $B\subset\Omega$ around $z$ 
such that the restriction of $v$ to $B$ is in
$C^{2,\alpha}(B)$.
\end{description}
\end{theorem}
\begin{proof}
For (i), what actually is the direct consequence of Caffarelli \cite{Caf90b} Theorem~1 is that
 if $v$ is strictly convex, and $f$ is positive and continuous, then each $z\in\Omega$ has an open ball $B\subset\Omega$ around $z$ 
such that the restriction of $v$ to $B$ is 
in the Sobolev space $W^{2,l}(B)$ for any $l>1$. However, the Sobolev Embedding Theorem (see Demengel, Demengel \cite{DeD12}) yields that
if $l>n$ is chosen in a way such that $\frac{n}l=1-\alpha$, then $W^{2,l}(B)\subset C^{1,\alpha}(B)$.

Finally, (ii) is just Theorem~2 of Caffarelli \cite{Caf90b}.
\end{proof}

We will use, in the rest of the section, that there exists {an} $\omega\in(0,1)$ depending on $Q$ such that
\begin{equation}
\label{Qnormest}
\omega\|x\|\leq \|x\|_Q\leq \omega^{-1}\|x\|\mbox{ \ for $x\in\R^n$.}
\end{equation}

\noindent{\bf Proof of Theorem~\ref{regularityaway0}. } We recall that
for some $c_2>c_1>0$, we have
$$
c_1<f(u)<c_2\mbox{ \ for $u\in S^{n-1}$}
$$
in \eqref{Monge-Ampere0}.

First, we analyze
Lemma~\ref{MongeAmpereRn-lemma} for an $e\in S^{n-1}\backslash N(K,{o})$.
Since $h_K$ is continuous, there exist $\alpha(e)\in(0,1)$ and 
$\delta(e)>0$ depending on $e$ and $K$ such that $h_K(u)\geq\delta(e)$  for 
$u\in{\rm cl}\Omega(e,\alpha(e))$, and hence 
${\rm cl}\Omega(e,\alpha(e))\cap N(K,o)=\emptyset$.
 It also follows that there exists $\xi(e)\in(0,1)$ depending on $e$ and $K$ such that
if some $u\in{\rm cl}\Omega(e,\alpha(e))$ is the exterior normal at an $x\in{\partial} K$, then
$\xi(e)\leq \|x\|\leq 1/\xi(e)$. 
Let us consider the open $(n-1)$-ball $\Omega_e=\pi^{-1}(\Omega(e,\alpha(e)))$ for the $\pi$ defined in \eqref{pidef}, and let
$v$ be the function of Lemma~\ref{MongeAmpereRn-lemma} on $e^\bot$.
We deduce from 
\eqref{hder}, \eqref{DhDv} and \eqref{Qnormest} that there exists $\tilde{\xi}(e)\in(0,1)$ 
depending on $e$ and $K$ such that
if $v$ is differentiable at $y\in\Omega_e$, then
\begin{equation}
\label{uatxlowupp}
\tilde{\xi}(e)\leq \|Dv(y)+\left(\langle Dv(y),y\rangle-v(y)\right)\cdot e\|_Q\leq 1/\tilde{\xi}(e).
\end{equation}

Since $v$ is bounded on ${\rm cl}\Omega_e$ with an upper bound depending on $e$ and $K$ and
$v(y)=h_K(y+e)\geq \delta(e)$ for $y\in{\rm cl}\Omega_e$, 
we deduce from  \eqref{uatxlowupp} and Lemma~\ref{MongeAmpereRn-lemma}
that there exist $\lambda_2(e)>\lambda_1(e)>0$ 
depending on $e$ and $K$ such that
\begin{equation}
\label{KminusXismoothvest}
\lambda_1(e)\leq \det D^2v(y)\leq \lambda_2(e)\mbox{ \ for }y\in \Omega_e.
\end{equation}

In order to prove that $K\backslash\Xi_K$ is $C^1$,
we claim that for any $z\in {\partial}K$,
\begin{equation}
\label{KminusXismooth}
{\rm dim}N(K,z)\geq 2\mbox{ \ yields \ }N(K,z)\subset N(K,o).
\end{equation}
{We assume, on the contrary} that  there exists $z\in {\partial}K$ such that
$$
{\rm dim}N(K,z)\geq 2\mbox{ \ and \ }N(K,z)\not\subset N(K,o),
$$
and seek a contradiction. It follows that there exists an extremal vector $e$ of $N(K,z)\cap S^{n-1}$ such that
$h_K(e)>0$.

We observe that $v(y)=h_K(y+e)\geq\langle y+e,z\rangle$ for $y\in \Omega$ with equality if and only if 
$y\in S=\pi^{-1}(N(K,z)\cap \Omega(e,\alpha(e)){)}$, therefore the first degree polynomial $l(y)=\langle y+e,z\rangle$
satisfies
$$
v(y)-l(y)
\left\{
\begin{array}{rcl}
=&0&\mbox{ if }y\in S\\
>&0&\mbox{ if }y\in \Omega\backslash S.
\end{array}\right.
$$
We have $\lambda_1(e)\leq \det D^2(v-l)\leq \lambda_2(e)$ on $\Omega$ by \eqref{KminusXismoothvest}.
Since ${\rm dim} S\geq 1$ for $S=\{y\in e^\bot:\,v(y)-l(y)=0\}$ and $o$ is an {extremal} point of $S$ by the choice of $e$, we have contradicted Caffarelli's Theorem~\ref{Caffarelli-smooth} (i), completing the proof
of \eqref{KminusXismooth}.

In turn, \eqref{KminusXismooth} yields that
\begin{equation}
\label{outXiKsmooth}
\begin{array}{rcl}
{\partial}K\backslash \Xi_K&=&\{z\in {\partial}K:\,h_K(u)>0\mbox{ at some }u\in N(K,z)\},\quad {and}\\
{\partial}K\backslash \Xi_K &\mbox{ is }&C^1.
\end{array}
\end{equation}

Next we prove that $v$ is strictly convex on ${\rm cl}\Omega_e$ for $e\in S^{n-1}\backslash N(K,{o})$; or equivalently,
\begin{equation}
\label{vstrictconvex}
v\left(\frac{y_1+y_2}2\right)<\frac{v(y_1)+v(y_2)}2\mbox{ \ for $y_1,y_2\in\Omega_e$ with $y_1\neq y_2$.}
\end{equation}
Let $e+\frac12(y_1+y_2)$ be an exterior normal at $z\in {\partial}K$.
Since $\Omega_e{\cap} N(K,o)=\emptyset$, it follows from {\eqref{outXiKsmooth}} that
$z\in {\partial}K\backslash \Xi_K$ and $z$ is a smooth point. For $i=1,2$,
$e+y_i$ and $e+\frac12(y_1+y_2)$ are independent, therefore
$$
v(y_i)=h_K(e+y_i)>\langle z, e+y_i\rangle.
$$
We conclude that
$$
\frac{v(y_1)+v(y_2)}2>\left\langle z, e+\frac{y_1+y_2}2\right\rangle=
h_K\left(e+\frac{y_1+y_2}2\right)=v\left(\frac{y_1+y_2}2\right),
$$
{proving} \eqref{vstrictconvex}.

We deduce from \eqref{KminusXismoothvest}, the strict convexity \eqref{vstrictconvex} of $v$,
and {from} Caffarelli's Theorem~\ref{Caffarelli-smooth} (ii) that $v$ is $C^1$ on ${\rm cl}\Omega_e$ for any
$e\in S^{n-1}\backslash N(K,{o})$. We deduce that $h_K$ is $C^1$ on $\R^n\backslash N(K,{o})$,
and hence ${\partial}K\backslash {\Xi}_K$ contains no segment, completing the proof of 
Theorem~\ref{regularityaway0} (i).

Next, we turn to Theorem~\ref{regularityaway0} (ii) and (iii), and hence {we assume that} $f$ is continuous. Let $e\in S^{n-1}\backslash N(K,{o})$, and we apply
again Lemma~\ref{MongeAmpereRn-lemma} for this $e$. Since $v$ is $C^1$ on ${\rm cl}\Omega_e$,
we deduce that the right hand side of \eqref{MongeAmpereRn} is {continuous}.
As $v$ is strictly convex on ${\rm cl}\Omega_e$ by \eqref{vstrictconvex},
Theorem~\ref{Caffarelli-smooth-alpha} (i) applies, and hence
there exists an open ball $B$ of $e^\bot$ centered at $o$ such that $v$ is $C^{1,\alpha}$ on $B$ for any $\alpha\in(0,1)$.
In turn, we deduce Theorem~\ref{regularityaway0} (ii).

Finally, let us assume that $f$ is $C^\alpha$ on $S^{n-1}$. As $v$ is $C^{1,\alpha}$ on $B$, it follows that
 the right hand side of \eqref{MongeAmpereRn} is $C^\alpha$ on $B$, as well. 
Therefore Theorem~\ref{Caffarelli-smooth-alpha} (ii) yields that
$v$ is $C^{2,\alpha}$ on an open ball $B_0\subset B$ of $e^\bot$ centered at $o$.
We deduce from \eqref{KminusXismoothvest} that $\det D^2v(0)>0$,
 concluding the proof of Theorem~\ref{regularityaway0}.
\hfill $\Box$\\

Next, we discuss how the ideas leading to Theorem~\ref{regularityaway0} work for any
$p,q\in\R$ provided that $o\in{\rm int}K$. First of all, the following is the version of
Lemma~\ref{MongeAmpereRn-lemma} {for the case when} $K\in\mathcal{K}_{(o)}^n$, which can be proved just like
Lemma~\ref{MongeAmpereRn-lemma}.

\begin{lemma}
\label{MongeAmpereRn-lemmaoin}
For $p,q\in\R$, $Q\in\mathcal{S}_{{(o)}}^n$, $e\in S^{n-1}$ and $K\in\mathcal{K}_{(o)}^n$,
if $h=h_K$ is a solution of \eqref{Monge-AmpereQ0} for {a} non-negative {function} $f$,  and $v(y)=h(y+e)$ holds for 
$v:\,e^\bot\to \R$, then 
$v$ satisfies
$$
\det( D^2 v(y) )=v(y)^{p-1}\left\|Dv(y)+\left(\langle Dv(y),y\rangle-v(y)\right)\cdot e\right\|_Q^{n-q}\,g(y) \quad 
\text{ on $e^\bot$}
$$
in the sense of measure, where  for $y\in  e^\bot$, we have {that}
$$
g(y)=\left(1+\|y\|^2\right)^{-\frac{n+p}2} f\left(\frac{e+y}{\sqrt{1+\|y\|^2}}\right).
$$
\end{lemma}

\noindent{\bf Proof of Theorem~\ref{regularityaway0allpq}. }
Instead of Corollary~\ref{intgCqocorQ}, we use that according to
Lemma~5.1 in   Lutwak, Yang and Zhang \cite{LYZ18}, {which states that}
$$
\int_{S^{n-1}}g(u)\,d\widetilde{C}_{p,q}(K,Q,u)=
\frac1n\int_{\partial' K} g(\nu_K(x))\langle \nu_K(x),x\rangle^{1-p}\|x\|_Q^{q-n}\,d\HH^{n-1}(x)
$$
for any bounded Borel function $g:\,S^{n-1}\to \R$.
Therefore using \eqref{hder}, the Monge-Amp\`ere equation for $h_K$ can be written in the form
$$
 dS_K=h_K^{p-1}\|Dh_K\|_Q^{q-n}f\,d \mathcal{H}^{n-1}\mbox{ \ \ on $S^{n-1}$}.
$$
Now the same argument as for Theorem~\ref{regularityaway0} yields 
Theorem~\ref{regularityaway0allpq} (i), and the versions of
Theorem~\ref{regularityaway0allpq} (ii) and (iii), where $h_K$ is locally $C^{1,\alpha}$
on $S^{n-1}$ in Theorem~\ref{regularityaway0allpq} (ii), and $h_K$ is locally $C^{2,\alpha}$
on $S^{n-1}$ in Theorem~\ref{regularityaway0allpq} (iii). However, $S^{n-1}$ is compact, therefore
$h_K$ is globally $C^{1,\alpha}$
on $S^{n-1}$ in Theorem~\ref{regularityaway0allpq} (ii), and $h_K$ is globally $C^{2,\alpha}$
on $S^{n-1}$ in Theorem~\ref{regularityaway0allpq} (iii).
\hfill $\Box$\\

Finally, we start our preparations for proving Theorem~\ref{regularitystrictconvexity}. The following
Lemma~\ref{MongeAmperepointinside} is essentially proved in Lemma~3.2 and Lemma~3.3  in \cite{TrW08} (see the remarks after the Lemma~\ref{MongeAmperepointinside}).

\begin{lemma}
\label{MongeAmperepointinside}
Let $v$ be a convex function defined on the closure of an open bounded convex set $\Omega\subset \R^{{n}}$ such that the Monge-Amp\`ere measure $\mu_v$ is finite on $\Omega$ and $v\equiv 0$ on $\partial \Omega$, and let 
$z_0+tE\subset\Omega\subset z_0+E$ for $t>0$ and $z_0\in\Omega$ and an origin centered ellipsoid $E$.
\begin{description}
\item[(i)] If $z\in\Omega$ satisfies $(z+s\, E)\cap\partial\Omega\neq \emptyset$ for $s>0$, then
$$
|v(z)|\leq s^{1/n}\tau_0\mathcal{H}^d(\Omega)^{1/n}\mu_v(\Omega)^{1/{n}}
$$
for some $\tau_0>0$ depending on ${n},t$.
\item[(ii)] If  $\mu_v(z_0+tE)\geq b\,\mu_v(\Omega)$ for $b>0$, then
\begin{equation}
\label{formula_lemmaTW_2}
|v(z_0)|\geq \tau_1\mathcal{H}^{{n}}(\Omega)^{1/{n}}\mu_v(\Omega)^{1/n}
\end{equation}
for some $\tau_1>0$ depending on ${n}$, $t$ and $b$.
\end{description}
\end{lemma}

{We remark that Lemma~3.2 in~\cite{TrW08} proves~\eqref{formula_lemmaTW_2} with $\sup_\Omega|v|$ instead of $|v(z_0)|$. The inequality~\eqref{formula_lemmaTW_2} follows from that and the claim  that  if $\Omega$ is an open bounded convex set in $\R^{{n}}$ and $v$ is a  convex function on ${\rm cl}\Omega$, {that} vanishes on $\partial\Omega$, and  $z_0+tE\subset \Omega\subset z_0+E$ for an origin centered ellipsoid $E$, then 
	\begin{equation}
	\label{supval}
	|v(z_0)|\geq t/(t+1) \sup_\Omega |v|.
	\end{equation}
	To prove \eqref{supval}, we note that $v$ is non-positive, and choose $z_1\in {\rm cl}\Omega$ where $v$ attains its minimum.
	Since $z_2=z_0-t(z_1-z_0)\in {\rm cl}\Omega$ and $z_0=\frac{t}{1+t}\,z_1+\frac{1}{1+t}\,z_2$, we have
	$$
	v(z_0)\leq \frac{t}{1+t}\,v(z_1)+\frac1{1+t}\,v(z_2)\leq \frac{t}{1+t}\,v(z_1),
	$$
	verifying \eqref{supval}.}

Now we show that Theorem~\ref{regularitystrictconvexity}
is invariant under volume preserving linear transformations.

\begin{lemma}
\label{strictconvexityslnRinvariant}
Let $1<p<q$,  $Q\in\mathcal{S}_{(o)}^n$, $\varphi\in {\rm SL}(n,\R)$
and let $K\in\mathcal{K}^n_{o}$ 
with  $\HH^{n-1}(\Xi_K)=0$ and ${\rm int}K\neq\emptyset$. If 
$$
d\widetilde{C}_{q}(K,Q,\cdot)=h_K^{p}f\,d\HH^{n-1}
$$
for $c_2>c_1>0$ and for a Borel function $f$ on $S^{n-1}$ satisfying $c_1\leq f\leq c_2$, then 
$$
d\widetilde{C}_{q}(\varphi K,\cdot)=h_{\varphi K}^{p}\tilde{f}\,d\HH^{n-1}
$$
for $\tilde{c}_2>\tilde{c}_1>0$ and for a Borel function $\tilde{f}$ on $S^{n-1}$ satisfying 
$\tilde{c}_1\leq \tilde{f}\leq \tilde{c}_2$ where $\HH^{n-1}(\Xi_{\varphi K})=0$.
\end{lemma}
\begin{proof} Since $\varphi$ is Lipschitz, we deduce that  $\HH^{n-1}(\Xi_{\varphi K})=0$.

As a first step to analyze the density function of 
$\widetilde{C}_{q}(\varphi K,\cdot)$ with respect to $\HH^{n-1}$, we prove that
\begin{equation}
\label{strictconvexityslnRinvariant*}
d\widetilde{C}_{q}(\varphi K,\varphi Q,\cdot)=h_{\varphi K}^{p}f^*\,d\HH^{n-1}
\end{equation}
for $c^*_2>c^*_1>0$ and for a Borel function $f^*$ on $S^{n-1}$ satisfying 
$c^*_1\leq f^*\leq c^*_2$. For $\eta\subset S^{n-1}$, $\mathbf{1}_\eta$ denotes the characteristic function of $\eta$. 
We note that \eqref{strictconvexityslnRinvariant*} is equivalent to prove that
if $\eta\subset S^{n-1}\backslash N(\varphi K,o)$ is Borel,  then
\begin{equation}
\label{strictconvexityslnRinvariant0}
c^*_1\HH^{n-1}(\eta)\leq 
\widetilde{C}_{p,q}(\varphi K,\varphi Q,\eta)=
\int_{S^{n-1}} \mathbf{1}_\eta h_{\varphi K}^{-p}d\widetilde{C}_{q}(\varphi K,\varphi Q,\cdot)
\leq c^*_2\HH^{n-1}(\eta).
\end{equation}

We consider the $C^1$ diffeomorphism $\tilde{\varphi}:\,S^{n-1}\to S^{n-1}$ defined by
$$
\tilde{\varphi}(u)=\frac{\varphi^{t} u}{\|\varphi^{t} u\|},
$$
which satisfies that if $\eta\subset S^{n-1}$, then
$$
\mathbf{1}_{\tilde{\varphi}\eta}(u)=\mathbf{1}_\eta\left(\frac{\varphi^{-t} u}{\|\varphi^{-t} u\|}\right)
$$
There exist $\aleph_1,\aleph_2\in(0,1)$ depending on $\varphi$ such that
\begin{equation}
\label{phiestueta}
\begin{array}{rcccll}
\aleph_1\|u\|&\leq&\|\varphi^{-t}(u)\|&\leq &\aleph_1^{-1}\|u\|&\mbox{ for $u\in S^{n-1}$;}\\[1ex]
\aleph_2\HH^{n-1}(\eta)&\leq&\HH^{n-1}(\tilde{\varphi}(\eta))&\leq& 
\aleph_2^{-1}\HH^{n-1}(\eta)&\mbox{ for any Borel set $\eta\subset S^{n-1}$.}
\end{array}
\end{equation}
We also note if $u\in S^{n-1}$ is an exterior normal at $z\in{\partial}K$, then
$\varphi^{-t} u$ is an exterior normal at $\varphi z\in{\partial}\varphi K$, and hence
$$
h_{\varphi K}(\varphi^{-t} u)=\langle \varphi^{-t} u,\varphi z\rangle=h_K(u).
$$
It also follows that $\varphi^{t}N(\varphi K,o)=N(K,{o})$.
Therefore, if $\eta\subset S^{n-1}\backslash N(\varphi K,o)$ is Borel,  then 
$\tilde{\varphi}\eta\subset S^{n-1}\backslash N(K,o)$, and
we deduce from 
Lemma~\ref{CqslnQ} that
\begin{eqnarray*}
\widetilde{C}_{p,q}(\varphi K,\varphi Q,\eta)&=&
\int_{S^{n-1}} \mathbf{1}_\eta h_{\varphi K}^{-p}d\widetilde{C}_{q}(\varphi K,\varphi Q,\cdot)\\
&=&\int_{S^{n-1}} \mathbf{1}_\eta\left(\frac{\varphi^{-t} u}{\|\varphi^{-t} u\|}\right)
h_{\varphi K}\left(\frac{\varphi^{-t} u}{\|\varphi^{-t} u\|}\right)^{-p}
d\widetilde{C}_{q}( K,Q,u)\\
&=&\int_{S^{n-1}}\mathbf{1}_{\tilde{\varphi}\eta}(u) \|\varphi^{-t} u\|^p
h_{K}(u)^{-p}
d\widetilde{C}_{q}( K,Q,u).
\end{eqnarray*}
We deduce from \eqref{phiestueta} and the condition on 
$\widetilde{C}_{p,q}( K,Q,\cdot )$ that
$$
c_1\aleph_1^p\HH^{n-1}(\tilde{\varphi}\eta)\leq 
\widetilde{C}_{p,q}(\varphi K,\varphi Q,\eta)
\leq c_2\aleph_1^{-p}\HH^{n-1}(\tilde{\varphi}\eta).
$$
Therefore applying the estimate of \eqref{phiestueta} on 
$\HH^{n-1}(\tilde{\varphi}\eta)$ yields 
\eqref{strictconvexityslnRinvariant0}.

{According} to Corollary~\ref{intgCqocorQ}, we have {that}
$$
\widetilde{C}_{p,q}(\varphi K,\varphi Q,\eta)=
\frac1n\int_{\partial' \varphi K} \mathbf{1}_\eta(\nu_{\varphi K}(x))
\langle \nu_{\varphi K}(x),x\rangle^{1-p}\|x\|_{\varphi Q}^{q-n}\,d\HH^{n-1}(x).
$$
{There exists an} $\aleph_3\in(0,1)$ depending on $\varphi$ and $Q$ such that
$$
\aleph_3\|x\|\leq\|x\|_{\varphi Q}\leq \aleph_3^{-1}\|x\|\mbox{ \ for $x\in \R^n$.}
$$
In particular, the last estimate,  Corollary~\ref{intgCqocorQ}
and \eqref{strictconvexityslnRinvariant0} imply 
$$
\tilde{c}_1\HH^{n-1}(\eta)\leq 
\widetilde{C}_{p,q}(\varphi K,\eta)
\leq \tilde{c}_2\HH^{n-1}(\eta)
$$
holds for any Borel {set} $\eta\subset S^{n-1}\backslash N(\varphi K,o)$, 
where $\tilde{c}_1=c^*_1\min\{\aleph_3^{q-n},\aleph_3^{n-q}\}$
and $\tilde{c}_2=1/\tilde{c}_1$,
completing the proof of Lemma~\ref{strictconvexityslnRinvariant}.
\end{proof}

We use Lemma~\ref{strictconvexityslnRinvariant} as follows.
For any convex body $K\in \mathcal{K}^n_o$ such that
$o\in{\rm bd K}$ and ${\rm int}K\neq \emptyset$, 
there exist $a\in S^{n-1}$,  $\beta\in(0,1)$ and $r_0>0$ such that
$$
\{x\in r_0B^n:\,\langle x,a\rangle \geq \beta\,\|x\|\}\subset K.
$$
Therefore, there exists $\varphi\in{\rm SL}(n,\R)$ such that $\varphi\, a=\lambda a$ for $\lambda>0$
and $\varphi(x)=\frac{\sqrt{3}}2/\sqrt{1-\beta^2}\,x$ for $x\in a^\bot$, thus
 for some
$r_1>0$, we have
$$
\left\{x\in r_1B^n:\,\langle x,a\rangle \geq \frac12\,\|x\|\right\}\subset \varphi K.
$$
In particular, for this $\varphi\in{\rm SL}(n,\R)$, we have
$\langle x,a\rangle \leq -\frac{\sqrt{3}}2\,\|x\|$ for any $x\in N(\varphi K,o)$, thus
\begin{equation}
\label{afterphi}
\langle x_1,x_2\rangle\geq \frac12\,\|x_1\|\,\|x_2\|\mbox{ \ for $x_1,x_2\in N(\varphi K,o)$.}
\end{equation}

\noindent{\bf Proof of Theorem~\ref{regularitystrictconvexity}. } 
If $o\in {\rm int}\,K$, then Theorem~\ref{regularityaway0allpq} (i) yields 
Theorem~\ref{regularitystrictconvexity}. Therefore, we assume that
$o\in{\partial}K$ and ${\rm int}K\neq \emptyset$ for $K\in \mathcal{K}^n_o$.
We may {also} assume by 
Lemma~\ref{strictconvexityslnRinvariant} and \eqref{afterphi} that
on the one hand, we have
\begin{equation}
\label{afterphi0}
\langle x_1,x_2\rangle\geq \frac12\,\|x_1\|\,\|x_2\|\mbox{ \ for $x_1,x_2\in N(K,o)$,}
\end{equation}
and on the other hand, using \eqref{Monge-Ampere0}
that there exist $c_2>c_1>0$ and {a} real Borel function $f$ on $S^{n-1}$ with
$c_1<f<c_2$ such that
\begin{equation}
\label{Monge-Ampere*}
\det(\nabla^2 h_K+h_K\,{\rm Id})=\mbox{$\frac1n$}\,h_K^{p-1}
\left(\|\nabla h_K\|^2+h_K^2\right)^{\frac{n-q}2}\cdot f.
\end{equation}

We {assume, on the contrary,} that $h_K$ is not differentiable at some point of $S^{n-1}$, or equivalently,
that $\partial K$ contains  an at least one dimensional face according to \eqref{duality_body_support2},
and seek a contradiction. It follows from Theorem~\ref{regularityaway0} (i) that any at least
one dimensional face of $K$ contains the origin $o$.

For $\Xi_K=\cup\{F(K,u):\,u\in S^{n-1}\mbox{ and }h_K(u)=0\}$, 
we define $\gamma>0$ and $w\in S^{n-1}$ such that
\begin{equation}
\label{gammawXiK}
\gamma=\max\{\|z\|:z\in \Xi_K\}>0 \mbox{ \ and \ }\gamma w\in \Xi_K.
\end{equation}
Let $e\in S^{n-1}$ be an exterior normal at $(\gamma/2)w\in \Xi_K$, therefore
\eqref{hderF} yields 
\begin{equation}
\label{hKderF}
\partial h_K(e)=F(K,e) \mbox{ and } o,\gamma w\in F(K,e).
\end{equation}
We may choose a closed convex cone $C_0$ with apex {$o$} such that
\begin{equation}
\label{C0def}
\begin{array}{rcl}
N(K,o)\backslash \{o\}&\subset& {\rm int} C_0\\
\|x\|&<&2\gamma \mbox{ for any $x\in {\partial}K$ with $\nu_K(x)\cap C_0\neq \emptyset$.}
\end{array}
\end{equation}
We choose $\delta>0$ such that
\begin{equation}
\label{C0delta}
 \{u\in S^{n-1}:\,h_K(u)\leq\delta\}\subset {\rm int}\,C_0.
\end{equation}

Let $v$ be the function of Lemma~\ref{MongeAmpereRn-lemma} associated to $e$ and $h_K$ on $e^\bot$, and hence
\eqref{Monge-Ampere*} yields {that}
\begin{eqnarray}
\label{MongeAmpereRn-low}
 \det( D^2 v(y) )&\geq &\frac{c_1}{\left(1+\|y\|^2\right)^{\frac{n+p}2}}\,v(y)^{p-1}
(\|Dv(y)\|^2+(\langle Dv(y),y\rangle-v(y))^2)^{\frac{n-q}2},\\
\label{MongeAmpereRn-upp}
\det( D^2 v(y) )&\leq & \frac{c_2}{\left(1+\|y\|^2\right)^{\frac{n+p}2}}\,
v(y)^{p-1}(\|Dv(y)\|^2+(\langle Dv(y),y\rangle-v(y))^2)^{\frac{n-q}2}
\end{eqnarray}
in the sense of measure.

It follows from \eqref{gammawXiK} and \eqref{hKderF} that
\begin{eqnarray*}
\partial v(o)&=&\partial h_K(e)|e^\bot=F(K,e);\\
\gamma&=&\max\{\|z\|:z\in \partial v(o)\}>0;\\
\gamma w&\in& \partial v(o),\mbox{ \ where $w\in S^{n-1}\cap e^\bot$}.
\end{eqnarray*}
Since $v$ is convex, we have {that}
\begin{equation}
\label{vestlow}
v(y)\geq \max\{0,\gamma \langle w,y\rangle\}\mbox{ \ for any $y\in e^\bot$,}
\end{equation}
and if $t>0$ tends to zero, then
\begin{equation}
\label{vestw}
v(tw)=\gamma\,t+o(t).
\end{equation}

For small $\varepsilon>0$, let us consider the first degree polynomial $l_\varepsilon$ on $e^\bot$ defined by
$$
l_\varepsilon(y)=(\gamma-\sqrt{\varepsilon})\langle w,y\rangle+\varepsilon,
$$
whose graph passes through $\varepsilon e$ and $\sqrt{\varepsilon}\,w+\gamma\sqrt{\varepsilon}\,e$.

We define
\begin{eqnarray*}
\Omega_\varepsilon&=&\{y\in e^\bot:\,v(y)< l_\varepsilon(y)\},\\
\widetilde{\Omega}_\varepsilon&=&\{y\in e^\bot:\,v(y)\leq l_\varepsilon(y)\}={\rm cl}\,\Omega_\varepsilon,
\end{eqnarray*}
where $\widetilde{\Omega}_\varepsilon$ is a closed convex set, and $\Omega_\varepsilon$ is its relative interior with respect to $e^\bot$. We have $o\in\Omega_\varepsilon$, and
since $v(y)\geq (\gamma-\sqrt{\varepsilon})\langle w,y\rangle$ for $y\in e^\bot$ by \eqref{vestlow}, we also have
\begin{equation}
\label{maxonOmega}
\max\{l_\varepsilon(y)-v(y):\,y\in \widetilde{\Omega}_\varepsilon\}=
\max\{l_\varepsilon(y)-v(y):\,y\in \Omega_\varepsilon\}=l_\varepsilon(o)-v(o)=\varepsilon.
\end{equation}
We observe that $l_\varepsilon(y)\geq \gamma\langle w,y\rangle$ for $y\in e^\bot$
if and only if $\langle w,y\rangle\leq \sqrt{\varepsilon}$.
It follows that assuming that $\varepsilon>0$ is small enough to satisfy $\sqrt{\varepsilon}<\gamma$, 
if $y\in \widetilde{\Omega}_\varepsilon$, then we have
\begin{equation}
\label{Omegainbetween}
\begin{array}{rcccl}
\frac{-2\varepsilon}{\gamma}&<&\langle w,y\rangle&\leq& \sqrt{\varepsilon},\\
&& v(y)&\leq &\gamma \sqrt{\varepsilon}.
\end{array}
\end{equation}
We observe that if $t\in(0,\sqrt{\varepsilon}/2)$, then $l_\varepsilon(tw)-\gamma t\geq \varepsilon/2$, and hence
 \eqref{vestw} yields the existence of $\theta_\varepsilon\in(0,\sqrt{\varepsilon}]$ such that
\begin{equation}
\label{thetaOmega}
\theta_\varepsilon\,w\in\Omega_\varepsilon\mbox{ \ and \ }
\lim_{\varepsilon\to 0^+}\varepsilon/\theta_\varepsilon=0.
\end{equation}

We consider the set
$$
\mathcal{U}=\left((e+e^\bot)\cap {\rm int}C_0\right)-e\subset e^\bot,
$$
that is open in the topology of $e^\bot$. If $v(y)\leq \delta$ for $y\in e^\bot$, then $h_K(u)\leq \delta$
for $u=(y+e)/\|y+e\|\in S^{n-1}$, therefore \eqref{C0delta} yields
$$
 \{y\in e^\bot:\,v(y)\leq\delta\}\subset\mathcal{U}.
$$
In particular,  we deduce from \eqref{hder}, \eqref{DhDv} and \eqref{C0def} that 
if $v(y)\leq\delta$ at some $y\in e^\bot$ where $v$ is differentiable, then
\begin{equation}
\label{Dvgamma}
\left\|Dv(y)+(\langle Dv(y),y\rangle-v(y))\cdot e\right\|\leq 2\gamma.
\end{equation} 

Let $L=e^\bot\cap w^\bot$, and let us consider the closed convex set
$$
Y=\{y\in e^\bot:\,v(y)=0\}=\left(N(K,o)\cap(e+e^\bot)\right)-e,
$$
and hence \eqref{afterphi0} and  \eqref{Omegainbetween} imply
$$
o\in Y\subset \sqrt{3} B^n\mbox{ \ and \ }Y=\cap_{\varepsilon>0}\Omega_\varepsilon.
$$ 
Therefore \eqref{Omegainbetween} and \eqref{Dvgamma} yield the existence of
some $\varepsilon_0>0$ such that if $\varepsilon\in (0,\varepsilon_0)$, then
\begin{eqnarray}
\label{Omegabounded}
\Omega_\varepsilon&\subset& 2B^n;\\
\label{OmegaboundedDv}
\left(\|Dv(y)\|^2+(\langle Dv(y),y\rangle-v(y))^2\right)^{\frac12}&\leq & 2\gamma\mbox{ \ provided  $v$ is differentiable at $y\in \Omega_\varepsilon$.}
\end{eqnarray} 

Using \eqref{Omegabounded} and \eqref{OmegaboundedDv}, we deduce that
\eqref{MongeAmpereRn-low} and \eqref{MongeAmpereRn-upp} yield the existence
of $\tilde{c}_1,\tilde{c}_2>0$ depending on $K$ and $e$ and independent of $\varepsilon$ such that
if $\varepsilon\in (0,\varepsilon_0)$, then
\begin{equation}
\label{MongeAmpereRn-lowupp}
\tilde{c}_1\,v(y)^{p-1}
\|Dv(y)\|^{n-q}\leq  \det( D^2 v(y) )\leq \tilde{c}_2 v(y)^{p-1} 
\end{equation}
hold on $\Omega_\varepsilon$ in the sense of measure.

We deduce from \eqref{thetaOmega} that we may also assume that if $\varepsilon\in (0,\varepsilon_0)$, then we have (compare \eqref{Omegainbetween})
\begin{equation}
\label{thetagamma}
\frac{\theta_\varepsilon}{16n}\geq \frac{2\varepsilon}{\gamma}.
\end{equation}
In the following, we assume $\varepsilon\in (0,\varepsilon_0)$.

As $\Omega_\varepsilon$ is bounded by \eqref{Omegabounded}, Loewner's (or John's) theorem
provides an $(n-1)$-dimensional ellipsoid $E_\varepsilon\subset e^\bot$ centered at the origin and a $z_\varepsilon\in\Omega_\varepsilon$ such that
\begin{equation}
\label{LoewnerOmega}
z_\varepsilon+\mbox{$\frac1n$}\,E_\varepsilon\subset \Omega_\varepsilon\subset z_\varepsilon+E_\varepsilon.
\end{equation}
Let $h_{E_\varepsilon}(w)=h_\varepsilon$, and let $a_\varepsilon\in E_\varepsilon$ satisfy 
$\langle a_\varepsilon,w\rangle=h_\varepsilon$. It follows from
\eqref{thetaOmega} and \eqref{LoewnerOmega}  that
$z_\varepsilon+E_\varepsilon$ contains a segment of length $\theta_\varepsilon$, therefore
\begin{equation}
\label{hepsilonlimit}
h_\varepsilon\geq \theta_\varepsilon/2\geq \frac{16n\varepsilon}{\gamma}\mbox{ and }\lim_{\varepsilon\to 0^+}\frac{\varepsilon}{h_\varepsilon}=0.
\end{equation}
On the one hand, $o\in\Omega_\varepsilon\subset z_\varepsilon+E_\varepsilon$ yields
$\langle z_\varepsilon,w\rangle\leq h_\varepsilon$, and on the other hand, we deduce from
\eqref{Omegainbetween}, \eqref{thetagamma} and \eqref{LoewnerOmega} that
$$
\langle z_\varepsilon,w\rangle-\frac{h_\varepsilon}n
=\left\langle z_\varepsilon-\frac{a_\varepsilon}n,w\right\rangle
\geq \frac{-2\varepsilon}{\gamma}\geq \frac{-h_\varepsilon}{8n}, 
$$
therefore
\begin{equation}
\label{hepsilonzepsilon}
\frac{7h_\varepsilon}{8n}\leq \langle z_\varepsilon,w\rangle\leq h_\varepsilon.
\end{equation}

If  $y\in \Omega_\varepsilon\subset z_\varepsilon+E_\varepsilon$, then
$\langle w,y\rangle\leq 2h_\varepsilon$ by \eqref{hepsilonzepsilon}, thus
the definition of $l_\varepsilon$ and \eqref{hepsilonlimit} imply
$$
v(y)\leq l_\varepsilon(y)\leq \gamma\langle w,y\rangle+\varepsilon\leq 
\left(2\gamma+\frac{\gamma}{16n}\right)h_\varepsilon.
$$
We write $\aleph_1,\aleph_2,\ldots$ to denote constants that depend on $n,p,q,\gamma,K,e$ and are independent of $\varepsilon$.
We deduce from \eqref{MongeAmpereRn-lowupp} that
 \begin{equation}
\label{muvupp}
\mu_v(\Omega_\varepsilon)\leq 
\int_{\Omega_\varepsilon}\tilde{c}_2\,v(y)^{p-1}
\,d\mathcal{H}^{n-1}(y) \leq \aleph_1h_\varepsilon^{p-1}\mathcal{H}^{n-1}(\Omega_\varepsilon)
 \leq \aleph_1h_\varepsilon^{p-1}\mathcal{H}^{n-1}(E_\varepsilon).
\end{equation}

In order to apply Lemma~\ref{MongeAmperepointinside}, we prove
 \begin{equation}
\label{muvlow}
\mu_v(z_\varepsilon+\mbox{$\frac1{2n}$}\,E_\varepsilon)\geq  \aleph_2h_\varepsilon^{p-1}
\mathcal{H}^{n-1}(E_\varepsilon).
\end{equation}
We note that
 \begin{equation}
\label{EepsLeps}
y-\mbox{$\frac12$}\,a_\varepsilon\in E_\varepsilon\mbox{ \ for $y\in \frac12\,E_\varepsilon$.}
\end{equation}
Let $Z_\varepsilon=z_\varepsilon+ \frac1{2n}\,E_\varepsilon$,
and hence
\begin{equation}
\label{Zarea}
\mathcal{H}^{n-1}(Z_\varepsilon)=\frac1{2^{n-1}n^{n-1}}\,\mathcal{H}^{n-1}(E_\varepsilon).
\end{equation}
It follows from \eqref{LoewnerOmega} and \eqref{EepsLeps} that if 
$y\in Z_\varepsilon$, then
$$
y-\mbox{$\frac{1}{2n}$}\,a_\varepsilon\in z_\varepsilon+\mbox{$\frac1n$}\,E_\varepsilon\subset 
\Omega_\varepsilon.
$$
In turn, we deduce from \eqref{Omegainbetween} and \eqref{hepsilonlimit} that if 
$y\in Z_\varepsilon$, then
$$
\langle y,w\rangle-\frac{h_\varepsilon}{2n}=
\left\langle y-\mbox{$\frac{1}{2n}$}\,a_\varepsilon,w\right\rangle
\geq \frac{-2\varepsilon}{\gamma}\geq \frac{-h_\varepsilon}{8n}, 
$$
therefore 
\begin{equation}
\label{heighty}
\langle y-\mbox{$\frac{3h_\varepsilon}{8n}$}\,w,w\rangle\geq 0.
\end{equation}

On the one hand, it follows from \eqref{vestlow} and \eqref{heighty} that if 
$y\in Z_\varepsilon$, then 
\begin{equation}
\label{yZv}
v(y)\geq \gamma \langle y,w\rangle\geq \gamma\cdot \frac{3h_\varepsilon}{8n}.
\end{equation}
On the other hand, it follows from  \eqref{heighty} and the convexity of $v$, and finally by \eqref{hepsilonlimit} that if $v$ is differentiable at
$y\in Z_\varepsilon$, then
\begin{eqnarray*}
\gamma\langle y,w\rangle-\langle Dv(y),\mbox{$\frac{3h_\varepsilon}{8n}$}\,w\rangle
&\leq& v(y)-\langle Dv(y),\mbox{$\frac{3h_\varepsilon}{8n}$}\cdot w\rangle\leq
v\left( y-\mbox{$\frac{3h_\varepsilon}{8n}$}\,w\right)\leq 
l_\varepsilon\left( y-\mbox{$\frac{3h_\varepsilon}{8n}$}\,w\right)\\
&\leq & \gamma\left\langle y-\mbox{$\frac{3h_\varepsilon}{8n}$}\,w,w\right\rangle+\varepsilon
\leq \gamma\langle y,w\rangle-\gamma\cdot \mbox{$\frac{3 h_\varepsilon}{8n}$}+
\gamma\cdot \mbox{$\frac{h_\varepsilon}{16n}$}=
\gamma\langle y,w\rangle-\gamma\cdot \mbox{$\frac{5 h_\varepsilon}{16n}$}.
\end{eqnarray*}
In particular, if $v$ is differentiable at
$y\in Z_\varepsilon$, then $\langle Dv(y), w\rangle\geq \frac56\,\gamma$, which, in turn, yields that
\begin{equation}
\label{yZDv}
\|Dv(y)\|\geq  \mbox{$\frac{5}{6}$}\,\gamma.
\end{equation}
Since \eqref{MongeAmpereRn-lowupp} implies
$$
\mu_v(z_\varepsilon+\mbox{$\frac1{2n}$}\,E_\varepsilon)\geq 
\int_{Z_\varepsilon} c_1v(y)^{p-1}\|Dv(y)\|^{n-q}\,d\mathcal{H}^{n-1}(y),
$$
we conclude \eqref{muvlow} from \eqref{Zarea},  \eqref{yZv} and \eqref{yZDv}.

We deduce from combining \eqref{muvupp} and \eqref{muvlow} that
\begin{equation}
\label{muvlowupp}
\mu_v(z_\varepsilon+\mbox{$\frac1{2n}$}\,E_\varepsilon)\geq \aleph_3
\mu_v(\Omega_\varepsilon)
\end{equation}
for $\aleph_3=\aleph_1/\aleph_2$.

We define $\tilde{v}=v-l_\varepsilon$, which also satisfies \eqref{muvlowupp}. In particular,
$\tilde{v}$ satisfies the conditions of Lemma~\ref{MongeAmperepointinside} with $\Omega=\Omega_\varepsilon$, $E=E_\varepsilon$, $t=\frac1{2n}$,
$b=\aleph_3$, $z=o$ and $z_0=z_\varepsilon$. In addition, we deduce from \eqref{Omegainbetween} that we can use
$$
s=\frac{2\varepsilon}{\gamma h_\varepsilon}
$$
in Lemma~\ref{MongeAmperepointinside}. We conclude from Lemma~\ref{MongeAmperepointinside} that
\begin{equation}
\label{vovzeps}
\frac{|\tilde{v}(o)|}{|\tilde{v}(z_\varepsilon)|}\leq \aleph_4s^{1/n}.
\end{equation}
However, $\tilde{v}(o)=-\varepsilon$, and \eqref{supval} yield that
$$
|\tilde{v}(z_\varepsilon)|\geq \frac{t}{t+1} \cdot \sup_{\Omega_\varepsilon} |\tilde{v}|\geq \frac{\varepsilon}{4n}.
$$
We deduce from \eqref{vovzeps} that if $\varepsilon\in(0,\varepsilon_0)$, then
\begin{equation}
\label{vovzeps0}
2n\leq \aleph_4\left(\frac{2\varepsilon}{\gamma h_\varepsilon}\right)^{1/n}.
\end{equation}
Here $\lim_{\varepsilon\to 0^+}\frac{\varepsilon}{h_\varepsilon}=0$ according to \eqref{hepsilonlimit},
which fact clearly contradicts \eqref{vovzeps0}. Finally, this contradiction proves
Theorem~\ref{regularitystrictconvexity}. 
\hfill $\Box$\\

\noindent{\bf Remark. } The reason that our method of proof does not work if $q>n$ is that in that case
$\|Dv(y)\|^{n-q}$ can be arbitrarily large if $v(y)>0$ and is very small.

\section{Acknowledgements}
 The authors are indebted to L\'aszlo Sz\'ekelyhidi Jr. for helpful discussions. Both authors are supported by grant
NKFIH K 116451. First named author is also supported by grants NKFIH ANN 121649 and NKFIH K 109789.

\end{document}